\documentclass[reqno]{amsart}

\pdfoutput=1

\usepackage[alphabetic]{amsrefs}
\usepackage{amsmath}
\usepackage{amsfonts}
\usepackage{amssymb}
\usepackage{enumerate}
\usepackage{amstext}
\usepackage{amsbsy}
\usepackage{amsopn}
\usepackage{bbm,amsthm}
\usepackage{amscd}
\usepackage[pdftex]{color}
\usepackage{amsxtra}
\usepackage{upref}
\usepackage{graphicx,color}
\usepackage{epstopdf}
\usepackage{framed}
\usepackage{MnSymbol,wasysym}
\usepackage{keystroke,marvosym}
\usepackage{hyperref}
\usepackage{scalerel}

\DeclareFontFamily{OML}{rsfs}{\skewchar\font'177}
\DeclareFontShape{OML}{rsfs}{m}{n}{ <5> <6> rsfs5 <7> <8> <9> rsfs7
  <10> <10.95> <12> <14.4> <17.28> <20.74> <24.88> rsfs10 }{}
\DeclareMathAlphabet{\mathfs}{OML}{rsfs}{m}{n}

\newtheorem{theorem}{Theorem}
\newtheorem{lemma}[theorem]{Lemma}
\newtheorem{proposition}[theorem]{Proposition}
\newtheorem{corollary}[theorem]{Corollary}

\theoremstyle{definition}

\theoremstyle{remark}
\newtheorem{remark}[theorem]{\bf Remark}

\numberwithin{equation}{section}
\numberwithin{theorem}{section}

\newcommand{\intav}[1]{\mathchoice {\mathop{\vrule width 6pt height 3 pt depth  -2.5pt
\kern -8pt \intop}\nolimits_{\kern -6pt#1}} {\mathop{\vrule width
5pt height 3  pt depth -2.6pt \kern -6pt \intop}\nolimits_{#1}}
{\mathop{\vrule width 5pt height 3 pt depth -2.6pt \kern -6pt
\intop}\nolimits_{#1}} {\mathop{\vrule width 5pt height 3 pt depth
-2.6pt \kern -6pt \intop}\nolimits_{#1}}}

\newcommand{\intavl}[1]{\mathchoice {\mathop{\vrule width 6pt height 3 pt depth  -2.5pt
\kern -8pt \intop}\limits_{\kern -6pt#1}} {\mathop{\vrule width 5pt
height 3  pt depth -2.6pt \kern -6pt \intop}\nolimits_{#1}}
{\mathop{\vrule width 5pt height 3 pt depth -2.6pt \kern -6pt
\intop}\nolimits_{#1}} {\mathop{\vrule width 5pt height 3 pt depth
-2.6pt \kern -6pt \intop}\nolimits_{#1}}}

\newcommand{\un}{\underline}
\newcommand{\ve}{\varepsilon}
\newcommand{\vf}{\varphi}

\newcommand{\R}{\mathbb{R}}
\newcommand{\N}{\mathbb{N}}
\newcommand{\Q}{\mathbb{Q}}
\newcommand{\Z}{\mathbb{Z}}
\newcommand{\T}{\mathbb{T}}

\renewcommand{\exp}[1]{{\rm exp}_{#1}}

\newcommand{\inj}{{\rm inj}}
\newcommand{\Hol}[1]{{\rm Hol}_{#1}}

\newcommand{\vertiii}[1]{{\left\vert\kern-0.2ex\left\vert\kern-0.2ex\left\vert #1 
    \right\vert\kern-0.2ex\right\vert\kern-0.2ex\right\vert}}

\begin{document}

\title[Symbolic dynamics for nonuniformly hyperbolic systems]{Symbolic dynamics for nonuniformly\\ hyperbolic systems}
\author{Yuri Lima}
\thanks{The author is supported by CNPq and Instituto Serrapilheira, grant ``Jangada Din\^amica:
Impulsionando Sistemas Din\^amicos na Regi\~ao Nordeste''.}

\address{Yuri Lima, Departamento de Matem\'atica, Universidade Federal do Cear\'a (UFC), Campus do Pici,
Bloco 914, CEP 60440-900. Fortaleza -- CE, Brasil}
\email{yurilima@gmail.com}

\date{\today}


\maketitle

\begin{abstract}
This survey describes the recent advances in the construction of Markov partitions
for nonuniformly hyperbolic systems. One important feature of this development comes
from a finer theory of nonuniformly hyperbolic systems, which we also describe. The 
Markov partition defines a symbolic extension that is finite-to-one and onto a non-uniformly
hyperbolic locus, and this provides dynamical and statistical consequences
such as estimates on the number of closed orbits and properties of equilibrium measures.  
The class of systems includes diffeomorphisms, flows, and maps with singularities.
\end{abstract}

\tableofcontents

\section*{Introduction}
Markov partitions are a powerful tool in the modern theories of dynamical systems and ergodic theory.
They were introduced to these fields at the end of the sixties, see the foundational works of Adler \& Weiss and
Sina{\u\i} \cite{Adler-Weiss-1967,Adler-Weiss-Similarity-Toral-Automorphisms,Sinai-MP-U-diffeomorphisms,Sinai-Construction-of-MP} and references
therein, and have played a crucial role ever since.
Roughly speaking, a Markov partition is a partitioning of the phase space of a system into pieces, that allows to
represent trajectories by paths on a graph. The dynamics of paths on 
a graph is much simpler to understand, and many of its statistical properties can therefore be pushed
to the original dynamical system. This approach was extensively developed in the late sixties and 
early seventies to uniformly hyperbolic systems, and its consequences include many breakthroughs
in smooth ergodic theory.
The method developed by Bowen \cite{Bowen-LNM} will be of particular importance to us:
locally representing the dynamics as a small perturbation of a hyperbolic matrix, he used the theory of pseudo-orbits
used by Anosov \cite{Anosov-Certain} and by
himself \cite{Bowen-Topological-entropy,Bowen-Periodic-points,Bowen-Periodic-orbits} to
obtain a Markov cover and then refine it to a Markov partition. Due to uniform hyperbolicity,
the Markov partitions are finite.

\medskip
Following Bowen's philosophy, Katok showed that a hyperbolic ergodic measure
that is invariant under a $C^{1+\beta}$ diffeomorphism has horseshoes approximating its entropy.
A measure is hyperbolic if its Lyapunov exponents are non-zero, and this introduces the concept
of {\em nonuniform hyperbolicity}: the hyperbolicity is not necessarily observed at every iteration but only
on the average. In the late seventies, Pesin developed a global theory to treat $C^{1+\beta}$ nonuniformly hyperbolic systems \cite{Pesin-Izvestia-1976,Pesin-Characteristic-1977,Pesin-Geodesic-1977},
nowadays known as {\em Pesin theory}. See the book \cite{Barreira-Pesin-2007}.
Pesin's idea was to construct local charts, nowadays called {\em Pesin charts},
to represent the dynamics of a nonuniformly hyperbolic diffeomorphism
again as a small perturbation of a hyperbolic matrix. The difference from the uniformly hyperbolic
situation is that the domain of the Pesin chart is no longer uniform in size and depends on the quality of hyperbolicity
at the point. In \cite{Katok-IHES}, Katok combined Pesin theory with a fine theory of pseudo-orbits and, to avoid the possible degeneracy of
Pesin charts, restricted the analysis to {\em Pesin blocks}, which are non-invariant subsets of the phase space
where nonuniform hyperbolicity is essentially uniform. For details, see the supplementary chapter by Katok \&
Mendoza \cite{Katok-Hasselblatt-Book}.
Since a horseshoe naturally carries a Markov partition, Katok's result can be seen as the construction
of finite Markov partitions that approximate the topological entropy. 
The applications using the now called {\em Katok horseshoes}
are countless. Nevertheless, this approach is not genuinely nonuniformly hyperbolic,
since it does not treat at once regions where the degeneracy of Pesin charts occurs. In other words,
a single Pesin block  does not encompass the whole dynamics (for instance, it usually does not have
full topological entropy).

\medskip
This difficulty stood unsolved for many years, until Sarig recently bypassed it,
constructing countable Markov partitions with full topological entropy for $C^{1+\beta}$
surface diffeomorphisms \cite{Sarig-JAMS}.
His methods are more suitable for adaptations and generalizations, and
are now being further refined to settings in which the previous theory was not able
to reach, such as billiard maps. Here are some of the developments:
\begin{enumerate}[$\circ$]
\item Lima and Sarig for three dimensional flows without fixed points \cite{Lima-Sarig}.
\item Lima and Matheus for surface maps with discontinuities \cite{Lima-Matheus}.
\item Ben Ovadia for diffeomorphisms in any dimension \cite{Ben-Ovadia-2019}.
\end{enumerate}
Now, the Markov partitions are countable. This is unavoidable to treat the regions
where the Pesin charts degenerate. 
Not only these latter results are stronger than the previous ones in the literature, 
they also cover much broader classes of examples such as 
geodesic flows on surfaces with nonpositive curvature and Bunimovich stadia.
Using them, many dynamical and statistical properties were established:
counting on the number of closed orbits
\cite{Sarig-JAMS,Lima-Sarig,Lima-Matheus,Baladi-Demers,Ben-Ovadia-2019,Buzzi-2019},
counting on the number of measures of maximal entropy
\cite{Sarig-JAMS,Lima-Sarig,Ben-Ovadia-2019,Buzzi-Crovisier-Sarig},
ergodic properties of equilibrium measures \cite{Sarig-Bernoulli-JMD,Ledrappier-Lima-Sarig},
the almost Borel structure of surface diffeomorphisms \cite{Boyle-Buzzi}, and the generic
simplicity of Lyapunov spectrum of nonuniformly hyperbolic diffeomorphisms \cite{BLPV}.

\medskip
Generally speaking, countable Markov partitions are indeed necessary
to code nonuniformly hyperbolic systems: while the set of topological entropies
associated to finite Markov partitions is countable, the set of topological entropies
of nonuniformly hyperbolic systems is $[0,\infty)$. This occurs e.g. among $C^\infty$
diffeomorphisms in surfaces, where the topological entropy is continuous
\cite[Thm. 6]{Newhouse-Entropy}. For instance, consider the two-dimensional disc:
the identity map has zero entropy, and Smale's horseshoe has topological entropy
equal to $\log 2$ (see Section \ref{Subsec-UH-definitions} for details on this latter example).
Since these maps are homotopic, the set of values for the topological entropy contains 
$[0,\log 2]$ and thus, taking powers, it is equal to $[0,\infty)$.
The same occurs for $C^\infty$ diffeomorphisms in the two-dimensional torus:
in the one-parameter family of standard maps $f_k(x,y)=(-y+2x+k\sin(2\pi x),x)$,
the topological entropy reaches arbitrarily large values \cite{Duarte-standard-maps}.

\medskip
As already mentioned, Markov partitions provide many dynamical and statistical
consequences because the dynamics of paths on 
a graph is simple to understand. In general, any partition generates
a symbolic representation of the system, given by the shift map acting on a subset of
the symbolic space of paths on a graph. For Markov partitions, such symbolic representation
is defined not only on a subset but on the whole space of paths on the graph. This is already a big advantage, 
but for effectiveness of applications it is important to understand the {\em coding map},
that relates real trajectories to paths on the graph. If, for instance, the coding map is finite-to-one
(i.e. every point has finitely many pre-images) then measures on the original system are related to
measures on the symbolic space, and the relation preserves entropy (by the Abramov-Rokhlin formula).
This happens for uniformly hyperbolic systems almost automatically, but constitutes a
major difficulty for nonuniformly hyperbolic ones. Indeed, all previous attempts before Sarig
failed exactly at this point. Sarig did not prove that the coding map is finite-to-one, but that
it is {\em morally} finite-to-one: after passing to recurrent subsets (defined by some recurrence assumptions),
the coding map is finite-to-one.
This was the motivation to perform Pesin theory in a much finer way,
which has a central importance in the recent constructions of Markov partitions.
Having this in mind, this survey has two main goals:
\begin{enumerate}[$\circ$]
\item Discuss the theory of nonuniformly hyperbolic systems.
\item Use this theory to construct countable Markov partitions that generate finite-to-one
coding maps.
\end{enumerate}
Since the main reason to construct Markov partitions and finite-to-one codings is to understand
dynamical and statistical properties of smooth dynamical system, we also provide applications in this
context. 

\medskip
Two words of caution. Firstly, we do not provide a historical account on the
developments of Markov partitions.
Secondly, we do not discuss symbolic dynamics
in great extent, but only finite-to-one codings for systems with uniform and nonuniform hyperbolicity.
Away from these contexts, there are various tools in symbolic dynamics that are
important on their own and provide far reaching conclusions, such as Milnor-Thurston's theory
of kneading sequences \cite{Milnor-Thurston}, Hofbauer towers
\cite{Takahashi-isomorphisms,Hofbauer-beta-shifts,Hofbauer-intrinsic-I,Hofbauer-intrinsic-II},
symbolic extensions \cite{Boyle-Downarowicz-Inventiones,Burguet-Inventiones,Downarowicz-Book},
Yoccoz puzzles \cite{Yoccoz-Puzzles}, Young towers \cite{Young-towers}, and more.

%
%
%

\medskip
We divide the survey into three parts. In Part 1, we discuss the theory of invariant manifolds for
uniformly and nonuniformly hyperbolic systems, including the construction of local charts and
graph transforms. For simplicity of exposition, most of the arguments will be discussed in dimension two,
both for diffeomorphisms and maps with discontinuities, but we also sketch how to make the constructions
in higher dimension. In Part 2, we extend this theory to pseudo-orbits, and explain how to use them
to construct Markov partitions and finite-to-one coding maps. 
In Part 3, we provide applications.

\section*{Acknowledgements}

These notes grew up from a series of minicourses the author has given in the past three years, in various institutions:
ICERM (Providence, USA), Jagiellonian University (Krakow, Poland), ICTP (Trieste, Italy),
Brigham Young University (Provo, USA), PUC-Chile (Santiago, Chile),
University of Warwick (Coventry, UK), and Tsinghua University (Beijing, China).
The author would like to thank the institutions and local organizers for the invitations, and
Mark Demers, Gerhard Knieper for valuable suggestions. 
The first version of this document was prepared while the author was supported by CNPq and
by Instituto Serrapilheira, grant Serra-1709-20498.

\part{Charts, graph transforms, and invariant manifolds}\label{Part-1}

We introduce tools that allow to pass from the infinitesimal information given by the
assumption on the derivative of the system to a representation of its local dynamics.
The main goal is to introduce three tools:
\begin{enumerate}[$\circ$]
\item Local charts, which locally represent the dynamics as a small perturbation of a hyperbolic matrix.
\item Graph transforms, which explore the hyperbolic feature of the local representation to
identify points that remain close to trajectories.
\item Invariant manifolds, which provide dynamical coordinates and 
allow to separate the future and past behavior of
the system.
\end{enumerate}
For methodological reasons, the discussion is divided into sections, each of them treating
a different class of systems. In Section \ref{Section-UH-Systems} we deal with
uniformly hyperbolic diffeomorphisms. In Section \ref{Section-NUH-systems} we consider
nonuniformly hyperbolic diffeomorphisms. In the last two sections, 
we discuss nonuniformly hyperbolic surface maps with discontinuities: in Section
\ref{Section-NUH-dis1} we assume bounded derivative (e.g. Poincar\'e return maps
of flows without fixed points), and in Section \ref{Section-NUH-dis2} we allow the derivative to
grow polynomially fast to infinity (e.g. billiard maps). The discussion in each new section
emphasizes the new input that is necessary to make the construction work, so we recommend the
reader to follow the text as presented here.

\section{Uniformly hyperbolic systems}\label{Section-UH-Systems}

Uniformly hyperbolic systems are at the heart of the great developments that tailored the
beginning of the modern theories of dynamical systems and ergodic theory, and constitute
one of the nicest situations in which a system shows chaos in almost any sense of the word:
exponential divergence of the trajectories, denseness of periodic orbits, among others. 
The study of uniformly hyperbolic systems now has a long history, that began already
in the 19th century with the study of geodesic flows on surfaces of constant negative
curvature by Hadamard \cite{Hadamard-1898}.
This topic was extensively developed between 1920 and 1940, among which we mention the work
of Morse \cite{Morse-PNAS}, Hedlund \cite{Hedlund-1939}, and Hopf \cite{Hopf-1939,Hopf-1940}.
About 1940, it became clear that geodesic flows were a particular case of the real
setup of interest, and Anosov realized that the theory goes through under a more general condition, that
he called the (U)-{\em condition}. In his own words, a system satisfies the (U)-condition
if it has ``exponential dichotomy of solutions'' \cite[pp. 22]{Anosov-Geodesic-Flows}.
Anosov made fundamental contributions to the study of (U)-systems, including their
ergodicity \cite[Thm. 4]{Anosov-Geodesic-Flows}. Nowadays, (U)-systems are called
{\em Anosov systems}, and the assumption on exponential dichotomy of solutions is called
{\em hyperbolicity}.

\medskip
While the Russian school focused on the probabilistic aspects of dynamical systems,
the American school led by Smale focused on the topological aspects. Smale discovered the {\em horseshoe},
which is the first example of a system shown to have infinitely many periodic points and yet being structurally stable.
The history of the discovery is explained in \cite{Smale-horseshoe-1998}, where Smale claims
that ``the horseshoe is a natural consequence of a geometrical way of
looking at the equations of Cartwright-Littlewood and Levinson''.
A horseshoe has similar properties to Anosov systems, because the recurrent (but not all) trajectories
are hyperbolic. For the purpose of dynamics,
this is satisfactory because a non-recurrent trajectory is uninteresting for dynamical purposes.
Having this in mind, Smale introduced the notion of {\em Axiom A systems}, where hyperbolicity
is required to hold only on the non-wandering set.
For transitive Anosov systems, the notions of Anosov and Smale coincide, but there are
Axiom A systems that are not Anosov. What we call {\em uniformly hyperbolic}
are Anosov and Axiom A systems. Nowadays there are great textbooks describing such
systems, see e.g. \cite{Brin-Stuck-Book,Shub-Book,Katok-Hasselblatt-Book}.

\medskip
The main result of this section is the existence of local invariant manifolds.
It holds for $C^1$ uniformly hyperbolic systems,
see e.g. \cite{Shub-Book},  but to keep an analogy with the nonuniformly 
hyperbolic context to be discussed in Section \ref{Section-NUH-systems},
we will assume for most of the time that the system is $C^{1+\beta}$,
see definition in Section \ref{Subsec-UH-preliminaries}.

\subsection{Definitions and examples}\label{Subsec-UH-definitions}

Let $M$ be a closed (compact without boundary) connected
smooth Riemannian manifold, and let $f:M\to M$ be a $C^1$ diffeomorphism.

\medskip
\noindent
{\sc Anosov diffeomorphism:} We call $f$ an {\em Anosov diffeomorphism} if there is
a continuous splitting $TM=E^s\oplus E^u$ and constants $C>0$, $\kappa<1$ s.t.:
\begin{enumerate}[(1)]
\item {\sc Invariance}: $df (E^{s/u}_x)=E^{s/u}_{f(x)}$ for all $x\in M$.
\item {\sc Contraction:}
\begin{enumerate}[$\circ$]
\item Vectors in $E^s$ contract in the future: $\|df^n v\|\leq C\kappa^n\|v\|$ for all $v\in E^s$, $n\geq 0$.
\item Vectors in $E^u$ contract in the past: $\|df^{-n} v\|\leq C\kappa^n\|v\|$ for all $v\in E^u$, $n\geq 0$.
\end{enumerate}
\end{enumerate}

\medskip
A closed $f$--invariant set $\Lambda$ satisfying the above properties is called {\em uniformly hyperbolic} or simply
{\em hyperbolic}, hence a diffeomorphism is Anosov if the whole phase space $M$ is hyperbolic.
The continuity condition of the splitting in the definition
indeed follows from the other assumptions, see e.g.
\cite[Proposition 5.2.1]{Brin-Stuck-Book}. As a matter of fact, the splitting is H\"older continuous, as proved
by Anosov \cite{Anosov-tangential}, see also the appendix of \cite{Ballmann-lecture-notes} for a simpler proof.
Condition (2) is the exponential dichotomy of solutions
mentioned by Anosov. Usually, the above assumptions are rather restrictive because they require
the properties on all of $M$, and sometimes parts of $M$ are not dynamically relevant. The set
where interesting dynamics can occur is called the non-wandering set.

\medskip
\noindent
{\sc Non-wandering set $\Omega(f)$:} The {\em non-wandering set} of $f$, denoted by $\Omega(f)$,
is the set of all $x\in M$ s.t. for every neighborhood $U\ni x$ there
exists $n\neq 0$ s.t. $f^n(U)\cap U\neq\emptyset$.

\medskip
In other words, a point is non-wandering if, no matter how small we choose a neighborhood, 
it does self-intersect in the future or in the past. In particular, every periodic point is non-wandering.
Let ${\rm Per}(f)$ denote the set of periodic points of $f$.

\medskip
\noindent
{\sc Axiom A diffeomorphism:} We call $f$ an {\em Axiom A diffeomorphism} if:
\begin{enumerate}[(1)]
\item {\sc Denseness of periodic orbits:} $\overline{{\rm Per}(f)}=\Omega(f)$.
\item {\sc Hyperbolicity:} $\Omega(f)$ is hyperbolic, i.e. there exists a continuous splitting $T_{\Omega(f)}M=E^s\oplus E^u$
and constants $C>0$, $\kappa<1$ s.t.:
\begin{enumerate}[$\circ$]
\item $df (E^{s/u}_x)=E^{s/u}_{f(x)}$ for all $x\in \Omega(f)$.
\item $\|df^n v\|\leq C\kappa^n\|v\|$ for all $v\in E^s$, $n\geq 0$.
\item $\|df^{-n} v\|\leq C\kappa^n\|v\|$ for all $v\in E^u$, $n\geq 0$.
\end{enumerate}
\end{enumerate}

\medskip
Every Anosov diffeomorphism is Axiom A, but the converse is false. 
Now let $\vf:M\to M$ be a flow generated by a vector field $X$ of class $C^1$.
The definitions of uniformly hyperbolic flows are similar to the ones above, having in mind
that in the flow direction there is no contraction nor expansion. Below,
$\langle X\rangle$ represents the subbundle generated by $X$, whose vector space
at $x$ is the line generated by $X_x$.

\medskip
\noindent
{\sc Anosov flow:} We call $\vf$ an {\em Anosov flow} if $X\neq 0$ everywhere and if there is
a continuous splitting $TM=E^s\oplus \langle X\rangle\oplus E^u$ and constants $C>0$, $\kappa<1$ s.t.:
\begin{enumerate}[(1)]
\item {\sc Invariance}: $d\vf^t (E^{s/u}_x)=E^{s/u}_{\vf^t(x)}$ for all $x\in M$, $t\in\R$.
\item {\sc Contraction:}
\begin{enumerate}[$\circ$]
\item Vectors in $E^s$ contract in the future: $\|d\vf^t v\|\leq C\kappa^t\|v\|$ for all $v\in E^s$, $t\geq 0$.
\item Vectors in $E^u$ contract in the past: $\|d\vf^{-t} v\|\leq C\kappa^t\|v\|$ for all $v\in E^u$, $t\geq 0$.
\end{enumerate}
\end{enumerate}

\medskip
Similarly, a closed $f$--invariant set $\Lambda$ satisfying the above properties is
called {\em uniformly hyperbolic} or simply {\em hyperbolic}.

\medskip
\noindent
{\sc Non-wandering set $\Omega(\vf)$:} The {\em non-wandering set} of $\vf$, denoted by $\Omega(\vf)$,
is the set of all $x\in M$ s.t. for every neighborhood $U\ni x$ and for every $t>0$ there
exists $T\in\R$ with $|T|> t$ s.t. $\vf^T(U)\cap U\neq\emptyset$.

\medskip
The above definition is natural, since $\vf^t(U)\cap U\neq \emptyset$
for any $t$ sufficiently small. Let ${\rm Per}(\vf)$ denote the set of periodic points of $\vf$.

\medskip
\noindent
{\sc Axiom A flow:} We call $\vf$ an {\em Axiom A flow} if $X\neq 0$ on $\Omega(\vf)$ and:

\begin{enumerate}[(1)]
\item {\sc Denseness of periodic orbits:} $\overline{{\rm Per}(\vf)}=\Omega(\vf)$.
\item {\sc Hyperbolicity:} $\Omega(\vf)$ is hyperbolic, i.e. there exists a continuous splitting
$T_{\Omega(f)}M=E^s\oplus \langle X\rangle\oplus E^u$ and constants $C>0$, $\kappa<1$ s.t.:
\begin{enumerate}[$\circ$]
\item $d\vf^t (E^{s/u}_x)=E^{s/u}_{\vf(x)}$ for all $x\in \Omega(f)$, $t\in\R$.
\item $\|d\vf^t v\|\leq C\kappa^t\|v\|$ for all $v\in E^s$, $t\geq 0$.
\item $\|d\vf^{-t} v\|\leq C\kappa^t\|v\|$ for all $v\in E^u$, $t\geq 0$.
\end{enumerate}
\end{enumerate}

\medskip
We call a system {\em uniformly hyperbolic} if it is either Anosov or Axiom A.
Here are three classical examples.

\medskip
\noindent
1. Every hyperbolic matrix\footnote{A matrix is hyperbolic if none of its eigenvalues
lie on the unit circle.} $A\in {\rm SL}(n,\mathbb R)$ induces an Anosov diffeomorphism
$f=f_A:\mathbb T^n\to \mathbb T^n$ on the $n$--dimensional torus $\mathbb T^n=\R^n/\Z^n$.
For $A=\left[\begin{array}{cc}2&1\\ 1&1\end{array}\right]$, the Anosov diffeomorphism $f$ is
known as {\em Arnold's cat map} or simply {\em cat map}. Although the dynamics of $A$ is simple,
the dynamics of $f$ as seen on the canonical fundamental
domain $[0,1]^2$ of $\mathbb T^2$ is rather complicated, see Figure \ref{figure-cat-map}.
\begin{figure}[hbt!]
\centering
\def\svgwidth{8cm}
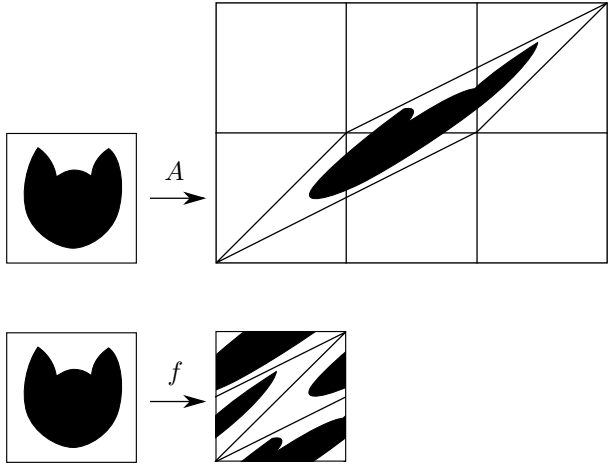
\caption{Arnold's cat map.}\label{figure-cat-map}
\end{figure}
See also \cite[\S 1.7]{Brin-Stuck-Book}.

\medskip
\noindent
2. Smale's horseshoe, generated by the geometrical configuration in Figure \ref{figure-horseshoe} below.
See \cite{Shub-Book,Brin-Stuck-Book}.
\begin{figure}[hbt!]
\centering
\def\svgwidth{7cm}
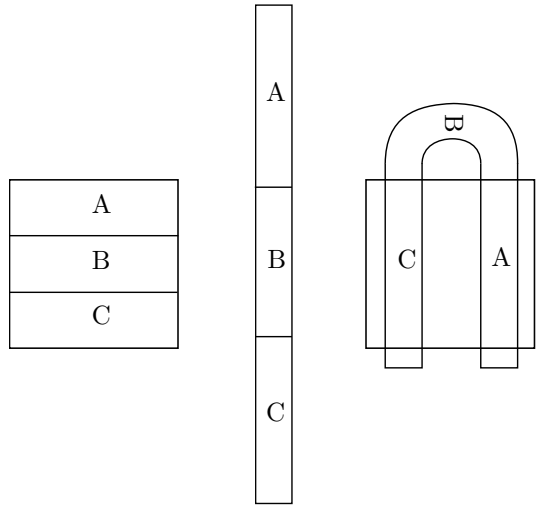
\caption{The geometrical mechanism in the creation of a horseshoe.}\label{figure-horseshoe}
\end{figure}

\medskip
\noindent
3. The geodesic flow on a closed manifold with negative sectional curvature is Anosov,
see Figure \ref{figure-hyperbolic-surface}.
Its hyperbolicity is more complicated to describe, since it is defined on the (unit) tangent bundle
of the manifold and its derivative in the tangent bundle of this (unit) tangent bundle.
We refer the reader to \cite[Chapter 1]{Barreira-Pesin-new-book}
for a discussion on the two-dimensional case with constant curvature,
to \cite{Eberlein} for a more general discussion,
and to \cite[Section 1.3]{Knieper-Handbook-Chapter} for a proof of hyperbolicity.
\begin{figure}[hbt!]
\centering
\def\svgwidth{10cm}
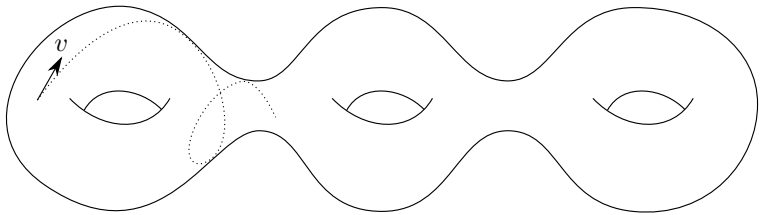
\caption{A surface with negative curvature, whose geodesic flow
is Anosov. Attention: the figure is somewhat misleading because
it induces to think that the curvature in the extreme left and right hand sides is positive.}\label{figure-hyperbolic-surface}
\end{figure}

\subsection{Preliminaries on the geometry of $M$}\label{Subsec-UH-preliminaries}

It is easy to define H\"older continuity for maps on euclidean spaces.
For instance, $f:U\subset \R^n\to \R^m$ is $\beta$--H\"older if there is $\mathfrak K>0$ s.t.
$\|f(x)-f(x)\|\leq \mathfrak K\|x-y\|^\beta$ for all $x,y\in U$. Similarly, $f$ is $C^{1+\beta}$ if
it is $C^1$ and there is $\mathfrak K>0$ s.t.$\|df^{\pm 1}_x-df^{\pm 1}_y\|\leq \mathfrak K\|x-y\|^\beta$
for all $x,y\in U$.
For general manifolds, this is slightly more complicated because the derivatives
are defined in different tangent spaces and so we need to compare the geometry of
nearby tangent spaces. For that we use the local charts provided by the exponential maps, which
are intrinsic of the geometry of $M$.
For the inexperienced reader, we suggest to make the calculations in the euclidean situation,
where all exponential maps are identity.

\medskip
We are assuming $M$ is a closed connected
smooth Riemannian manifold. We denote open balls in $M$ by $B(x,r)$.
Given $r>0$, let $B_x[r]\subset T_xM$ be the open ball with center 0 and radius $r$.
%
For each $x\in M$, let $\exp{x}: T_xM\to M$ be the {\em exponential map} at $x$, i.e.
$\exp{x}(v)=\gamma(1)$ where $\gamma$ is the unique geodesic s.t. $\gamma'(0)=v$.
Given $x\in M$, let $\inj(x)$ be the {\em radius of injectivity} at $x$,
i.e. $\inj(x)$ is the largest $r>0$ s.t. the restriction of $\exp{x}$ to $B_x[r]$ is a diffeomorphism onto 
its image. Choose $\mathfrak r_0>0$ s.t. for $D_x:=B(x,2\mathfrak r_0)$ the following holds:
\begin{enumerate}[$\circ$]
\item $\exp{x}:B_x[2\mathfrak r_0]\to M$ is a 2--bi-Lipschitz diffeomorphism onto its image. 
\item If $y\in D_x$ then $\inj(y)\geq 2\mathfrak r_0$ and $\exp{y}^{-1}:D_x\to T_yM$
is a 2--bi-Lipschitz diffeomorphism onto its image.
\end{enumerate}
Such $\mathfrak r_0>0$ exists because $d(\exp{x})_0={\rm Id}$ and $M$ is compact. 

\medskip
For $x,x'\in\ M$, let $\mathfs L _{x,x'}:=\{A:T_xM\to T_{x'}M:A\text{ is linear}\}$
and $\mathfs L _x:=\mathfs L_{x,x}$. 
If $d(x,y)<\inj(x)$, then there is a unique radial geodesic $\gamma$
joining $x$ to $y$, and the parallel transport $P_{x,y}$ along this geodesic is in
$\mathfs L_{x,y}$.
Let $x,x'\in M$ and $y,z\in M$ s.t. $d(x,y)<\inj(x)$ and $d(x',z)<\inj(x')$. Given $A\in \mathfs L_{y,z}$,
let $\widetilde{A}\in\mathfs L_{x,x'}$, $\widetilde{A}:=P_{z,x'} \circ A\circ P_{x,y}$.
By definition, $\widetilde{A}$ depends on $x,x'$ but different base points define
a map that differs from $\widetilde{A}$ by pre and post composition with isometries.
In particular, $\|\widetilde{A}\|$ does not depend on the choice of $x,x'$.
Similarly, if $A_i\in\mathfs L_{y_i,z_i}$ then $\|\widetilde{A_1}-\widetilde{A_2}\|$ does
not depend on the choice of $x,x'$.
With this notation, we say that $f$ is a $C^{1+\beta}$ diffeomorphism if $f\in C^1$ and there exists
$\mathfrak K>0$ s.t.:
\begin{enumerate}[$\circ$]
\item If $y_1,y_2\in D_x$ and $f(y_1),f(y_2)\in D_{x'}$ then
$\|\widetilde{df_{y_1}}-\widetilde{df_{y_2}}\|\leq \mathfrak Kd(y_1,y_2)^\beta$.
\item If $y_1,y_2\in D_x$ and $f^{-1}(y_1),f^{-1}(y_2)\in D_{x''}$ then
$\|\widetilde{df_{y_1}^{-1}}-\widetilde{df_{y_2}^{-1}}\|\leq \mathfrak Kd(y_1,y_2)^\beta$.
\end{enumerate}

\subsection{Lyapunov inner product}\label{Subsec-UH-adapted}

By the definitions in Subsection \ref{Subsec-UH-definitions}, hyperbolicity implies that
the restriction of $df$ to $E^s$ is eventually a contraction, exactly when $n$ is large enough
so that $C\kappa^n<1$, and the same occurs to the restriction of $df^{-1}$ to $E^u$. It turns out
that we can define a new inner product, equivalent to the original, for which $df\restriction_{E^s}$ and
$df^{-1}\restriction_{E^u}$ are contractions already since the first iterate. Such inner product
is known as {\em adapted metric} or  {\em Lyapunov inner product}. For consistency
with the nonuniformly hyperbolic situation, we will use the later notation.
The idea of changing an eventual contraction to a contraction is a popular trick in dynamics,
and it appears in various contexts, from Picard's theorem on existence and uniqueness
of solutions of ordinary differential equations to the construction of invariant manifolds,
as we will see here. There are many different ways of defining such inner product, see e.g.
\cite[Proposition 4.2]{Shub-Book}. Here, we follow an approach similar to 
\cite[Proposition 5.2.2]{Brin-Stuck-Book}.

\medskip
We assume that $f:M\to M$ is a uniformly hyperbolic diffeomorphism, and we let
$\langle\cdot,\cdot\rangle$ be the Riemannian metric on $M$. For simplicity of notation,
we assume that $f$ is Anosov, with invariant splitting $TM=E^s\oplus E^u$ (for Axiom A,
the definitions are made inside $\Omega(f)$).
Fix $\kappa<\lambda<1$.

\medskip
\noindent {\sc Lyapunov inner product:} We define an inner product $\llangle \cdot,\cdot\rrangle$ on $M$,
called {\em Lyapunov inner product}, by the following identities:
\begin{enumerate}[$\circ$]
\item For $v_1^s,v_2^s\in E^s$:
$$
\llangle v^s_1,v^s_2\rrangle=2\sum_{n\geq 0}\lambda^{-2n}\langle df^nv^s_1,df^nv^s_2\rangle.
$$
\item For $v_1^u,v_2^u\in E^u$:
$$
\llangle v^u_1,v^u_2\rrangle=2\sum_{n\geq 0}\lambda^{-2n}\langle df^{-n}v^u_1,df^{-n}v^u_2\rangle.
$$
\item For $v^s\in E^s$ and $v^u\in E^u$:
$$
\llangle v^s,v^u\rrangle=0.
$$
\end{enumerate}

\medskip
\noindent
We can show, using the uniform hyperbolicity, that $\llangle \cdot,\cdot\rrangle$ is equivalent
to and as smooth as $\langle\cdot,\cdot\rangle$. That is why it is also called {\em adapted metric}.
Letting $\vertiii{\cdot}$ denote the norm induced by $\llangle \cdot,\cdot\rrangle$,
if $v^s\in E^s\backslash\{0\}$ then
$$
\vertiii{df v^s}^2=2\sum_{n\geq 0}\lambda^{-2n}\|df^n(df v^s)\|^2=
\lambda^2\left[\vertiii{v^s}^2-2\right]< \lambda^2\vertiii{v^s}^2,
$$
hence $\vertiii{df v^s}<\lambda \vertiii{v^s}$. 
Similarly, if $v^u\in E^u\backslash\{0\}$ then $\vertiii{df^{-1}v^u}<\lambda\vertiii{v^u}$.

\medskip
When $M$ is a surface, the information of the Lyapunov inner product at each $x\in M$ can be
recorded by three parameters $s(x),u(x),\alpha(x)$, which we now introduce.
The bundles $E^s,E^u$ are one-dimensional, so there are vectors
$e^s_x\in E^s_x$ and $e^u_x\in E^u_x$, unitary in the metric $\langle\cdot,\cdot\rangle$.
In the Lyapunov inner product $\llangle\cdot,\cdot\rrangle$,
we have that
$\vertiii{e^s_x},\vertiii{e^u_x}\in\left[\sqrt{2},C\lambda\sqrt{\tfrac{2}{\lambda^2-\kappa^2}}\right]$
are uniformly bounded away from zero and infinity\footnote{Indeed: $2<\vertiii{e^s_x}^2=
2\displaystyle\sum_{n\geq 0}\lambda^{-2n}\|df^ne^s_x\|^2\leq
2C^2\sum_{n\geq 0}\left(\tfrac{\kappa}{\lambda}\right)^{2n}
=\tfrac{2C^2\lambda^2}{\lambda^2-\kappa^2}$.}.

\medskip
\noindent
{\sc Parameters $s(x),u(x),\alpha(x)$:}
\begin{align*}
s(x)&=\vertiii{e^s_x}=\sqrt{2}\left(\sum_{n\geq 0}\lambda^{-2n}\|df^ne^s_x\|^2\right)^{1/2}\\
u(x)&=\vertiii{e^u_x}=\sqrt{2}\left(\sum_{n\geq 0}\lambda^{-2n}\|df^{-n}e^u_x\|^2\right)^{1/2}\\
\alpha(x)&=\angle(E^s_x,E^u_x).
\end{align*} 

\medskip
As observed, $s(x),u(x)$ are uniformly bounded away from zero and infinity. Since the splitting
$E^s\oplus E^u$ is continuous,  $\alpha(x)$ is also uniformly bounded away from zero and $\pi$.


\subsection{Diagonalization of derivative}\label{Subsec-UH-diagonalization}

For the ease of exposition, we continue assuming that $M$ is a surface.
We now use the Lyapunov inner product (or, more specifically, the parameters $s(x),u(x),\alpha(x)$)
to represent $df_x$ as a hyperbolic matrix. Let $e_1=\begin{bmatrix}1 \\ 0 \end{bmatrix}$
and $e_2=\begin{bmatrix}0 \\ 1 \end{bmatrix}$ be the canonical basis of $\R^2$.

\medskip
\noindent
{\sc Linear map $C(x):$} For $x\in M$, let
$C(x):\R^2\to T_xM$ be the linear map s.t.
$$
C(x):e_1\mapsto \frac{e^s_x}{s(x)}\ ,\ C(x): e_2\mapsto \frac{e^u_x}{u(x)}\cdot
$$

\medskip
The linear transformation $C(x)$ sends the canonical inner product on $\R^2$ to the Lyapunov inner product
$\llangle\cdot,\cdot\rrangle$ on $T_xM$.
For a geometer, this may be a simple description of $\llangle\cdot,\cdot\rrangle$,
but for practical reasons we do not explore such description. Instead, we study the
relation of $C(x)$ with the parameters $s(x),u(x),\alpha(x)$.
Given a linear transformation, let $\|\cdot\|$ denote its sup norm and $\|\cdot\|_{\rm Frob}$ its
Frobenius norm\footnote{The Frobenius norm
of a $2\times 2$ matrix $A=\left[\begin{array}{cc}a & b\\ c & d\end{array}\right]$ is
$\|A\|_{\rm Frob}=\sqrt{a^2+b^2+c^2+d^2}$.}. These two norms are equivalent,
with $\|\cdot\|\leq \|\cdot\|_{\rm Frob}\leq \sqrt{2}\|\cdot\|$.
The next lemma proves that $C$ diagonalizes $df$.
  
\begin{lemma}\label{Lemma-linear-reduction}
There is $\mathfs L>1$ s.t. the following holds for all $x\in M$:
\begin{enumerate}[{\rm (1)}]
\item $\|C(x)\|_{\rm Frob}\leq 1$
and $\|C(x)^{-1}\|_{\rm Frob}=\tfrac{\sqrt{s(x)^2+u(x)^2}}{|\sin\alpha(x)|}$, with
$\|C(x)\|,\|C(x)^{-1}\|\leq \mathfs L$.
\item $C(f(x))^{-1}\circ df_x\circ C(x)$ is a diagonal matrix with diagonal entries $A,B\in\R$
s.t. $|A|,|B^{-1}|<\lambda$.
\end{enumerate}
\end{lemma}

\begin{proof}
(1) In the basis $\{e_1,e_2\}$ of $\R^2$ and the basis $\{e^s_x,(e^s_x)^\perp\}$ of $T_xM$, $C(x)$ takes the
form $\left[\begin{array}{cc}\tfrac{1}{s(x)}& \tfrac{\cos\alpha(x)}{u(x)}\\ 0&\tfrac{\sin\alpha(x)}{u(x)}\end{array}\right]$,
hence $\|C(x)\|_{\rm Frob}^2=\tfrac{1}{s(x)^2}+\tfrac{1}{u(x)^2}\leq 1$. The inverse of
$C(x)$ is
$\left[\begin{array}{cc}s(x)& -\tfrac{s(x)\cos\alpha(x)}{\sin\alpha(x)}\\ 0&\tfrac{u(x)}{\sin\alpha(x)}\end{array}\right]$,
therefore $\|C(x)^{-1}\|_{\rm Frob}=\tfrac{\sqrt{s(x)^2+u(x)^2}}{|\sin\alpha(x)|}$.
Finally, since
$s(x),u(x),\alpha(x)$ are uniformly bounded away from zero and infinity, there is
$\mathfs L>1$ s.t. $\|C(x)\|,\|C(x)^{-1}\|\leq \mathfs L$ for all $x\in M$.

\medskip
\noindent
(2) It is clear that $e_1,e_2$ are eigenvectors of $C(f(x))^{-1}\circ df_x\circ C(x)$.
We start calculating the eigenvalue of $e_1$ .
Since $dfe^s_x=\pm\|dfe^s_x\|e^s_{f(x)}$,
$[df_x\circ C(x)](e_1)= df_x\left[\tfrac{e^s_x}{s(x)}\right]=\pm\tfrac{\|dfe^s_x\|}{s(x)}e^s_{f(x)}$, hence
$[C(f(x))^{-1}\circ df_x\circ C(x)](e_1)=\pm\|dfe^s_x\|\tfrac{s(f(x))}{s(x)}e_1$.
Thus $A:=\pm\|dfe^s_x\|\tfrac{s(f(x))}{s(x)}$ is the eigenvalue of $e_1$. Since
\begin{align*}
s(f(x))^2&=\tfrac{2\lambda^2}{\|dfe^s_x\|^2}\sum_{n\geq 1}\lambda^{-2n}\|df^ne^s_x\|^2
=\tfrac{\lambda^2}{\|dfe^s_x\|^2}(s(x)^2-2)<\tfrac{\lambda^2s(x)^2}{\|dfe^s_x\|^2},
\end{align*}
we have $|A|<\lambda$. Similarly, $B:=\pm\|dfe^u_x\|\tfrac{u(f(x))}{u(x)}$ is the eigenvalue of $e_2$.
Observing that $\|df^{-1}e^u_{f(x)}\|\cdot \|df e^u_x\|=1$, we have
\begin{align*}
u(f(x))^2&=2+\sum_{n\geq 1}\lambda^{-2n}\|df^{-n}e^u_{f(x)}\|^2=
2+\tfrac{u(x)^2}{\lambda^2\|df e^u_x\|^2}=2+\tfrac{u(f(x))^2}{B^2\lambda^2}>\tfrac{u(f(x))^2}{B^2\lambda^2},
\end{align*}
and so $|B|>\lambda^{-1}$.
\end{proof}

Although $s,u,\alpha,C$ depend on the choice of $\lambda$, we will not emphasize
this dependence because all calculations will be made for some a priori fixed
$\lambda$.

\subsection{Lyapunov charts, change of coordinates}\label{Subsec-UH-charts}

From now on, we assume that $f$ is $C^{1+\beta}$.
The next step is to compose the linear transformation $C(x)$ with the exponential map
to obtain a local chart of $M$ in which $f$ itself becomes a small perturbation of a hyperbolic
matrix. Since this is a natural consequence of the use of the Lyapunov inner product, we will call these
charts {\em Lyapunov charts}, as in \cite[Section 6.4.2]{Barreira-Pesin-new-book}.
Fix a small number $\ve\in (0,\mathfrak r_0)$ (how small depends on a finite number
of inequalities that $\ve$ has to satisfy). Let $Q=\ve^{3/\beta}$. 

\medskip
\noindent
{\sc Lyapunov chart:} The {\em Lyapunov chart} at $x$ is the map $\Psi_x:[-Q,Q]^2\to M$
defined by $\Psi_x:=\exp{x}\circ C(x)$.
\begin{figure}[hbt!]
\centering
\def\svgwidth{11cm}
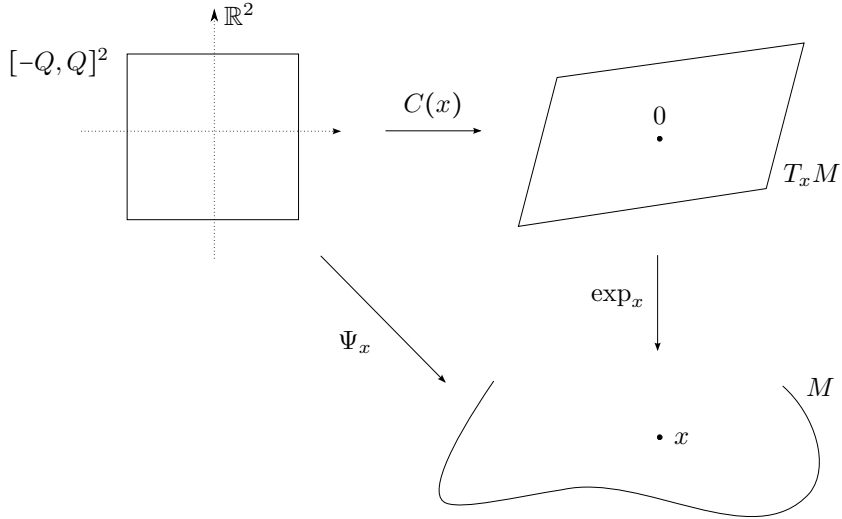
\caption{The Lyapunov chart $\Psi_x$ at $x$.}\label{figure-chart}
\end{figure}

\medskip
Since $Q<\ve<\mathfrak r_0$ and $C(x)$ is a contraction, we
have $C(x)([-Q,Q]^2)\subset B_x[2\mathfrak r_0]$ and so $\Psi_x$ is a diffeomorphism
onto its image. By Lemma \ref{Lemma-linear-reduction}(1), $\Psi_x$ is $2$--Lipschitz
and its inverse is $2\mathfs L$--Lipschitz.
In Lyapunov charts, $f$ takes the form $f_x:=\Psi_{f(x)}^{-1}\circ f\circ \Psi_x$.
The next theorem shows that $f_x$ is a small perturbation of a hyperbolic matrix.

\begin{theorem}\label{Thm-Lyapunov-chart}
The following holds for all $\ve>0$ small enough.
\begin{enumerate}[{\rm (1)}]
\item $d(f_x)_0=C(f(x))^{-1}\circ df_x\circ C(x)=\left[\begin{array}{cc}A & 0\\ 0 & B\end{array}\right]$
with $|A|,|B^{-1}|<\lambda$, cf. Lemma \ref{Lemma-linear-reduction}.
\item $f_x(v_1,v_2)=(Av_1+h_1(v_1,v_2),Bv_2+h_2(v_1,v_2))$ for $(v_1,v_2)\in [-Q,Q]^2$ where:
\begin{enumerate}[{\rm (a)}]
\item $h_1(0,0)=h_2(0,0)=0$ and $\nabla h_1(0,0)=\nabla h_2(0,0)=0$.
\item $\|h_1\|_{1+\beta/2}<\ve$ and $\|h_2\|_{1+\beta/2}<\ve$, where the norms are taken in $[-Q,Q]^2$.
\end{enumerate}
\end{enumerate}
A similar statement holds for $f_x^{-1}:=\Psi_x^{-1}\circ f^{-1}\circ \Psi_{f(x)}$.
\end{theorem}

\begin{proof}
Property (1) is clear since $d(\Psi_x)_0=C(x)$ and $d(\Psi_{f(x)})_0=C(f(x))$.
By Lemma \ref{Lemma-linear-reduction},
$d(f_x)_0=\left[\begin{array}{cc}A & 0 \\ 0 & B\end{array}\right]$ with $|A|,|B^{-1}|<\lambda$.
Define $h_1,h_2:[-Q,Q]^2\to\R $ by 
$f_x(v_1,v_2)=(Av_1+h_1(v_1,v_2),Bv_2+h_2(v_1,v_2))$. Then (a) is automatically
satisfied. It remains to prove (b), which will follow estimating 
$\|d(f_x)_{w_1}-d(f_x)_{w_2}\|$. For the inexperienced reader, we suggest to make the calculation
in the euclidean situation (hence all exponential maps are identity). Below we do the general case.
For $i=1,2$, define
$$
A_i= \widetilde{d(\exp{f(x)}^{-1})_{(f\circ \exp{x})(w_i)}}\,,\
B_i=\widetilde{df_{\exp{x}(w_i)}}\,,\ C_i=\widetilde{d(\exp{x})_{w_i}}.
$$
We first estimate $\|A_1 B_1 C_1-A_2 B_2 C_2\|$. Note that:
\begin{enumerate}[$\circ$]
\item $A_1,A_2$ are derivatives of the map $\exp{f(x)}^{-1}$ at nearby points, and so 
$\|A_1-A_2\|\leq \mathfs 2\mathfs H\|w_1-w_2\|$, where $\mathfs H>0$ is a constant that only
depends on the regularity of exponential maps and their inverses.
\item $B_1,B_2$ are derivatives of $f$ at nearby points, so $\|B_1-B_2\|\leq 2\mathfrak K\|w_1-w_2\|^\beta$.
\item $C_1,C_2$ are derivatives of exponential maps at nearby points, so
$\|C_1-C_2\|\leq \mathfs H\|w_1-w_2\|$.
\end{enumerate}
Applying some triangle inequalities, we obtain that
$$
\|A_1 B_1 C_1-A_2 B_2 C_2\|\leq 24\mathfrak K \mathfs H\|w_1-w_2\|^\beta.
$$
Now we estimate $\|d(f_x)_{w_1}-d(f_x)_{w_2}\|$:
\begin{align*}
&\,\|d(f_x)_{w_1}-d(f_x)_{w_2}\|\leq \|C(f(x))^{-1}\| \|A_1 B_1 C_1-A_2 B_2 C_2\| \|C(x)\|\\
&\leq 24\mathfrak K\mathfs H\mathfs L\|w_1-w_2\|^\beta.
\end{align*}
Since $\|w_1-w_2\|<4Q$, if $\ve>0$ is small enough then
$24\mathfrak K\mathfs H\mathfs L\|w_1-w_2\|^{\beta/2}
\leq 96\mathfrak K\mathfs H\mathfs L\ve^{3/2}<\ve$,
hence $\|d(f_x)_{w_1}-d(f_x)_{w_2}\|\leq \ve \|w_1-w_2\|^{\beta/2}$.
\end{proof}

\subsection{Graph transforms: construction of invariant manifolds}\label{Subsect-UH-graph}

A consequence of the hyperbolic behavior of $f_x$ is that
it sends curves that are almost parallel to the vertical axis to curves with the same property;
similarly, the inverse map $f_x^{-1}$ sends curves that are almost parallel to the horizontal axis
to curves with the same property. This geometrical feature allows to construct
local stable and unstable manifolds. According to Anosov \cite[pp. 23]{Anosov-Geodesic-Flows},
this construction has been more or less known to Darboux, Poincar\'e and Lyapunov, but their proofs
required additional assumptions on the system. Hadamard and Perron were the ones
to observe that hyperbolicity is a sufficient condition. Below, we explain the method of Hadamard.
The idea is to find  the local invariant manifolds among 
graphs of functions, which we call {\em admissible manifolds}. The maps $f_x^{\pm 1}$ define
operators on the spaces of admissible manifolds, called {\em graph transforms}, and the local
invariant manifolds are  limit points of compositions of such operators. As already mentioned, usually
$f$ is only assumed to be $C^1$, but we take $f\in C^{1+\beta}$ to keep the analogy with 
the remaining of the text, and make use of the Lyapunov charts constructed in Subsection \ref{Subsec-UH-charts}.


\medskip
\noindent
{\sc Admissible manifolds:} An {\em $s$--admissible manifold} at $\Psi_x$ is a set
of the form $V^s=\Psi_x\{(t,F(t)):|t|\leq Q\}$, where $F:[-Q,Q]\to\R$ is a $C^1$ function s.t.
$F(0)=0,F'(0)=0$ and $\|F'\|_{C^0}\approx 0$. Similarly, a {\em $u$--admissible manifold} at
$\Psi_x$ is a set of the form $V^u=\Psi_x\{(G(t),t):|t|\leq Q\}$, where $G:[-Q,Q]\to\R$ is a $C^1$
function s.t. $G(0)=0,G'(0)=0$ and $\|G'\|_{C^0}\approx 0$.

\medskip
We call $F,G$ the {\em representing functions} of $V^s,V^u$ respectively. We prefer not to specify
the quantifier for the condition  $\|F'\|_{C^0},\|G'\|_{C^0}\approx 0$.
Instead, think of an $s/u$--admissible as an almost
horizontal/vertical curve that is tangent to the horizontal/vertical axis at the origin. 
Let $\mathfs M^s_x,\mathfs M^u_x$ be the space of all $s,u$--admissible manifolds at $\Psi_x$
respectively.
Introduce a metric on $\mathfs M^s_x$ as follows: for $V_1,V_2\in\mathfs M^s_x$ with representing functions
$F_1,F_2$, let
$$
{\rm dist}(F_1,F_2):=\|F_1-F_2\|_{C^0}.
$$ A similar definition holds for $\mathfs M^u_x$.

\medskip
\noindent
{\sc Graph transforms $\mathfs F^s_x,\mathfs F^u_x$:} 
The {\em stable graph transform} $\mathfs F^s_x:\mathfs M^s_{f(x)}\to\mathfs M^s_x$
is the map that sends $V^s\in \mathfs M^s_{f(x)}$ to the unique $\mathfs F^s_x[V^s]\in \mathfs M^s_x$
with representing function $F$ s.t. $\Psi_x\{(t,F(t)):|t|\leq Q\}\subset f^{-1}(V^s)$.
Similarly, the {\em unstable graph transform}
$\mathfs F^u_x:\mathfs M^u_x\to\mathfs M^u_{f(x)}$ is the map that sends
$V^u\in \mathfs M^u_x$ to the unique $\mathfs F^u_x[V^u]\in \mathfs M^u_{f(x)}$
with representing function $G$ s.t. $\Psi_{f(x)}\{(G(t),t):|t|\leq Q\}\subset f(V^u)$.

\medskip
In other words, $\mathfs F^s_x$ sends an $s$--admissible manifold at $\Psi_{f(x)}$
with representing function $F$ to an $s$--admissible manifold at $\Psi_x$
whose graph of the representing function is contained in the graph of $f_x^{-1}\circ F$,
and $\mathfs F^u_x$ sends a $u$--admissible manifold at $\Psi_x$
with representing function $G$ to a $u$--admissible manifold at $\Psi_{f(x)}$
whose graph of the representing function is contained in the graph of $f_x\circ G$.
See Figure \ref{figure-graph-transform} below.
\begin{figure}[hbt!]
\centering
\def\svgwidth{13cm}
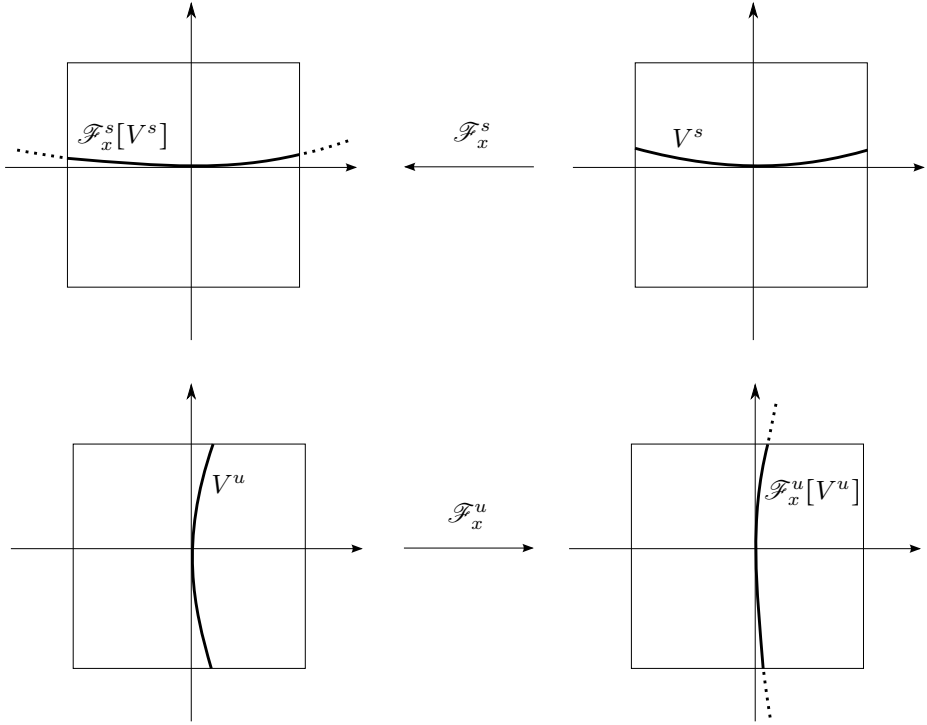
\caption{Graph transforms: the stable graph transform $\mathfs F_x^s$ sends an
$s$--admissible manifold at $\Psi_{f(x)}$
to an $s$--admissible manifold at $\Psi_x$, while the unstable graph transform
$\mathfs F_x^u$ sends a $u$--admissible manifold at $\Psi_x$ to a $u$--admissible
manifold at $\Psi_{f(x)}$.}\label{figure-graph-transform}
\end{figure}

\begin{theorem}\label{Thm-UH-graph-transform}
$\mathfs F^s_x$ and $\mathfs F^u_x$ are well-defined contractions.
\end{theorem}

The proof of the theorem follows from the
hyperbolicity of $f_x$. Using it, we can construct
local stable and unstable manifolds. 

\medskip
\noindent
{\sc Stable/unstable manifolds:} The {\em stable manifold} of $x\in M$ is
the $s$--admissible manifold $V^s[x]\in \mathfs M^s_x$ defined by
$$
V^s[x]:=\lim_{n\to\infty}(\mathfs F^s_x\circ\cdots\circ \mathfs F^s_{f^{n-1}(x)})[V_n]
$$
for some (any) sequence $\{V_n\}_{n\geq 0}$ with $V_n\in\mathfs M^s_{f^n(x)}$.
The {\em unstable manifold} of $x\in M$ is the $u$--admissible manifold $V^u[x]\in \mathfs M^u_x$ defined by
$$
V^u[x]:=\lim_{n\to-\infty}(\mathfs F^u_{f^{-1}(x)}\circ\cdots\circ \mathfs F^u_{f^{n}(x)})[V_n]
$$
for some (any) sequence $\{V_n\}_{n\leq 0}$ with $V_n\in\mathfs M^u_{f^{n}(x)}$.

\medskip
The sets $V^s[x]$ and $V^u[x]$ are well-defined because the graph transforms are contractions
(Theorem \ref{Thm-UH-graph-transform} above), and they are indeed admissible curves.
Note that $V^s[x]$ only depends on the future
$\{f^n(x)\}_{n\geq 0}$, while $V^u[x]$ only depends on the past
$\{f^n(x)\}_{n\leq 0}$. Also, since $\Psi_x$ and its inverse have uniformly bounded norms
(Lemma \ref{Lemma-linear-reduction}(a)), the stable and unstable manifolds have uniform sizes.

\subsection{Higher dimensions}\label{Subsect-UH-higher-dim}

We sketch how to make the construction in higher dimension.
As we have done, the definition of Lyapunov inner product works in any dimension,
but not the definition of $C(x)$. Since this matrix is used to send the canonical inner product of $\R^2$
to the Lyapunov inner product on $T_xM$, in arbitrary dimension we can similarly define
$C(x):\R^n\to T_xM$ to be a linear transformation s.t.
$\langle v,w\rangle_{\R^n}=\llangle C(x)v,C(x)w\rrangle$ for all $v,w\in\R^n$, i.e.
$C(x)$ is an isometry between $(\R^n,\langle\cdot,\cdot\rangle_{\R^n})$ and $(T_xM,\llangle\cdot,\cdot\rrangle)$.
Let $d_s,d_u\in\N$ be the dimensions of $E^s,E^u$.
The map $C(x)$ is not uniquely defined, and we can assume that it sends
$\R^{d_s}\times\{0\}$ to $E^s_x$ and $\{0\}\times \R^{d_u}$
to $E^u_x$. Doing this for all $x\in M$, we obtain a family $\{C(x)\}_{x\in M}$ of linear transformations.
Although the splitting $TM=E^s\oplus E^u$ is continuous,
we cannot always take $x\in M\mapsto C(x)$ continuously, because
$E^s$ and $E^u$ may have non-real exponents, causing rotations inside them.
But for our purpose, what matters is the behavior of the sequence $\{C(f^n(x))\}_{n\in\Z}$ for each $x\in M$.
For what it is worth, $x\in M\mapsto C(x)$ can be chosen measurably, see e.g.
\cite[Footnote at page 48]{Ben-Ovadia-2019}.

\medskip
The composition $C(f(x))^{-1}\circ df_x\circ C(x)$ takes the block form
$$
C(f(x))^{-1}\circ df_x\circ C(x)=
\left[
\begin{array}{cc}
D_s & 0 \\
0 & D_u 
\end{array}
\right],
$$
where $D_s$ is a $d_s\times d_s$ matrix s.t. $\|D_s v\|\leq \lambda\|v\|$ for all $v\in \R^{d_s}$,
and $D_u$ is a $d_u\times d_u$ matrix s.t. $\|D_u^{-1} w\|\leq \lambda\|w\|$ for all $w\in \R^{d_u}$.
This is the higher dimensional counterpart of Lemma \ref{Lemma-linear-reduction}(2).
Define the Lyapunov chart $\Psi_x$ as in Subsection \ref{Subsec-UH-charts},
which satisfies a higher dimensional version of Theorem \ref{Thm-Lyapunov-chart} with respect to
the above block form.
Defining an $s$--admissible manifold at $\Psi_x$ as a set
of the form $V^s=\Psi_x\{(t,F(t)):t\in [-Q,Q]^{d_s}\}$, where $F:[-Q,Q]^{d_s}\to\R^{d_u}$ is a
$C^1$ function s.t. $F(0)=0,F'(0)=0$ and $\|F'\|_{C^0}\approx 0$, and similarly $u$--admissible manifolds
at $\Psi_x$, Theorem \ref{Thm-UH-graph-transform} holds. Hence,
every $x\in M$ has local stable and unstable manifolds.

\section{Nonuniformly hyperbolic systems}\label{Section-NUH-systems}

As discussed in Section \ref{Section-UH-Systems}, uniformly hyperbolic systems
had a big impact on the development of dynamical systems and ergodic theory.
Unfortunately, uniform hyperbolicity is a condition that is hard to be satisfied. For instance, 
if a three dimensional manifold admits an Anosov flow then its fundamental group 
has exponential growth \cite{Plante-Thurston}.
During the seventies, new notions of hyperbolicity were proposed. These notions
substitute the uniform assumption to weaker ones. One of them, called
{\em nonuniform hyperbolicity}, was introduced by Pesin \cite{Pesin-Izvestia-1976,Pesin-Characteristic-1977,Pesin-Geodesic-1977}. In contrast to uniform hyperbolicity, which requires
hyperbolicity to hold every time (for all $n\in\mathbb Z$ or $t\in\mathbb R$) and everywhere (for all $x\in M$),
the notion of nonuniform hyperbolicity assumes an asymptotic hyperbolicity (on the average)
not necessarily in the whole phase space (almost everywhere), i.e. hyperbolicity
occurs but in a nonuniform way. 

\medskip
Here is a simplistic way of comparing uniform and nonuniform hyperbolicity. To prepare the dough
for bread, a baker needs to contract and stretch repeatedly the dough. An ideal baker would make
such operation at every movement, all over the dough. This is uniform hyperbolicity. 
On the other hand, a real-life baker performs nonuniform hyperbolicity: he does not make the
operation at every movement (he might get tired from time to time) and he can forget some tiny
parts of the dough. As it turns out, it is the notion of nonuniform hyperbolicity that allows for
applications outside of Mathematics. 

\medskip
At the same time that nonuniform hyperbolicity is weak enough to include many new examples and applications,
it is strong enough to recover many of the properties of uniformly hyperbolic systems, such
as stable manifolds and graph transforms. This is one of the reasons of the success of the theory of nonuniformly
hyperbolic systems, known as {\em Pesin theory}. Since its beginning, it constitutes
an important tool for the understanding of ergodic and statistical properties of smooth
dynamical systems.
Nowadays, Pesin theory is classical and there are great textbooks on the topic, see e.g.
\cite{Fathi-Herman-Yoccoz,Katok-Mendoza,Barreira-Pesin-2007}.
For the applications to symbolic dynamics in Part 2, we follow the modern
approach that was recently developed by Sarig \cite{Sarig-JAMS}, which has been
slightly improved in the past two years or so \cite{Lima-Sarig,Lima-Matheus,Ben-Ovadia-2019}.

\medskip
Now, we make {\em essential} use of the $C^{1+\beta}$ regularity. Indeed, the theory is just not true 
under $C^1$ regularity, see e.g. Pugh's example in \cite[Chapter 15]{Barreira-Pesin-new-book}.

\subsection{Definitions and examples}\label{Subsec-NUH-definitions}

Let $M$ be a closed (compact without boundary) connected
smooth Riemannian manifold, and let $f:M\to M$ be a $C^1$ diffeomorphism.
The objects that identify the asymptotic hyperbolicity are the {\em Lyapunov exponents}.

\medskip
\noindent
{\sc Lyapunov exponent:} For a nonzero vector  $v\in TM$, the {\em Lyapunov exponent}
of $f$ at $v$ is defined by
$$
\chi(v):=\lim_{n\to+\infty}\tfrac{1}{n}\log\|df^n v\|
$$
when the limit exists.

\medskip
The mere existence of the limit should not be taken for granted, even for
uniformly hyperbolic systems. It comes from the 
{\em Oseledets theorem}, which is a measure-theoretic statement that we now explain.
Assume that $\mu$ is a probability measure on $M$, invariant under $f$. For simplicity,
assume that $\mu$ is ergodic. The Oseledets theorem proves that, for $\mu$--a.e.
$x\in M$, the Lyapunov exponents of every nonzero $v\in T_xM$ exist \cite{Oseledets}.
Furthermore, they exists and are equal for future or past iterations. In our context, we state it as follows.

\begin{theorem}[Oseledets]\label{Thm-Oseledets}
Let $(f,\mu)$ be as above. Then there is a $f$--invariant subset $\widetilde M\subset M$ with
$\mu\left[\widetilde M\right]=1$, real numbers $\chi_1<\chi_2<\cdots<\chi_k$, and a splitting
$T\widetilde M=E^1\oplus E^2\oplus\cdots\oplus E^k$ satisfying:
\begin{enumerate}[{\rm (1)}]
\item {\sc Invariance:} $df(E^i_{x})=E^i_ {f(x)}$ for all $x\in \widetilde M$ and $i=1,\ldots,k$.
\item {\sc Lyapunov exponents:} For all $x\in\widetilde M$ and all nonzero $v\in E^i_x$,
$$
\chi(v)=\lim_{n\to\pm\infty}\tfrac{1}{n}\log\|df^n v\|=\chi_i.
$$
\end{enumerate}
\end{theorem}

When $\mu$ is not ergodic, we can apply a standard argument of ergodic decomposition
to conclude that $\chi(v)$ exists $\mu$--a.e., but now its value depends on $x$.
Theorem \ref{Thm-Oseledets} above follows from the general version of the Oseledets theorem
on cocycles satisfying an {\em integrability condition}, see e.g. the recent survey of Filip \cite{Filip-survey}.
In our setting, the integrability condition is that $\log \|df^{\pm 1}\|\in L^1(\mu)$.
Since $f$ is a diffeomorphism on a closed manifold, this condition is automatically satisfied.

\medskip
The notion of {\em nonuniform hyperbolicity} also depends on a measure. Let
$(f,\mu)$ be as above, where $\mu$ is not necessarily ergodic.

\medskip
\noindent
{\sc Nonuniformly hyperbolic diffeomorphism:} The pair $(f,\mu)$
is called {\em nonuniformly hyperbolic} if for $\mu$--a.e. $x\in M$
we have $\chi(v)\neq 0$ for all nonzero $v\in T_xM$. In this case, $\mu$ is called
a {\em hyperbolic} measure.

\medskip
If $f$ is uniformly hyperbolic (see the notation of Subsection \ref{Subsec-UH-definitions}),
then $\chi(v)\leq \log\kappa < 0$ for nonzero $v\in E^s$, and 
$\chi(v)\geq -\log\kappa > 0$ for nonzero $v\in E^u$, wherever the Lyapunov exponents exist.
Thus a uniformly hyperbolic diffeomorphism is nonuniformly hyperbolic, for 
any invariant probability measure. Here is another example: if $f$ is a diffeomorphism
and $p\in M$ is a hyperbolic periodic point with period $n$, then $\mu=\tfrac{1}{n}\sum_{k=0}^{n-1}\delta_{f^k(p)}$
is a hyperbolic measure. 
We will usually assume that the Lyapunov exponents are bounded away from zero.

\medskip
\noindent
{\sc $\chi$--hyperbolic measure:} Given $(f,\mu)$ nonuniformly hyperbolic
and $\chi>0$, $\mu$ is called {\em $\chi$--hyperbolic} if for $\mu$--a.e. $x\in M$
we have $|\chi(v)|>\chi$ for all nonzero $v\in T_xM$.

\medskip
In the latter example, $\mu$ is $\chi$--hyperbolic for all $\chi$ smaller
than the multiplier of $p$.
Now let $\vf:M\to M$ be a flow generated by a vector field $X$ of class $C^1$, and let
$\mu$ be a $\vf$--invariant probability measure on $M$. Since 
$d\vf^t\circ X=X\circ\vf^t$, we have $\chi(X_x)=0$ for all $x\in M$, hence the assumption of 
nonzero exponents is required along the remaining directions. 

\medskip
\noindent
{\sc Nonuniformly hyperbolic flow:} The pair $(\vf,\mu)$
is called {\em nonuniformly hyperbolic} if for $\mu$--a.e. $x\in M$
we have $\chi(v)\neq 0$ for all $v\in T_xM$ not proportional to $X_x$.
When this happens, the measure $\mu$ is called {\em hyperbolic}.

\medskip
Again, it is easy to see that a uniformly hyperbolic flow is nonuniformly hyperbolic
for any invariant probability measure, and that the Dirac measure
of a hyperbolic closed orbit is hyperbolic. The notion of $\chi$--hyperbolic
measure is defined accordingly. 

\medskip
\noindent
{\sc $\chi$--hyperbolic measure:} Given $(\vf,\mu)$ nonuniformly hyperbolic
and $\chi>0$, $\mu$ is called {\em $\chi$--hyperbolic} if for $\mu$--a.e. $x\in M$
we have $|\chi(v)|>\chi$ for all nonzero $v\in T_xM$ transverse to $X_x$.

\medskip
Let us mention some classical examples.

\medskip
\noindent
1. The slow down of $f_A:\T^2\to\T^2$,
see \cite{Katok-Bernoulli} and \cite[\S 1.3]{Barreira-Pesin-new-book}.

\medskip
\noindent
2. Let $f$ be a $C^1$ surface diffeomorphism,
and let $\mu$ be an ergodic $f$--invariant probability measure.
Let $h=h_\mu(f)$ be Kolmogorov-Sina{\u\i} entropy, and assume that $h>0$.
Then $(f,\mu)$ is nonuniformly hyperbolic, as consequence of
the Ruelle inequality applied to $f$ and to $f^{-1}$: 
\begin{enumerate}[$\circ$]
\item $f$ has a positive Lyapunov exponent $\chi^+\geq h>0$.
\item We have $h_\mu(f^{-1})=h$, and the Lyapunov spectrum of $(f^{-1},\mu)$
is minus the Lyapunov spectrum of $(f,\mu)$, hence $f$
has a Lyapunov exponent $\chi^{-}$ such that $-\chi^{-}\geq h$, i.e. $\chi^{-}\leq -h<0$.
If in addition $h>\chi$, then $\mu$ is $\chi$--hyperbolic. 
\end{enumerate}

\medskip
\noindent
3. Let $N$ be a closed manifold with nonpositive sectional curvature, e.g. the surface
in Figure \ref{figure-nonpositive} containing a flat cylinder between two regions
of negative curvature.
\begin{figure}[hbt!]
\centering
\def\svgwidth{10cm}
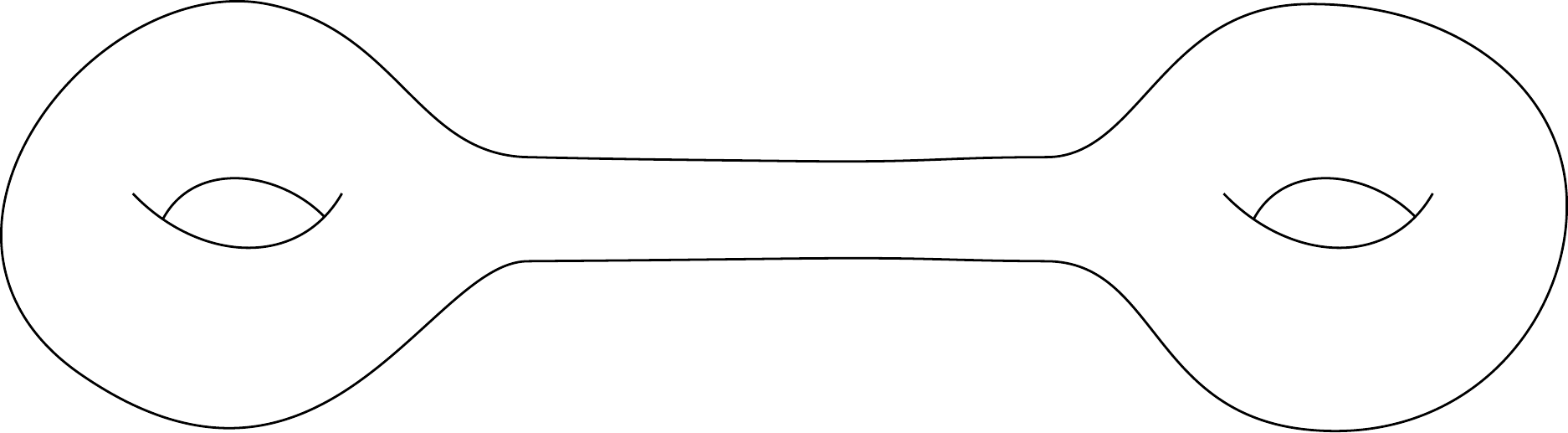\caption{Example of a surface with nonpositive curvature.}
\label{figure-nonpositive}
\end{figure}
The geodesic flow on $N$, which is defined on $M=T_1N$,
has a natural invariant volume mesure $\mu$.
Pesin showed that if the trajectory of a vector $x\in M$ spends a positive fraction of time in regions
of negative sectional curvature, then $\chi(v)\neq 0$ for all $v\in T_xM$ transverse to $X_x$,
see \cite[Thm 10.5]{Pesin-Characteristic-1977}.
The underlying philosophy (although not entirely correct), is that
in regions of negative sectional curvature the derivative behaves like in a uniformly hyperbolic
flow, and in regions of zero sectional curvature it only varies linearly, so the overall
exponential behavior beats the linear one.
Therefore, if $\mu$ is ergodic then it is hyperbolic. 
Unfortunately, the ergodicity of $\mu$ is still an open problem (even when $N$ is a surface).

\subsection{The nonuniformly hyperbolic locus ${\rm NUH}_\chi$}\label{Subsec-NUH-locus}

As we have done above, the notion of nonuniform hyperbolicity 
is an almost-everywhere statement that depends on a measure. Due to the Oseledets theorem,
we can still get almost-everywhere statements if we only consider
{\em Lyapunov regular} points, which are points that satisfy Theorem \ref{Thm-Oseledets}
and a non-degeneracy assumption on the angles $\angle(E^i,E^j)$ between the invariant subbundles.
For some applications, this restriction is cumbersome.
For example, if $x,y$ are Lyapunov regular, then most likely points in $W^s(x)\cap W^u(y)$
are {\em not} Lyapunov regular (this happens e.g. when $x,y$ have
different Lyapunov exponents).

\medskip  
In what follows, we employ a different approach. We fix some $\chi>0$ and consider the set
of points satisfying a weaker notion of nonuniform hyperbolicity, that still allows to construct
local invariant manifolds. This perspective appeared in an essential way in the work in preparation of the author
with Buzzi and Crovisier \cite{BCL}. Independently and simultaneously, Ben Ovadia recently obtained
a similar characterization in higher dimensions \cite{Ben-Ovadia-codable}.

\medskip
Let $f:M\to M$ be a $C^1$ diffeomorphism.
As in Section \ref{Section-UH-Systems}, we start assuming that $M$ is a closed surface.
Let $\chi>0$.

\medskip
\noindent
{\sc The nonuniformly hyperbolic locus ${\rm NUH}_\chi$:} It is the set of points \label{Def-NUH}
$x\in M$ for which there are transverse unitary vectors $e^s_x,e^u_x\in T_xM$ s.t.:
\begin{enumerate}[(NUH1)]
\item $e^s_x$ contracts in the future at least $-\chi$ and expands in the past:
$$\limsup_{n\to+\infty}\tfrac{1}{n}\log \|df^n e^s_x\|\leq -\chi\ \text{ and } 
\ \liminf_{n\to+\infty}\tfrac{1}{n}\log \|df^{-n} e^s_x\|>0.
$$ 
\item $e^u_x$ contracts in the past at least $-\chi$ and expands in the future:
$$\limsup_{n\to+\infty}\tfrac{1}{n}\log \|df^{-n} e^u_x\|\leq -\chi\text{ and } 
\liminf_{n\to+\infty}\tfrac{1}{n}\log \|df^n e^u_x\|>0.$$
\item The parameters $s(x),u(x)$ below are finite:
\begin{align*}
s(x)&=\sqrt{2}\left(\sum_{n\geq 0}e^{2n\chi}\|df^ne^s_x\|^2\right)^{1/2}\in [\sqrt{2},\infty)\\
u(x)&=\sqrt{2}\left(\sum_{n\geq 0}e^{2n\chi}\|df^{-n}e^u_x\|^2\right)^{1/2}\in [\sqrt{2},\infty).
\end{align*} 
\end{enumerate}

\medskip
Clearly, ${\rm NUH}_\chi$ is invariant by $f$.
Observe that the definitions of $s(x),u(x)$ are the same as those given in Subsection \ref{Subsec-UH-adapted},
where we change $\lambda$ to $e^{-\chi}$.
The first two conditions above guarantee that $e^s_x,e^u_x$ are defined up to a sign, and
the last condition guarantees asymptotic contractions of rates at least $-\chi$.
These conditions are weaker than Lyapunov regularity, hence
${\rm NUH}_\chi$ contains all Lyapunov regular points with exponents greater than $\chi$
in absolute value. In particular, ${\rm NUH}_\chi$ carries all $\chi$--hyperbolic measures.
But ${\rm NUH}_\chi$ might contain points with some Lyapunov exponents equal to $\pm\chi$,
and even non-regular points, where the contraction
rates oscillate infinitely often. Usually, ${\rm NUH}_\chi$ is a non-compact subset of $M$.
Observe that if (NUH3) holds, then the first conditions of (NUH1)--(NUH2) hold as well.
In practice, this is how we will show that $x\in{\rm NUH}_\chi$.

\medskip
The quality of hyperbolicity can be measured from the parameters $s(x),u(x)$ and from the
angle $\alpha(x)=\angle(e^s_x,e^u_x)$. More specifically, 
$x\in{\rm NUH}_\chi$ has bad hyperbolicity when at least one of the following situations occur:
\begin{enumerate}[$\circ$]
\item $s(x)$ is large: it takes a long time to see forward contraction along $e^s_x$.
\item $u(x)$ is large: it takes a long time to see backward contraction along $e^u_x$.
\item $\alpha(x)$ is small: it is hard to distinguish the stable and unstable directions.
\end{enumerate}
None of these situations happen for uniformly hyperbolic systems: as we have seen in 
Subsection \ref{Subsec-UH-adapted}, for uniformly hyperbolic systems the parameters $s,u,\alpha$
are uniformly bounded away from zero and infinity. For nonuniformly hyperbolic
systems, the behaviour is more complicated.
Another reason for complication is that,
contrary to uniformly hyperbolic systems, the maps
$x\in{\rm NUH}_\chi\mapsto e^s_x,e^u_x$ are usually not more than just measurable.

\subsection{Diagonalization of derivative}\label{Subsec-NUH-diagonalization}

As in Subsection \ref{Subsec-UH-diagonalization}, we define linear maps
$C(x)$ that diagonalize $df$, the difference being that we only take 
$x\in{\rm NUH}_\chi$.

\medskip
\noindent
{\sc Linear map $C(x):$} For $x\in {\rm NUH}_\chi$, let
$C(x):\R^2\to T_xM$ be the linear map s.t.
$$
C(x):e_1\mapsto \frac{e^s_x}{s(x)}\ ,\ C(x): e_2\mapsto \frac{e^u_x}{u(x)}\cdot
$$

\medskip
Above, $\{e_1,e_2\}$ is the canonical basis for $\R^2$.
If, for each $x\in {\rm NUH}_\chi$, we define a Lyapunov inner product
$\llangle\cdot,\cdot\rrangle$ on $T_xM$, then 
$C(x)$ sends the canonical metric on $\R^2$ to $\llangle\cdot,\cdot\rrangle$.
Lemma \ref{Lemma-linear-reduction} remains valid, except for the uniform bound on
$\|C(x)^{-1}\|$. This result is known as
{\em Oseledets-Pesin reduction}, see e.g. \cite[Theorem 6.10]{Barreira-Pesin-new-book}.

\begin{lemma}[Oseledets-Pesin reduction]\label{Lemma-linear-reduction-2}
The following holds for all $x\in {\rm NUH}_\chi$:
\begin{enumerate}[{\rm (1)}]
\item $\|C(x)\|_{\rm Frob}\leq 1$
and $\|C(x)^{-1}\|_{\rm Frob}=\tfrac{\sqrt{s(x)^2+u(x)^2}}{|\sin\alpha(x)|}$.
\item $C(f(x))^{-1}\circ df_x\circ C(x)$ is a diagonal matrix with diagonal entries $A,B\in\R$
s.t. $|A|,|B^{-1}|<e^{-\chi}$.
\end{enumerate}
\end{lemma}

The proof is the same as in Lemma \ref{Lemma-linear-reduction}.

\subsection{Pesin charts, the parameter $Q(x)$, change of coordinates}\label{Subsec-NUH-charts}

From now on, we assume that $f$ is $C^{1+\beta}$. Remember we are also assuming that
$M$ is a closed surface. Fix $\ve\in (0,\mathfrak r_0)$ small.

\medskip
\noindent
{\sc Pesin chart:} The {\em Pesin chart} at $x\in{\rm NUH}_\chi$ is the map
$\Psi_x:\left[-\ve^{3/\beta},\ve^{3/\beta}\right]^2\to M$ defined by $\Psi_x:=\exp{x}\circ C(x)$.

\medskip
This is exactly the same definition of Lyapunov chart given in Subsection \ref{Subsec-UH-charts},
but we call it Pesin chart for historical reasons. The map $\Psi_x$ is well-defined for
each $x\in{\rm NUH}_\chi$. In Pesin charts, $f$ takes the form
$f_x:=\Psi_{f(x)}^{-1}\circ f\circ\Psi_x$. Unfortunately, we might not be able to see
hyperbolicity for $f_x$, but only for a restriction: while in the uniformly hyperbolic situation $C(x),C(x)^{-1}$
are uniformly bounded, now the parameters
$s,u,\alpha$ can degenerate
and so $\|C(f(x))^{-1}\|$ can be arbitrarily large, causing a big distortion. To decrease the domain
of definition of $f_x$, we multiply its current size by a large negative power of $\|C(f(x))^{-1}\|$.

\medskip
\noindent
{\sc Parameter $Q(x)$:} For $x\in{\rm NUH}_\chi$,
define $Q(x)= \ve^{3/\beta}\|C(f(x))^{-1}\|^{-12/\beta}_{\rm Frob}$.\\

\medskip
The choice of the powers $3/\beta$ and $12/\beta$ is not canonical
but just an artifact of the proof, and any choice of powers bigger than these would 
also make the proof work.
This more complicated definition of $Q(x)$ is the price we pay
to detect hyperbolicity among nonuniformly hyperbolic systems,
as stated in the theorem below. 

\begin{theorem}[Pesin]\label{Thm-non-linear-Pesin}
The following holds for all $\ve>0$ small. If $x\in{\rm NUH}_\chi$ then:
\begin{enumerate}[{\rm (1)}]
\item $d(f_x)_0=C(f(x))^{-1}\circ df_x\circ C(x)=\left[\begin{array}{cc}A & 0\\ 0 & B\end{array}\right]$
with $|A|,|B^{-1}|<e^{-\chi}$, cf. Lemma \ref{Lemma-linear-reduction-2}.
\item $f_x(v_1,v_2)=(Av_1+h_1(v_1,v_2),Bv_2+h_2(v_1,v_2))$, $(v_1,v_2)\in [-10Q(x),10Q(x)]^2$, where:
\begin{enumerate}[{\rm (a)}]
\item $h_1(0,0)=h_2(0,0)=0$ and $\nabla h_1(0,0)=\nabla h_2(0,0)=0$.
\item $\|h_1\|_{1+\beta/2}<\ve$ and $\|h_2\|_{1+\beta/2}<\ve$, with norms taken in
$[-10Q(x),10Q(x)]^2$.
\end{enumerate}
\end{enumerate}
A similar statement holds for $f_x^{-1}=\Psi_x^{-1}\circ f^{-1}\circ \Psi_{f(x)}$.
\end{theorem}

The above statement and proof below are similar to \cite[Thm 2.7]{Sarig-JAMS}, see also 
\cite[ Thm 5.6.1]{Barreira-Pesin-2007}.

\begin{proof}
We proceed as in the proof of Theorem \ref{Thm-Lyapunov-chart}. The main difficulty resides
on part (2)(b). We still have the estimate
$$
\|A_1 B_1 C_1-A_2 B_2 C_2\|\leq 24\mathfrak K \mathfs H\|w_1-w_2\|^\beta,
$$
but now
\begin{align*}
&\,\|d(f_x)_{w_1}-d(f_x)_{w_2}\|\leq \|C(f(x))^{-1}\| \|A_1 B_1 C_1-A_2 B_2 C_2\| \|C(x)\|\\
&\leq 24\mathfrak K\mathfs H \|C(f(x))^{-1}\|\|w_1-w_2\|^\beta.
\end{align*}
If $w_1,w_2\in[-10Q(x),10Q(x)]^2$ then $\|w_1-w_2\|<40Q(x)$, hence for $\ve>0$ small
$$
24\mathfrak K\mathfs H \|C(f(x))^{-1}\|\|w_1-w_2\|^{\beta/2}\leq 
200\mathfrak K\mathfs H\ve^{3/2}\|C(f(x))^{-1}\|^{-5}\leq 200\mathfrak K\mathfs H\ve^{3/2}<\ve.
$$
This completes the proof.
\end{proof}

Therefore, at a smaller scale that depends on the quality of hyperbolicity at $x$,
the map $f_x$ is again the perturbation of a hyperbolic matrix.

\subsection{Temperedness and the parameter $q(x)$}\label{Subsec-q}

After successfully detecting hyperbolicity for $f_x$, the next step is 
to define graph transforms.
As seen in Section \ref{Section-UH-Systems}, for uniformly hyperbolic systems the domains of all
Lyapunov charts have
the same size. Since forward images of $u$--admissible manifolds and backward images
of $s$--admissible manifolds grow essentially $\lambda^{-1}$, their images do cross
the successive domains from one side to the other, see Figure \ref{figure-graph-transform}.
By Theorem \ref{Thm-non-linear-Pesin}, for nonuniformly hyperbolic systems the
forward images of $u$--admissible manifolds and backward images of $s$--admissible manifolds grow essentially 
$e^{\chi}$. Therefore we face a problem when the ratio $\tfrac{Q(f(x))}{Q(x)}$ is far from 1:
\begin{enumerate}[$\circ$]
\item If $Q(f(x))\gg Q(x)$, then the image of a $u$--admissible manifold at $x$ does not
cross the domain of $\Psi_{f(x)}$ from top to bottom.
\item If $Q(f(x))\ll Q(x)$, then the image of an $s$--admissible manifold at $f(x)$ does not
cross the domain of $\Psi_x$ from left to right.
\end{enumerate}
The parameters $s(x),u(x),\alpha(x)$ and $s(f(x)),u(f(x)),\alpha(f(x))$ differ roughly
by the action of $df_x$, so there is a constant $\mathfs C=\mathfs C(f)>1$ s.t. 
$\mathfs C^{-1}\leq \tfrac{Q(f(x))}{Q(x)}\leq \mathfs C$ for all $x\in{\rm NUH}_\chi$.
Nevertheless, this control is yet not enough to rule out the above problems, since we can still have $\mathfs C\gg e^{\chi}$.
To solve this issue, we need to further reduce the domains of Pesin charts, introducing a parameter
that varies regularly. \\

\medskip
\noindent
{\sc Parameter $q(x)$:} For $x\in{\rm NUH}_\chi$, define $q(x)= \inf\{e^{\ve |n|}Q(f^n(x)):n\in\Z\}$.\\

\medskip
While $Q(x)$ (essentially) does not depend on $\ve$, the parameter
$q(x)$ does and, if positive, it does behave nicely along orbits:
$$
e^{-\ve}\leq \frac{q(f(x))}{q(x)}\leq e^{\ve}.
$$
The above definition is motivated by the proof of Lemma 1.1.1 in \cite{Pesin-Izvestia-1976}, and provides
the optimal value for $q(x)\leq Q(x)$ satisfying the above inequalities.
This is known as the {\em tempering kernel lemma}, see e.g. \cite[Lemma 6.11]{Barreira-Pesin-new-book}.
We remark that there are other proofs of the tempering kernel lemma, but
that do not provide optimal $q(x)$, see e.g. \cite[Lemma 3.5.7]{Barreira-Pesin-2007}.

\medskip
Since $q(x)\leq Q(x)$, the restriction of $f_x$ to the smaller domain
$[-q(x),q(x)]^2$ is a small
perturbation of a hyperbolic matrix. Now we are safe: restricting $\Psi_x$
to $[-q(x),q(x)]^2$, if $\ve>0$ is small enough then the growth of $u/s$--admissible manifolds
beats the possible increase/decrease of domains. Motivated by this,
we consider the subset of ${\rm NUH}_\chi$ where $q$ is positive.

\medskip
\noindent
{\sc The nonuniformly hyperbolic locus ${\rm NUH}^*_\chi$:}
$$
{\rm NUH}^*_\chi=\{x\in {\rm NUH}_\chi: q(x)>0\}.
$$

\medskip
By the next lemma, ${\rm NUH}^*_\chi$ carries the same finite invariant
measures as ${\rm NUH}_\chi$.

\begin{lemma}\label{Lemma-same-measures-1}
If $\mu$ is an $f$--invariant probability measure supported on ${\rm NUH}_\chi$, then 
$\mu$ is supported on ${\rm NUH}^*_\chi$.
\end{lemma} 

\begin{proof}
By assumption, $\mu[{\rm NUH}_\chi]=1$. Clearly, if $\lim_{n\to\pm\infty}\tfrac{1}{n}\log Q(f^n(x))=0$ 
then $q(x)>0$. We will prove that
$\lim_{n\to\pm\infty}\tfrac{1}{n}\log Q(f^n(x))=0$ for $\mu$--a.e. $x\in{\rm NUH}_\chi$.
Define the function $\varphi:{\rm NUH}_\chi\to\R$
by
$$
\varphi(x):=\log\left[\tfrac{Q(f(x))}{Q(x)}\right]=\log Q(f(x))-\log Q(x).
$$
Since $\mathfs C^{-1}\leq \tfrac{Q(f(x))}{Q(x)}\leq \mathfs C$ for $x\in{\rm NUH}_\chi$,
we have $\varphi\in L^1(\mu)$. Let $\varphi_n=\log(Q\circ f^n)-\log Q$ be
the $n$--th Birkhoff sum of $\varphi$.
By the Birkhoff ergodic theorem, 
$\lim_{n\to+\infty}\tfrac{\varphi_n(x)}{n}$ exists $\mu$--a.e.
Since by the Poincar\'e recurrence theorem we have
$\liminf_{n\to+\infty}|\varphi_n(x)|=\liminf_{n\to+\infty}|\log Q(f^n(x))-\log Q(x)|<\infty$ $\mu$--a.e,
it follows that $\lim_{n\to+\infty}\tfrac{\varphi_n(x)}{n}=0$ for $\mu$--a.e. $x\in{\rm NUH}_\chi$.
Proceeding in the same way for
$n\to-\infty$, we conclude that $\lim_{n\to\pm\infty}\tfrac{1}{n}\log Q(f^n(x))=0$ for $\mu$--a.e.
$x\in{\rm NUH}_\chi$.
\end{proof}



\subsection{Sizes of invariant manifolds: the parameters $q^s(x),q^u(x)$}\label{Subsec-q^s-q^u}

Using what we have done so far, we can proceed as in Section \ref{Section-UH-Systems}
to construct invariant manifolds:
define $s/u$--admissible manifolds as graphs of functions
$F:[-q(x),q(x)]\to \mathbb R$ satisfying some regularity assumptions (that we will explain later),
and define graph transforms $\mathfs F^s_x,\mathfs F^u_x$. Hence
Theorem \ref{Thm-UH-graph-transform} holds, so we can construct (local) stable and
unstable manifolds for every $x\in{\rm NUH}_\chi^*$. This is essentially what is done 
in Pesin theory, see e.g. \cite[Chapter 7]{Barreira-Pesin-2007}.

\medskip
In general, $q(x)$ is not the optimal size for the local invariant manifolds, and in some applications
we need bigger sizes for them. This is the case for the construction of countable Markov partitions that we will
discuss in Part \ref{Part-2}.
Observe that $q(x)$ might be small for two different reasons:
\begin{enumerate}[$\circ$]
\item There is $n>0$ for which $e^{\ve n}Q(f^n(x))$ is small.
\item There is $n>0$ for which $e^{\ve n}Q(f^{-n}(x))$ is small. 
\end{enumerate}
In the first case, the forward behavior of $Q(f^n(x))$ is bad,
so we expect to construct a small stable manifold; but we are also constructing a small unstable manifold,
i.e. the bad forward behaviour is influencing the size of the unstable manifold!
Since the unstable manifold only depends on the past, its size should not be affected
by the future. To deal with this, we introduce two new parameters
$q^s(x)$ and $q^u(x)$, the first controlling the future behavior and the second controlling the past behavior.
Then we use them to construct invariant manifolds with larger sizes.

\medskip
\noindent
{\sc Parameters $q^s(x)$ and $q^u(x)$:} For $x\in{\rm NUH}_\chi^*$, define
\begin{align*}
q^s(x)&= \inf\{e^{\ve n}Q(f^n(x)):n\geq 0\}\\
q^u(x)&= \inf\{e^{\ve n}Q(f^{-n}(x)):n\geq 0\}.
\end{align*}

\medskip
In other words, $q^s(x),q^u(x)$ are the one-sided versions
of $q(x)$. Just like $q$, the parameters $q^s,q^u$ depend on $\ve$.
We will use $q^s(x)$ as the scale for considering the stable graph transform and 
$q^u(x)$ as the scale for considering the unstable graph transform.

\begin{lemma}\label{q^s}
For all $x\in{\rm NUH}^*_\chi$, the following holds:
\begin{enumerate}[{\rm (1)}]
\item {\sc Good definition:} $q^s(x),q^u(x)>0$ and $q(x)=\min\{q^s(x),q^u(x)\}$.
\item {\sc Greedy algorithm:}
\begin{align*} 
q^s(x)&=\min\{e^\ve q^s(f(x)),Q(x)\}\\
q^u(x)&=\min\{e^\ve q^u(f^{-1}(x)),Q(x)\}.
\end{align*}
\end{enumerate}
\end{lemma}

\medskip
The proofs are direct, see also \cite[Lemma 4.2]{Lima-Matheus}.

\subsection{Graph transforms: construction of invariant manifolds}\label{Subsec-NUH-graph}

There are dynamical explanations for Lemma \ref{q^s}(2). Let us discuss the first equality.
Assume that $s$--admissible manifolds at $x$ have representing functions defined in the interval $[-q^s(x),q^s(x)]$.
If $\ve>0$ is small enough, then the stable graph transform $\mathfs F^s_x$ takes the graph of
a representing function defined in $[-q^s(f(x)),q^s(f(x))]$
and expands it at least by a factor of $e^\ve$, so the new representing function is well-defined
in  $[-e^\ve q^s(f(x)),e^\ve q^s(f(x))]$. Since its domain of definition should not go
beyond $[-Q(x),Q(x)]$ (where we have a good control on $f_x$),
the best we can do is to define it in $[-q^s(x),q^s(x)]$.
In summary, $q^s$ provides maximal scales for the definition of stable graph transforms.
Similarly, $q^u$ provides maximal scales for the definition of unstable graph transforms.
With this in mind, we give a new definition of $s/u$--admissible manifolds.

\medskip
\noindent
{\sc Admissible manifolds:} An {\em $s$--admissible manifold} at $\Psi_x$ is a set
of the form $V^s=\Psi_x\{(t,F(t)):|t|\leq q^s(x)\}$, where $F:[-q^s(x),q^s(x)]\to\R$
is a $C^{1+\beta/3}$ function s.t. $F(0)=F'(0)=0$ and 
$\|F'\|_0+\Hol{\beta/3}(F')\leq\tfrac{1}{2}$, where the norms are taken in $[-q^s(x),q^s(x)]$.
Similarly, a {\em $u$--admissible manifold} at
$\Psi_x$ is a set of the form $V^u=\Psi_x\{(G(t),t):|t|\leq q^u(x)\}$, where
$G:[-q^u(x),q^u(x)]\to\R$ is a $C^{1+\beta/3}$ function s.t. $G(0)=G'(0)=0$ and 
$\|G'\|_0+\Hol{\beta/3}(G')\leq\tfrac{1}{2}$, with norms taken in $[-q^u(x),q^u(x)]$.

\medskip
As before, $F,G$ are called the {\em representing functions} of $V^s,V^u$ respectively,
and let $\mathfs M^s_x,\mathfs M^u_x$ be the space of all $s,u$--admissible manifolds at $\Psi_x$
respectively, which are metric spaces with the $C^0$ distance. Let $x\in{\rm NUH}_\chi^*$.

\medskip
\noindent
{\sc Graph transforms $\mathfs F^s_x,\mathfs F^u_x$:} 
The {\em stable graph transform} $\mathfs F^s_x:\mathfs M^s_{f(x)}\to\mathfs M^s_x$
is the map that sends $V^s\in \mathfs M^s_{f(x)}$ to the unique $\mathfs F^s_x[V^s]\in \mathfs M^s_x$
with representing function $F$ s.t. $\Psi_x\{(t,F(t)):|t|\leq q^s(x)\}\subset f^{-1}(V^s)$.
Similarly, the {\em unstable graph transform}
$\mathfs F^u_x:\mathfs M^u_x\to\mathfs M^u_{f(x)}$ is the map that sends
$V^u\in \mathfs M^u_x$ to the unique $\mathfs F^u_x[V^u]\in \mathfs M^u_{f(x)}$
with representing function $G$ s.t. $\Psi_{f(x)}\{(G(t),t):|t|\leq q^u(f(x))\}\subset f(V^u)$.

\medskip
The difference from the previous definition is that the stable and unstable graph transforms
are defined at different scales, see Figure \ref{figure-graph-transform-new} below.
\begin{figure}[hbt!]
\centering
\def\svgwidth{13cm}
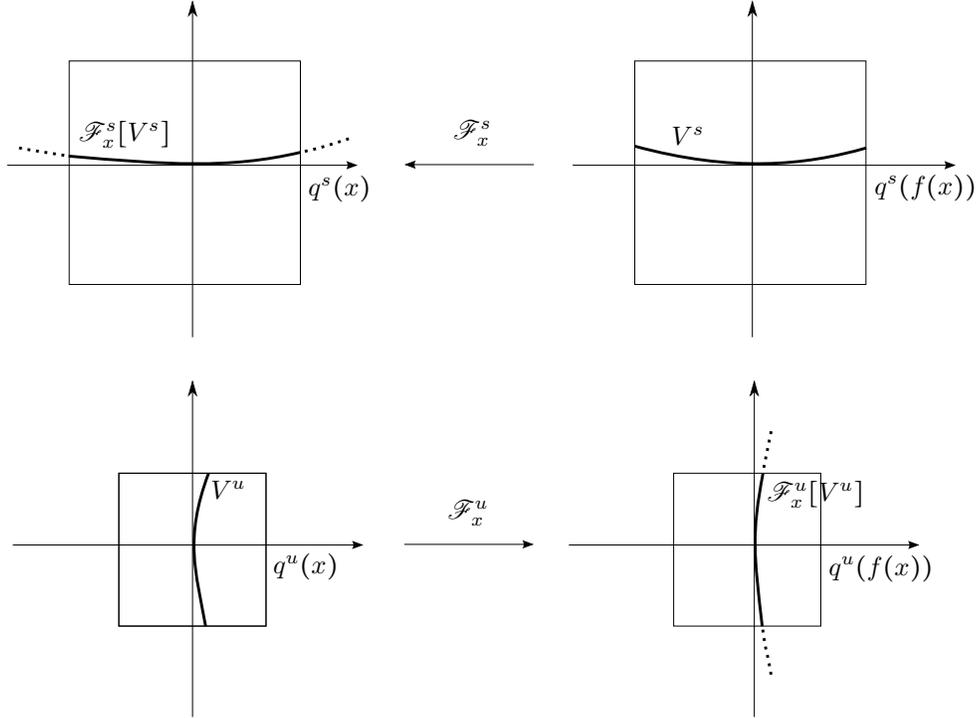
\caption{The stable graph transforms are defined at scales $q^s$, and 
the unstable graph transforms at scales $q^u$.}\label{figure-graph-transform-new}
\end{figure}

\begin{theorem}\label{Thm-NUH-graph-transform}
$\mathfs F^s_x$ and $\mathfs F^u_x$ are well-defined contractions.
\end{theorem}

For nonuniformly hyperbolic systems, this theorem was first proved by Pesin,
see \cite[Thm. 2.2.1]{Pesin-Izvestia-1976}.
The proof is similar to the proof of Theorem \ref{Thm-UH-graph-transform}.
In its present form, with scales $q^s$ and $q^u$,
the above result is a special case of \cite[Prop. 4.12]{Sarig-JAMS}.
For $x\in{\rm NUH}_\chi^*$, let $V^s[x]$ and $V^u[x]$ be the stable and unstable manifolds of $x$,
defined as in Subsection \ref{Subsect-UH-graph}. Then $V^s[x]$ is the image under $\Psi_x$
of the graph of a function defined in $[-q^s(x),q^s(x)]$, while $V^u[x]$ is the image under
$\Psi_x$ of the graph of a function defined in  $[-q^u(x),q^u(x)]$.

\subsection{Higher dimensions}

Now consider diffeomorphisms in any dimension.
The discussion follows \cite{Ben-Ovadia-2019} and in some sense \cite{Ben-Ovadia-codable}.
We can no longer perform the construction using only the parameters $s(x),u(x),\alpha(x)$,
because now the spaces $E^s,E^u$ are higher dimensional, and each vector defines
its own parameter. More specifically, consider the following definition, for each fixed $\chi>0$.

\medskip
\noindent
{\sc The nonuniformly hyperbolic locus ${\rm NUH}_\chi$:} It is the set of points \label{Def-NUH-highdim}
$x\in M$ for which there is a splitting $T_xM=E^s_x\oplus E^u_x$ s.t.:
\begin{enumerate}[(NUH1)]
\item Every $v\in E^s_x$ contracts in the future at least $-\chi$ and expands in the past:
$$\limsup_{n\to+\infty}\tfrac{1}{n}\log \|df^n v\|\leq -\chi\ \text{ and } 
\ \liminf_{n\to+\infty}\tfrac{1}{n}\log \|df^{-n} v\|>0.
$$ 
\item Every $v\in E^u_x$ contracts in the past at least $-\chi$ and expands in the future:
$$\limsup_{n\to+\infty}\tfrac{1}{n}\log \|df^{-n} v\|\leq -\chi\text{ and } 
\liminf_{n\to+\infty}\tfrac{1}{n}\log \|df^n v\|>0.$$
\item The parameters $s(x)=\sup_{v\in E^s_x\atop{\|v\|=1}}S(x,v)$ and
$u(x)=\sup_{w\in E^u_x\atop{\|w\|=1}}U(x,w)$ are finite, where:
\begin{align*}
S(x,v)&=\sqrt{2}\left(\sum_{n\geq 0}e^{2n\chi}\|df^n v\|^2\right)^{1/2},\\
U(x,w)&=\sqrt{2}\left(\sum_{n\geq 0}e^{2n\chi}\|df^{-n}w\|^2\right)^{1/2}.
\end{align*} 
\end{enumerate}
In \cite{Ben-Ovadia-codable}, this definition is similar to the definition
of the set $\chi$--summ. Again, 
${\rm NUH}_\chi$ is $f$--invariant, and for each $x\in {\rm NUH}_\chi$ we can
define a linear transformation $C(x):\R^n\to T_xM$ that sends the canonical
metric on $\R^n$ to the Lyapunov inner product $\llangle\cdot,\cdot\rrangle$ on $T_xM$.
We again have the block representation
$$
C(f(x))^{-1}\circ df_x\circ C(x)=
\left[
\begin{array}{cc}
D_s & 0 \\
0 & D_u 
\end{array}
\right],
$$
where $D_s$ is a $d_s\times d_s$ matrix s.t. $\|D_s v\|\leq e^{-\chi}\|v\|$ for all $v\in \R^{d_s}$
and $D_u$ is a $d_u\times d_u$ matrix s.t. $\|D_u^{-1} w\|\leq e^{-\chi}\|w\|$ for all $w\in \R^{d_u}$.
This is the higher dimensional Oseledets-Pesin reduction, see Lemma \ref{Lemma-linear-reduction-2}(2).
Define the Pesin chart $\Psi_x$ as in Subsection \ref{Subsec-NUH-charts},
and the parameter $Q(x)=\mathfs H\|C(f(x))^{-1}\|^{-48/\beta}$, 
where $\mathfs H=\mathfs H(\beta,\ve)$ is a constant that allows to keep the estimates
of order $\ve$ and to absorb multiplicative constants.
Then a higher dimensional version of Theorem \ref{Thm-non-linear-Pesin} holds,
see \cite[Thm 1.13]{Ben-Ovadia-2019}. From now on, we can repeat the two-dimensional
construction, defining the parameters $q(x),q^s(x),q^u(x)$, the nonuniformly hyperbolic locus
${\rm NUH}_\chi^*$, an $s$--admissible manifold at $\Psi_x$ as a set
of the form $V^s=\Psi_x\{(t,F(t)):t\in [-q^s(x),q^s(x)]^{d_s}\}$, where $F:[-q^s(x),q^s(x)]^{d_s}\to\R^{d_u}$
is a $C^{1+\beta/3}$ function s.t. $F(0)=F'(0)=0$ and 
$\|F'\|_0+\Hol{\beta/3}(F')\leq\tfrac{1}{2}$, where the norms are taken in $[-q^s(x),q^s(x)]^{d_s}$, and
similarly $u$--admissible manifolds. Then Theorem \ref{Thm-UH-graph-transform} holds, see
\cite[Prop. 2.8]{Ben-Ovadia-2019}, and so every $x\in{\rm NUH}_\chi^*$ has local stable and
unstable manifolds.

\section{Maps with discontinuities and bounded derivative}\label{Section-NUH-dis1}

In the previous section, we considered diffeomorphisms defined on closed (compact without boundary)
surfaces. There are natural examples that do not fit into this context, for example
Poincar\'e return maps of flows and billiard maps. Their common feature is the presence
of discontinuities, and the possible explosion of derivatives.
In the next two sections, we will discuss how to adapt the methods of Section \ref{Section-NUH-systems}
to cover these examples, focusing on the changes that are needed to make the 
arguments work. We start dealing with surface maps with discontinuities and bounded derivative.
The reference is \cite{Lima-Sarig}.

\subsection{Definitions and examples}\label{Subsec-NUH-dis1-definitions}

Let $M$ be a compact surface, possibly with boundary.
To avoid multiplicative constants in the calculations,
we assume that $M$ has diameter smaller than one.
Let $\mathfs D^+,\mathfs D^-$ be closed subsets of $M$, and consider a map
$f:M\backslash\mathfs D^+\to M$ with inverse $f^{-1}:M\backslash\mathfs D^-\to M$.
Let $\mathfs D:=\mathfs D^+\cup \mathfs D^-$ be the {\em set of discontinuities of $f$}.
We require $f,f^{-1}$ to be local $C^{1+\beta}$ diffeomorphisms.

\medskip
\noindent
{\sc Regularity of $f$:} There is a constant $\mathfs L>0$ with the following property:
\begin{enumerate}[$\circ$]
\item For every $x\in M\backslash\mathfs D^+$
there is an open set $U\ni x$ s.t. $f\restriction_U$ is a diffeomorphism onto its image with
$C^{1+\beta}$ norm at most $\mathfs L$.
\item For every $x\in M\backslash\mathfs D^-$
there is an open set $V\ni x$ s.t. $f^{-1}\restriction_V$ is a diffeomorphism onto its image with
$C^{1+\beta}$ norm at most $\mathfs L$.
\end{enumerate}

\medskip
In particular, $\|df^{\pm 1}\|$ is bounded away from zero and infinity, so 
the integrability condition in the Oseledets theorem holds for any $f$--invariant probability 
measure.
The main difficulty when dealing with $f$ as above is that, as $x$ approaches $\mathfs D$, 
the open sets $U,V$ become smaller, hence the domains of Pesin charts also need to be smaller.
To avoid this issue, we only consider trajectories that do not approach $\mathfs D$ exponentially fast.

\medskip
Here is the example to have in mind.
Let $N$ be a three dimensional closed Riemannian manifold, let $X$ be a $C^{1+\beta}$
vector field on $N$ s.t. $X(p)\neq 0$ for all $p\in N$, and let $\varphi=\{\varphi^t\}_{t\in\R}$
be the flow generated by $X$. We can reduce the dynamics of $\varphi$
to the dynamics of a surface map by constructing a global Poincar\'e section
$M$ for $\varphi$ as follows:
\begin{enumerate}[$\circ$]
\item Fix $\ve>0$ small enough.
\item For each $p\in N$, consider a closed differentiable disc $D(p)$ centered at $p$ with
diameter smaller than $\ve$ s.t. $\angle (T_qD(p),X(q))>\tfrac{\pi}{2}-\ve$ for all $q\in D(p)$.
\item Let ${\rm FB}(p):=\bigcup_{|t|\leq\ve}\varphi^t[D(p)]$ be the {\em flow box} defined by $D(p)$.
Using that $X\neq 0$, we see that ${\rm FB}(p)$ contains an open ball centered at $p$.
\item By compactness, $N$ is covered by finitely many flow boxes
${\rm FB}(p_1),\ldots,{\rm FB}(p_\ell)$.
\end{enumerate} 
Therefore $M=D(p_1)\cup \cdots \cup D(p_\ell)$ is a global Poincar\'e section for $\varphi$. 
With some extra work, we can make the discs $D(p_1),\ldots,D(p_\ell)$ to be pairwise disjoint,
hence the return time function $\mathfrak t:M\to(0,\infty)$ is bounded away from zero and infinity.
See \cite[Section 2]{Lima-Sarig} for details.

\medskip
Let $f:M\to M$ be
the Poincar\'e return map of $M$, i.e. $f(x)=\vf^{\mathfrak t(x)}(x)$. The map $f$ has discontinuities, with
$\mathfs D^{\pm}=\{x\in M:f^{\pm 1}(x)\in\partial M\}$.\footnote{Observe that in this case
$f$ is defined on all of $M$,
but $f^{\pm 1}$ is discontinuous on $\mathfs D^{\pm}$.} See Figure \ref{figure-return-map}.
\begin{figure}[hbt!]
\centering
\def\svgwidth{9cm}
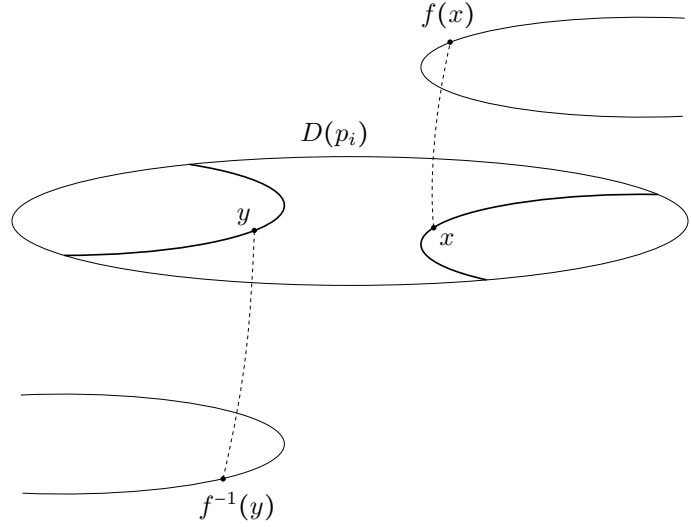
\caption{Discontinuities for $f$: in the picture, $x\in\mathfs D^+$ and $y\in\mathfs D^-$.}\label{figure-return-map}
\end{figure}
Nevertheless, where $f^{\pm 1}$ is continuous, its $C^{1+\beta}$ norm is uniformly bounded.
This occurs because, at continuity points, $f^{\pm 1}$ has the form $\varphi^{\tau}$ where $\tau$
has uniformly bounded $C^{1+\beta}$ norm,
see \cite[Lemma 2.5]{Lima-Sarig} for details.

\subsection{Nonuniform hyperbolicity}\label{Subsec-NUH-dis1-locus}

To apply the methods of
Section \ref{Section-NUH-systems}, we only consider trajectories
that do not approach $\mathfs D$ exponentially fast. Let $d$ be the distance in $M$.

\medskip
\noindent
{\sc The nonuniformly hyperbolic locus ${\rm NUH}^*_\chi$:} It is the set of points
$x\in M$ satisfying conditions (NUH1)--(NUH3) of page \pageref{Def-NUH} and the additional condition:
\begin{enumerate}[(NUH4)]
\item {\sc Subexponential convergence to $\mathfs D$:}
$$
\lim_{n\to\pm\infty}\tfrac{1}{n}\log d(f^n(x),\mathfs D)=0.
$$
\end{enumerate}

\medskip
The idea of looking at trajectories satisfying condition (NUH4) is not new. It goes back
to Sina{\u\i} in the context of billiards \cite{Sinai-billiards}, which we will discuss in Section \ref{Section-NUH-dis2}.
See also the section ``Overcoming influence of singularities'' in \cite{Katok-Strelcyn}.
At the level of invariant measures, (NUH4) is related to the following notion. 

\medskip
\noindent
{\sc $f$--adapted measure:} An $f$--invariant measure on $M$ is called {\em $f$--adapted}
if the function $\log d(x,\mathfs D)\in L^1(\mu)$. A fortiori $\mu(\mathfs D)=0$.

\medskip
By the Birkhoff ergodic theorem, if $\mu$ is $f$--adapted then (NUH4) holds $\mu$--a.e.
If in addition $\mu$ is $\chi$--hyperbolic, then (NUH1)--(NUH3) hold $\mu$--a.e. and therefore
 $\mu$ is carried by ${\rm NUH}_\chi$, i.e. $\mu[{\rm NUH}_\chi]=1$.

\medskip
For each $x\in{\rm NUH}_\chi$, the linear map $C(x)$ can be defined as before, and
Lemma \ref{Lemma-linear-reduction-2} remains valid with the same proof. 
To define the Pesin chart $\Psi_x$, we just need to adjust its domain of definition,
according to the distance of $x$ to $\mathfs D$.
Let $\delta(x)=\ve^{3/\beta}d(x,\mathfs D)$.

\medskip
\noindent
{\sc Pesin chart:}  The {\em Pesin chart} at $x\in{\rm NUH}_\chi$ is the map
$\Psi_x:[-\delta(x),\delta(x)]^2\to M$ defined by $\Psi_x:=\exp{x}\circ C(x)$.

\medskip
We also redefine $Q(x)$ accordingly. 
Let $\rho(x):=d(\{f^{-1}(x),x,f(x)\},\mathfs D)$.


\medskip
\noindent
{\sc Parameter $Q(x)$:} For $x\in{\rm NUH}_\chi$,
let $Q(x)= \ve^{3/\beta}\min\{\|C(f(x))^{-1}\|^{-12/\beta}_{\rm Frob},\ve\rho(x)\}$.\\

With this definition\footnote{We take the chance to observe that the definition of 
$Q(x)$ in \cite{Lima-Sarig} has a small error, since it does not depend
on $d(f(x),\mathfs D)$ and so we cannot guarantee that $f_x$ is well-defined, see
\cite[Thm 3.2 and Corollary 3.6]{Lima-Sarig}. Nevertheless,
this can be easily fixed with the definition we give here.}\label{Footnote-correct-Q},
the representation of $f$ in Pesin charts  $f_x:=\Psi_{f(x)}^{-1}\circ f\circ\Psi_x$
is well-defined in $[-10Q(x),10Q(x)]^2$. Indeed, $Q(x)\leq \ve\delta(f(x))$ and so
$$
(f\circ\Psi_x)\left([-10Q(x),10Q(x)]^2\right)\subset \Psi_{f(x)}\left([-\delta(f(x)),\delta(f(x))]^2\right).
$$
Again, in the domain $[-10Q(x),10Q(x)]^2$ the map $f_x$ is a small perturbation
of a hyperbolic matrix.

\medskip
Now define the parameters $q,q^s,q^u$, the set ${\rm NUH}_\chi^*$,
and the graph transforms $\mathfs F^{s/u}_x$ as in the previous section, then
construct local invariant manifolds for each $x\in{\rm NUH}_\chi^*$. We finish
this section proving an analogue of Lemma \ref{Lemma-same-measures-1}.

\begin{lemma}\label{Lemma-same-measures-2}
If $\mu$ is an $f$--invariant probability measure supported on ${\rm NUH}_\chi$, then 
$\mu$ is supported on ${\rm NUH}^*_\chi$.
\end{lemma}

\begin{proof}
By assumption, $\lim_{n\to\pm\infty}\tfrac{1}{n}\log d(f^n(x),\mathfs D)=0$ for $\mu$--a.e.
$x\in M$.
Let $\widetilde Q(x)= \ve^{3/\beta}\|C(f(x))^{-1}\|^{-12/\beta}_{\rm Frob}$ be the ``old'' $Q$.
Since $df$ is uniformly bounded, we can proceed exactly as in Lemma \ref{Lemma-same-measures-1}
to conclude that $\lim_{n\to\pm\infty}\tfrac{1}{n}\log{\widetilde Q}(f^n(x))=0$ for $\mu$--a.e. $x\in M$.
But then $\lim_{n\to\pm\infty}\tfrac{1}{n}\log Q(f^n(x))=0$ for $\mu$--a.e. $x\in M$.
\end{proof}

\section{Maps with discontinuities and unbounded derivative}\label{Section-NUH-dis2}

Next, we consider surface maps with discontinuities and unbounded derivative.
Added to the difficulty that Pesin charts are defined in smaller domains, 
now $\|df^{\pm 1}\|$ can approach zero and infinity, so even the integrability condition
in the Oseledets theorem is no longer automatic.
The first development of Pesin theory in this context was \cite{Katok-Strelcyn}, whose 
interest was to apply it for billiard maps. 
Since its beginning, ergodic theory of billiard maps was mainly focused in a reference
Liouville measure. This is the case in \cite{Katok-Strelcyn}, where the authors construct invariant manifolds
Lebesgue almost everywhere and use them to prove ergodic theoretic properties such as 
the Ruelle inequality. Contrary to this, in the sequel we follow the same approach of Section
\ref{Section-NUH-systems}, not focusing on a particular measure
but rather on the set of points with some hyperbolicity.
The reference for this section is \cite{Lima-Matheus}.

\subsection{Definitions and examples}\label{Subsec-NUH-dis2-definitions}

Let $M$ be a compact surface, possibly with boundary. Again, 
we assume that $M$ has diameter smaller than one.
Let $\mathfs D^+,\mathfs D^-$ be closed subsets of $M$, and consider 
$f:M\backslash\mathfs D^+\to M$ with inverse $f^{-1}:M\backslash\mathfs D^-\to M$.
Let $\mathfs D:=\mathfs D^+\cup \mathfs D^-$ be the {\em set of discontinuities of $f$}.
We require the following conditions on $f$.

\medskip
\noindent
{\sc Regularity of $f$:} There are constants $0<\beta<1<a$ and $\mathfrak K>0$ s.t. for all
$x\in M\backslash\mathfs D$ there is $d(x,\mathfs D)^a<\mathfrak r(x)<d(x,\mathfs D)$ s.t.
if $D_x=B(x,\mathfrak r(x))$ then the following holds:
\begin{enumerate}[$\circ$]
\item If $y\in D_x$ then $\|df^{\pm 1}_y\|\leq d(x,\mathfs D)^{-a}$.
\item If $y_1,y_2\in D_x$ and $f(y_1),f(y_2)\in D_{x'}$ then
$\|\widetilde{df_{y_1}}-\widetilde{df_{y_2}}\|\leq \mathfrak Kd(y_1,y_2)^\beta$,
and if $y_1,y_2\in D_x$ and $f^{-1}(y_1),f^{-1}(y_2)\in D_{x''}$ then
$\|\widetilde{df_{y_1}^{-1}}-\widetilde{df_{y_2}^{-1}}\|\leq \mathfrak Kd(y_1,y_2)^\beta$.
\end{enumerate}

\medskip
The first assumption says that $df^{\pm 1}$ blows up at most polynomially fast, and the second 
says that $df^{\pm 1}$ is locally $\beta$--H\"older.
The examples to have in mind are billiard maps, as we now explain.
Given a compact domain $T\subset\R^2$ or $T\subset \mathbb T^2$ with piecewise $C^3$ boundary,
consider the straight line motion of a particle inside $T$, with specular reflections
in $\partial T$. The phase space of configurations is $M=\partial T\times[-\tfrac{\pi}{2},\tfrac{\pi}{2}]$
with the convention that $(r,\theta)\in M$ represents $r=$ collision position at $\partial T$ and $\theta=$
angle of collision. Given $(r,\theta)\in M$, let $(r^+,\theta^+)$ be the next collision and $(r^-,\theta^-)$ be
the previous collision. Let $\{r_1,\ldots,r_k\}$ be the break points of $\partial T$, and define:
\begin{align*}
\mathfs D^+&=\{r^+=r_i \text{ for some }i\}\cup \left\{\theta^+=\pm\tfrac{\pi}{2}\right\}\\
\mathfs D^-&=\{r^-=r_i \text{ for some }i\}\cup \left\{\theta^-=\pm\tfrac{\pi}{2}\right\}.
\end{align*}
The {\em billiard map} is $f:M\backslash\mathfs D^+\to M$ defined by $f(r,\theta)=(r^+,\theta^+)$,
with inverse $f:M\backslash\mathfs D^-\to M$ defined by $f(r,\theta)=(r^-,\theta^-)$.
Since $\partial T$ (usually) has two normal vectors at $r_i$, we cannot define $f^{\pm 1}(r,\theta)$ if $r^{\pm}=r_i$.
When $\theta^{\pm}=\pm\tfrac{\pi}{2}$, the trajectory has a grazing collision, and
$f^{\pm 1}$ is usually discontinuous on $(r,\theta)$. Furthermore, $df^{\pm 1}$ becomes arbitrarily large
in a neighborhood of $(r,\theta)$. This justifies the choice of $\mathfs D^{\pm}$ above.
See \cite{Chernov-Markarian} for details.

\medskip
Sina{\u\i} showed that $f$ has a natural invariant Liouville measure $\mu_{\rm SRB}=\cos\theta drd\theta$,
which is ergodic for dispersing billiards \cite{Sinai-billiards}.
Bunimovich constructed examples of ergodic nowhere dispersing
billiards \cite{Bunimovich-close-to-scattering,Bunimovich-ergodic-properties,Bunimovich-Nowhere-dispersing}.
These billiards, known as {\em Bunimovich billiards}, are nonuniformly hyperbolic.
See some examples in Figure \ref{figure-billiards}.
\begin{figure}[hbt!]
\centering
\def\svgwidth{12.5cm}
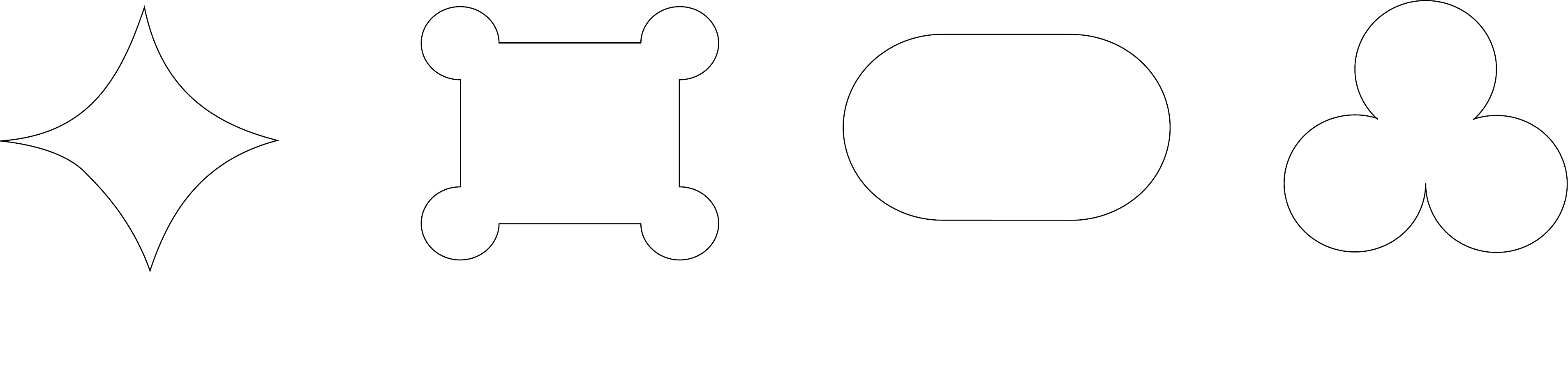
\caption{(1) is a Sina{\u\i} billiard table. The others are Bunimovich billiard tables:
(2) is the pool table with pockets, (3) is the stadium, (4) is the flower.}
\label{figure-billiards}
\end{figure}
Recently, Baladi and Demers constructed measures of maximal entropy
for some finite horizon periodic Lorentz gases \cite{Baladi-Demers}, see more on Section \ref{Section-periodic}.

\subsection{Nonuniform hyperbolicity}\label{Subsec-NUH-dis2-locus}

We continue only considering trajectories
that do not approach $\mathfs D$ exponentially fast, and define
the nonuniformly hyperbolic locus ${\rm NUH}_\chi$ as in 
Subsection \ref{Subsec-NUH-dis1-locus}. Similarly, an $f$--invariant measure
$\mu$ is called {\em $f$--adapted} if $\log d(x,\mathfs D)\in L^1(\mu)$.

\medskip
Since $\log \|df^{\pm 1}\|$ is usually unbounded, some measures might not satisfy
the integrability condition in the Oseledets theorem.
Due to the regularity of $f$, the functions $\log \|df^{\pm 1}\|$ and
$\log d(x,\mathfs D)$ are comparable, therefore
$\log \|df^{\pm 1}\|\in L^1(\mu)$ iff $\log d(x,\mathfs D)\in L^1(\mu)$. Hence the
Oseledets theorem holds for $f$--adapted measure, which shows that $f$--adaptability
is a natural assumption. In particular, if $\mu$ is $f$--adapted and $\chi$--hyperbolic,
then (NUH1)--(NUH4) hold $\mu$--a.e. and so $\mu$ is carried by ${\rm NUH}_\chi$.
For each $x\in{\rm NUH}_\chi$, we define $C(x)$ as before,
and Lemma \ref{Lemma-linear-reduction-2} remains valid with the same proof. 

\medskip
\noindent
{\sc Pesin chart:}  The {\em Pesin chart} at $x\in{\rm NUH}_\chi$ is
$\Psi_x:[-d(x,\mathfs D)^a,d(x,\mathfs D)^a]^2\to M$, defined by $\Psi_x:=\exp{x}\circ C(x)$.

\medskip
The definition of $Q(x)$ is more complicated. Let
$\rho(x)=d(\{f^{-1}(x),x,f(x)\},\mathfs D)$.

\medskip
\noindent
{\sc Parameter $Q(x)$:} \label{Def-dis2-Q}
For $x\in{\rm NUH}_\chi$, define
$$
Q(x)= \ve^{3/\beta}\min\left\{\|C(x)^{-1}\|_{\rm Frob}^{-24/\beta},\|C(f(x))^{-1}\|^{-12/\beta}_{\rm Frob}\rho(x)^{72a/\beta}\right\}.
$$

\medskip
As before, the choice of the powers is not canonical but just an artifact of the proof.
The above definition depends on $f^{-1}(x),x,f(x)$, and is strong enough
to construct local invariant manifolds, and to run the methods of Part \ref{Part-2}.
Firstly, in the domain $[-10Q(x),10Q(x)]^2$ the representation
of $f$ in Pesin charts $f_x:=\Psi_{f(x)}^{-1}\circ f\circ\Psi_x$ is a small perturbation
of a hyperbolic matrix, see \cite[Thm 3.3]{Lima-Matheus}. Defining the parameters $q,q^s,q^u$, the set ${\rm NUH}_\chi^*$,
and the graph transforms $\mathfs F^{s/u}_x$ as before,
we construct local invariant manifolds for each $x\in{\rm NUH}_\chi^*$. Finally, 
we establish an analogue of Lemmas \ref{Lemma-same-measures-1} and \ref{Lemma-same-measures-2}.

\begin{lemma}\label{Lemma-same-measures-3}
If $\mu$ is an $f$--adapted probability measure supported on ${\rm NUH}_\chi$, then 
$\mu$ is supported on ${\rm NUH}^*_\chi$.
\end{lemma}

\begin{proof}
It is enough to show that
$$
\lim_{n\to\pm\infty} \tfrac{1}{n}\log\|C(f^n(x))^{-1}\|_{\rm Frob}=
\lim_{n\to\pm\infty}\tfrac{1}{n} \log\rho(f^n(x))=0
$$
for $\mu$--a.e. $x\in M$.
By (NUH4), the second equality holds. For the first equality, 
let $\widetilde Q(x)= \ve^{3/\beta}\|C(f(x))^{-1}\|^{-12/\beta}_{\rm Frob}$ be the ``old'' $Q$,
and observe that the regularity assumption on $f$ implies that $\log \left[\tfrac{\widetilde Q\circ f}{\widetilde Q}\right]\in L^1(\mu)$
iff $\log d(x,\mathfs D)\in L^1(\mu)$.
Since $\mu$ is $f$--adapted, we get that $\log \left[\tfrac{\widetilde Q\circ f}{\widetilde Q}\right]\in L^1(\mu)$.
Now proceed as in Lemma \ref{Lemma-same-measures-1} to conclude that
$\lim_{n\to\pm\infty} \tfrac{1}{n}\log\|C(f^n(x))^{-1}\|_{\rm Frob}=0$ for $\mu$--a.e. $x\in M$.
\end{proof}

\part{Symbolic dynamics}\label{Part-2}

Symbolic dynamics is an important tool for the understanding of ergodic and statistical
properties of dynamical systems, both smooth and non-smooth.
The field of symbolic dynamics is enormous and covers various contexts, from the study of symbolic spaces
to the theory of complexity functions, see e.g. \cite{Kitchens-Book,Ferenczi-survey}.
Here, we only discuss the use of symbolic dynamics to
represent smooth dynamical systems. The main idea is simple,
and can be summarised in two steps:
firstly, divide the phase space of a system into finitely or countably many pieces, which we call {\em rectangles};
secondly, instead of describing the trajectory of a point by the exact positions in the phase space,
just record the sequence of rectangles that the trajectory visits.
We call the second step above a {\em coding}. This procedure can be made in a wide setting.
For instance, any partition defines a coding in the usual way. Such
flexibility allows its use in various contexts:
\begin{enumerate}[$\circ$]
\item Periodic points of continuous intervals maps: proof of the Sharkovsky theorem using
{\em Markov graphs}, see e.g. \cite{Burns-Hasselblatt}.
\item Milnor-Thurston's kneading theory of continuous intervals maps: description of the trajectory
of the critical point with respect to monotonicity intervals \cite{Milnor-Thurston}.
\item Geodesics on surfaces of constant negative curvature: Hadamard represented closed
geodesics using sequences of symbols \cite{Hadamard-1898},
see also \cite{Katok-Ugarcovici-Symbolic}.
\end{enumerate}
In this survey, we focus on symbolic dynamics for smooth systems with some hyperbolicity, uniform
and nonuniform. The final goal is to describe the invariant measures and ergodic theoretical
properties of such systems, and for that a mere coding is not enough: it is important
to recover codings, i.e. to know which trajectories are coded in the same way.
This reverse procedure is called {\em decoding}. Good codings are those 
for which we can satisfactorily decode. It has long been observed that uniform
expansion provides a good decoding: two different trajectories eventually stay
far apart and therefore cannot visit the same rectangles. This property, that different trajectories
eventually stay far apart, is known as {\em expansivity}. It also occurs for (U)-systems,
due to the exponential dichotomy of solutions mentioned in Section \ref{Section-UH-Systems}.

\medskip
Another required property on the coding is that the space of sequences coding the
trajectories should be as simple as possible and at the same time rich enough to reflect the structure
of the original smooth system. This property is certainly satisfied
if the rectangles have the {\em Markov property}, since in this case
every path on the graph is naturally associated to a genuine orbit.
Apart from some technical assumptions, when this occurs we call the partition a {\em Markov partition}.

\medskip
Let us give a simple example of its usefulness. Let $f:K\to K$ be the horseshoe
map described in Example 2 of Subsection \ref{Subsec-UH-definitions} (Smale's horseshoe).
Since $K$ has a fractal
structure, it seems rather complicated to understand its periodic points and
invariant measures. But the system has a Markov partition that induces a continuous bijection
$\pi:\Sigma\to K$ between the symbolic space $\Sigma=\{0,1\}^\Z$ and $K$, that commutes the shift map
on $\Sigma$ and the map $f$. Hence we can analyse the dynamical properties of $f$ by means of the
dynamical properties of the shift map.

\medskip
In the next sections, we explain how to construct Markov partitions for smooth systems
with some hyperbolicity, both uniform and nonuniform. The conclusion is the existence
of a {\em symbolic model}.
Let us give the definitions.
Let $\mathfs G=(V,E)$ be an oriented graph. We assume that $V$ is finite or countable,
and that for each $v,w\in V$ there is at most one edge $v\to w$. 

\medskip
\noindent
{\sc Topological Markov shift (TMS):} Let
$$
\Sigma=\left\{\{v_n\}_{n\in\Z}\in V^\Z:v_n\to v_{n+1},\forall n\in\Z\right\}$$
be the set of $\Z$--indexed paths on $\mathfs G$,
and let $\sigma:\Sigma\to\Sigma$ be the {\em left shift}. The pair $(\Sigma,\sigma)$ is called a
{\em topological Markov shift}.

\medskip
For short, we will write TMS. An element of $\Sigma$ is denoted by $\un v=\{v_n\}_{n\in \Z}$.
We endow $\Sigma$ with the distance $d(\un v,\un w):={\rm exp}[-\min\{|n|:n\in\Z\text{ s.t. }v_n\neq w_n\}]$.
Let $\Sigma^\#$ be the {\em recurrent set} of $\Sigma$, defined by
$$
\Sigma^\#=\left\{\{v_n\}_{n\in\Z}\in\Sigma:
\begin{array}{l}
\exists v,w\in V\text{ s.t. }v_n=v\text{ for infinitely many }n>0\\
\text{and }v_n=w\text{ for infinitely many }n<0
\end{array}
\right\}.
$$
When $V$ is finite, $\Sigma^\#=\Sigma$. Let $f:M\to M$ be a diffeomorphism.

\medskip
\noindent
{\sc Symbolic model for diffeomorphism:} A {\em symbolic model} for $f:M\to M$
is a triple $(\Sigma,\sigma,\pi)$ where $(\Sigma,\sigma)$
is a TMS and $\pi:\Sigma\to M$ is a H\"older continuous map s.t. $\pi\circ\sigma=f\circ\pi$
and the restriction $\pi\restriction_{\Sigma^\#}:\Sigma^\#\to \pi[\Sigma^\#]$ is finite-to-one.

\medskip
Hence a symbolic model is a TMS together with a projection map $\pi$ that commutes $f$
and $\sigma$, and that is finite-to-one on $\pi[\Sigma^\#]$. A diffeomorphism
can have many symbolic models. Some of them are bad, when $\pi[\Sigma^\#]$
is much smaller than the subsets where $f$ displays an interesting dynamics.
A good symbolic model is one for which $\pi[\Sigma^\#]$ contains the subset where $f$ displays
chaotic dynamics. For us, this occurs when $\pi[\Sigma^\#]$ carries $\chi$--hyperbolic measures.
To define a symbolic model for flows, we add the flow direction to the TMS.

\medskip
\noindent
{\sc Topological Markov flow (TMF):} \label{Def-TMF}
Given a TMS $(\Sigma,\sigma)$ and
a  H\"older continuous function $r:\Sigma\to\R$ with $0<\inf r\leq \sup r<\infty$, define the
{\em topological Markov flow} $(\Sigma_r,\sigma_r)$ by:
\begin{enumerate}[$\circ$]
\item $\Sigma_r=\{(\un v,t):\un v\in\Sigma, 0\leq t\leq r(\un v)\}$ with the identification
$(\un v,r(\un v))\sim(\sigma(\un v),0)$.
\item $\sigma_r=\{\sigma_r^t\}_{t\in\R}:\Sigma_r\to\Sigma_r$ the unit speed vertical flow on $\Sigma_r$,
called the {\em suspension flow}, see Figure \ref{figure-suspension}.
\end{enumerate}

\begin{figure}[hbt!]
\centering
\def\svgwidth{6cm}
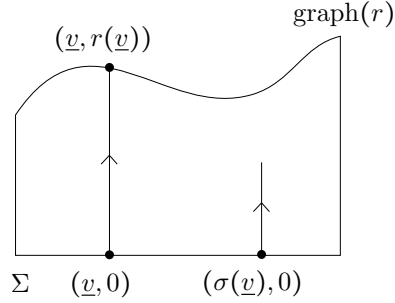
\caption{The suspension flow $\sigma_r$: starting at $(\un v,0)$, flow at unit speed until hitting the graph of $r$,
then return to the basis via the identification $(\un v,r(\un v))\sim(\sigma(\un v),0)$ and continue flowing.}
\label{figure-suspension}
\end{figure}

\medskip
An element of $\Sigma_r$ is denoted by $(\un v,t)$. Let
$\Sigma_r^\#$ be the {\em recurrent set} of $\Sigma_r$,
$$
\Sigma_r^\#=\left\{(\un v,t)\in\Sigma_r:\un v\in\Sigma^\#\right\}.
$$
See \cite{Lima-Sarig} for basic properties on $(\Sigma_r,\sigma_r)$.
Let $\vf:M\to M$ be a flow.

\medskip
\noindent
{\sc Symbolic model for flow:} A {\em symbolic model} for $\vf:M\to M$ 
is a triple $(\Sigma_r,\sigma_r,\pi_r)$ where $(\Sigma_r,\sigma_r)$
is a TMF and $\pi_r:\Sigma_r\to M$ is a H\"older continuous map s.t.
$\pi_r\circ\sigma_r^t=\vf^t\circ \pi_r$ for all $t\in \R$ for which the restriction
$\pi_r\restriction_{\Sigma_r^\#}:\Sigma_r^\#\to \pi_r[\Sigma_r^\#]$ is finite-to-one.


\section{Symbolic dynamics for uniformly hyperbolic systems}\label{Section-MP-UH}

There are at least two general ways of constructing Markov partitions for uniformly hyperbolic systems.
One of them, due to Sina{\u\i}, is called the {\em method of successive approximations}
\cite{Sinai-MP-U-diffeomorphisms,Sinai-Construction-of-MP}.
The second, due to Bowen, is called the {\em method of pseudo-orbits} \cite{Bowen-LNM}.
Their common feature is the use of the local invariant
manifolds constructed in Part \ref{Part-1} as dynamically defined systems of coordinates. 
For completeness and ease of understanding,
we also describe the construction of Adler and Weiss for hyperbolic toral
automorphisms \cite{Adler-Weiss-1967}. Since for nonuniformly hyperbolic systems
we will make use of the method of pseudo-orbits, we will only sketch the other techniques, 
and the details can be found in the original papers. We start defining Markov partitions
(and their flow counterpart) in the context of uniformly hyperbolic systems,
and explain how they generate symbolic models. It is important mentioning that, for uniformly
hyperbolic systems, the vertex set $V$ of the oriented graph $\mathfs G$ is finite.

\subsection{Markov partitions/sections}\label{Subsec-UH-MP}

Let $f:M\to M$ be a diffeomorphism. As already mentioned, our goal is to construct a partition
of $M$ so that the dynamics of $f$ can be represented by a TMS. Let $\mathfs G=(V,E)$
be the graph defining the TMS $(\Sigma,\sigma)$. The vertex set $V$ is the set of partition elements,
and each edge in $E$ will represent one possible transition
by the iteration of $f$. If $v_0\to v_1$ and $v_1\to v_2$ are edges,
then their concatenation is a path from $v_0$ to $v_2$. This property, translated to the dynamics of $f$,
is the {\em Markov property}. More precisely, if $R_0,R_1,R_2$ are the partition elements
associated to $v_0,v_1,v_2$, then there is a point $x\in R_0$ s.t. $f(x)\in R_1$ and there
is a point $y\in R_1$ s.t. $f(y)\in R_2$. The Markov property ensures that there is a point
$z\in R_0$ s.t. $f(z)\in R_1$ and $f^2(z)\in R_2$.

\medskip
Imagine, for a moment,
that $f$ is uniformly expanding. If we define edges $R_0\to R_1$
when $f(R_0)\cap R_1\neq\emptyset$, then the above property is not guaranteed. 
If instead we define $R_0\to R_1$ when $f(R_0)\supset R_1$, then  
the referred concatenation holds. When $f$ is a uniformly hyperbolic diffeomorphism, 
the definition of edges requires two inclusions, one for each invariant direction.
Let us give the definitions. The references for the discussion in this section are
\cite{Bowen-LNM} for diffeomorphisms and \cite{Bowen-Symbolic-Flows} for flows.
Let $f:M\to M$ be an Axiom A diffeomorphism, and fix $\ve>0$ small. By Part \ref{Part-1},
each $x\in\Omega(f)$ has a local stable manifold
$W^s_{\rm loc}(x)=V^s[x]$ and a local unstable manifold $W^u_{\rm loc}(x)=V^u[x]$.
By definition, $W^{s/u}_{\rm loc}(x)$ is tangent to $E^{s/u}_x$ at $x$, hence 
$W^s_{\rm loc}(x),W^u_{\rm loc}(x)$ are transversal at $x$. The
maps $x\in \Omega(f)\mapsto W^{s/u}_{\rm loc}(x)$
are continuous, see e.g. \cite[Thm 6.2(2)]{Shub-Book}. Hence, if $x,y\in \Omega(f)$
with ${\rm dist}(x,y)\ll 1$ then $W^s_{\rm loc}(x)$ and $W^u_{\rm loc}(y)$
intersect transversally at a single point. Fix $\delta\ll \ve$, and consider the following definition.

\medskip
\noindent
{\sc Smale bracket:} For $x,y\in\Omega(f)$ with $d(x,y)<\delta$, the
{\em Smale bracket} of $x$ and $y$ is defined by $\{[x,y]\}:=W^s_{\rm loc}(x)\cap W^u_{\rm loc}(y)$.

\medskip
For $R\subset \Omega(f)$, let $R^*$
denote the interior of $R$ in the induced topology of $\Omega(f)$.

\medskip
\noindent
{\sc Rectangle:} A subset $R\subset \Omega(f)$ is called a {\em rectangle} if it satisfies:
\begin{enumerate}[(1)]
\item {\sc Regularity:} $R=\overline{R^*}$ and ${\rm diam}(R)<\delta$.
\item {\sc Product structure:} $x,y\in R\Rightarrow [x,y]\in R$.
\end{enumerate}

\medskip
The product structure means that $R$ is a rectangle in the system of coordinates
given by the local invariant manifolds. Let $W^{s/u}(x,R):=W^{s/u}_{\rm loc}(x)\cap R$.
Regardless $W^{s/u}_{\rm loc}(x)$ are smooth manifolds,
since $\Omega(f)$ is usually a fractal set, then $W^{s/u}(x,R)$ are also usually fractal.
\medskip
It is easy to construct rectangles\footnote{Given $\rho>0$, let 
$W^{s/u}_\rho(x)=W^{s/u}_{\rm loc}(x)\cap B(x,\rho)$. If $x\in \Omega(f)$,
then $[W^u_\rho(x)\cap\Omega(f),W^s_\rho(x)\cap\Omega(f)]$ is a rectangle for
all $\rho>0$ small enough.}. Let $\mathfs R$ be a finite cover of $\Omega(f)$ by rectangles. 

\medskip
\noindent
{\sc Markov partition:} $\mathfs R$ is called a {\em Markov partition} for $f$ if it satisfies:
\begin{enumerate}[(1)]
\item {\sc Disjointness:} The elements of $\mathfs R$ can only intersect
at their boundaries\footnote{Boundaries are considered with respect to the relative topology
of $W^{s/u}(x,R)$, see \cite{Bowen-LNM}.}.
\item {\sc Markov property:} If $x\in R^*$ and $f(x)\in S^*$, then
$$
f(W^s(x,R))\subset W^s(f(x),S)\ \text{ and }\ f^{-1}(W^u(f(x),S))\subset W^u(x,R).
$$
\end{enumerate}

\medskip
If $\mathfs R$ only satisfies (2), we call it a {\em Markov cover}.
The two latter inclusions represent two Markov properties, one
for each invariant direction.
Geometrically, they assure that if two rectangles intersect, then the intersection occurs all the way
from one side to the other, with respect to the system of coordinates of the local invariant manifolds,
see Figure \ref{Figure-Markov-property}. We stress that, while here every rectangle has non-empty
interior, in the nonuniformly hyperbolic situation we will not be able to guarantee this. 
\begin{figure}[hbt!]
\centering
\def\svgwidth{12.9cm}
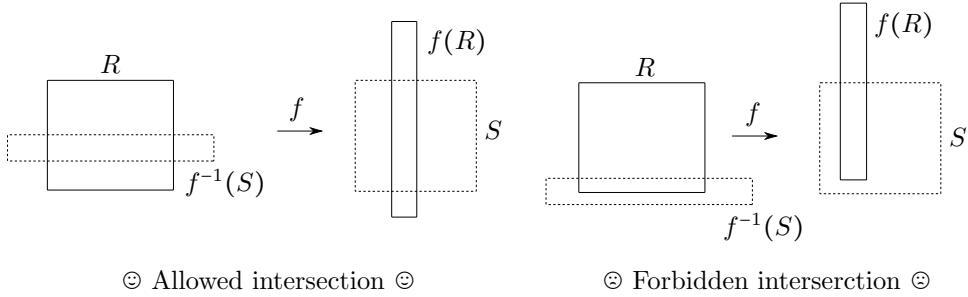
\caption{The Markov property: if $f(R)$ intersects $S$ non-trivially,
then $f(R)$ crosses $S$ completely all the way from one side to the other.}\label{Figure-Markov-property}
\end{figure}

\medskip
Now let $\vf:M\to M$ be an Axiom A flow. Recall the definitions 
of Subsection \ref{Subsec-NUH-dis1-definitions}.
Given an interval $I\subset \R$ and $Y\subset M$, let $\vf^I(Y):=\bigcup_{t\in I}\vf^t(Y)$.

\medskip
\noindent
{\sc Proper section:} A finite family $\mathfs M=\{B_1,\ldots,B_n\}$
is a {\em proper section of size $\alpha$} if there are closed differentiable discs $D_1,\ldots,D_n$ 
transverse to the flow direction s.t.:
\begin{enumerate}[(1)]
\item {\sc Closedness:} Each $B_i$ is a closed subset of $\Omega(\vf)$. 
\item {\sc Cover:} $\Omega(\vf)=\bigcup_{i=1}^n\vf^{[0,\alpha]}(B_i)$.
\item {\sc Regularity:} $B_i\subset {\rm int}(D_i)$ and $\overline{B_i^*}=B_i$, where $B_i^*$
is the interior of $B_i$ in the induced topology of $D_i\cap\Omega(\vf)$.
\item {\sc Partial order:} For $i\neq j$, at least one of the sets $D_i\cap \vf^{[0,4\alpha]}(D_j)$
and $D_j\cap\vf^{[0,4\alpha]}(D_i)$ is empty; in particular $D_i\cap D_j=\emptyset$.
\end{enumerate}

\medskip
For simplicity, denote $B_1\cup\cdots\cup B_n$ also by $\mathfs M$.
Let $f:\mathfs M\to\mathfs M$ be the
Poincar\'e return map of $\mathfs M$, and $\mathfrak t:\mathfs M\to (0,\infty)$
the return time function.
By properties (2) and (4),
$0<\inf\mathfrak t\leq \sup\mathfrak t\leq\alpha$.
By transversality, the stable/unstable directions
of $\vf$ project to stable/unstable directions of the Poincar\'e map $f$.
Also, local invariant manifolds of $f$ are projections, in the flow direction, of local invariant manifolds
of $\vf$, and we can similarly define the Smale bracket $[\cdot,\cdot]$ for $f$.

\medskip
The maps $f,\mathfrak t$ are not continuous, but they are continuous on the subset
$$
\mathfs M':=\left\{x\in\mathfs M:f^k(x)\in \bigcup B_i^*,\forall k\in\Z\right\}.
$$ 
Considering points in $\mathfs M'$ avoids
many problems, the first being the definition of the Markov property. We do not want to consider
a transition from $B_i$ to $B_j$ when $f(B_i)\cap B_j$ is a subset of $\partial B_j$.

\medskip
\noindent
{\sc Transitions:} We write $B_i\to B_j$ if there exists $x\in\mathfs M'$ s.t. $x\in B_i$,
$f(x)\in B_j$. When this happens, define $\mathfs T^s(B_i,B_j):=\overline{\{x\in\mathfs M':x\in B_i,f(x)\in B_j\}}$
and $\mathfs T^u(B_i,B_j):=\overline{\{y\in\mathfs M':y\in B_j,f^{-1}(y)\in B_i\}}$.

\medskip
\noindent
{\sc Markov section:} $\mathfs M$ is called a {\em Markov section}
of size $\alpha$ for $\vf$ if it is a proper section of size $\alpha$ with the following additional properties:
\begin{enumerate}[(1)]
\item[(5)] {\sc Product structure:} Each $B_i$ is a rectangle.
\item[(6)] {\sc Markov property:} If $B_i\to B_j$, then
\begin{align*}
x\in\mathfs T^s(B_i,B_j)&\Rightarrow W^s(x,B_i)\subset \mathfs T^s(B_i,B_j)\\
y\in\mathfs T^u(B_i,B_j)&\Rightarrow W^u(y,B_j)\subset\mathfs T^u(B_i,B_j).
\end{align*}
\end{enumerate}

\medskip
Above, $W^s(x,B_i)=\{[x,y]:y\in B_i\}$ is the intersection of the local stable manifold of $f$ at $x$
with $B_i$. The definition of $W^u(y,B_j)$ is similar.

\subsection{Markov partitions/sections generate symbolic models}\label{Subsec-MP-implies}

If $\mathfs R$ is a Markov partition for $f$, then $f$ has a symbolic model defined by:
\begin{enumerate}[$\circ$]
\item $\mathfs G=(V,E)$ with $V=\mathfs R$ and $E=\{R\to S:f(R^*)\cap S^*\neq\emptyset\}$.
\item $\pi:\Sigma\to\Omega(f)$ is defined for $\underline{R}=\{R_n\}_{n\in\Z}\in\Sigma$
by
$$
\{\pi(\underline{R})\}:=\bigcap_{n\geq 0} f^n(R_{-n})\cap\cdots\cap f^{-n}(R_n)
=\bigcap_{n\geq 0} \overline{f^n(R_{-n})\cap\cdots\cap f^{-n}(R_n)}.
$$
\end{enumerate}
Alternatively, $\pi(\underline{R})$ is the unique $x\in\Omega(f)$ s.t. $f^n(x)\in R_n$, $\forall n\in\Z$. 
The map $\pi$ is well-defined due to the Markov property and uniform hyperbolicity.
Clearly $f\circ\pi=\pi\circ\sigma$.
Additionally, $\pi$ is a finite-to-one continuous surjection that is one-to-one on a residual subset of $\Omega(f)$,
see \cite[Thm. 3.18]{Bowen-LNM} for details.

\medskip
If $\mathfs M$ is a Markov section for $\vf$, then $\vf$ has a symbolic model:
\begin{enumerate}[$\circ$]
\item $\mathfs G=(V,E)$ with $V=\mathfs M$ and $E=\{B_i\to B_j:\exists x\in\mathfs M'\text{ s.t. }
x\in B_i^*,f(x)\in B_j^*\}$.
\item $\pi:\Sigma\to \mathfs M$ is defined for $\underline{B}=\{B_n\}_{n\in\Z}\in\Sigma$
by
$$
\{\pi(\underline{B})\}:=\bigcap_{n\geq 0} f^n(B_{-n})\cap\cdots\cap f^{-n}(B_n)
=\bigcap_{n\geq 0} \overline{f^n(B_{-n})\cap\cdots\cap f^{-n}(B_n)}.
$$ 
\item $r:\Sigma\to\R$ is defined by $r:=\mathfrak t\circ\pi$.
\item $\pi_r:\Sigma_r\to \Omega(\vf)$ is defined by $\pi_r(\un B,t):=\vf^t[\pi(\un B)]$.
\end{enumerate}
Again, $\pi$ is well-defined because of the Markov property and uniform hyperbolicity,
and satisfies $f\circ\pi=\pi\circ\sigma$.
Also, $\pi$ is a finite-to-one continuous surjection that is one-to-one on $\mathfs M'$,
see \cite{Bowen-Symbolic-Flows} for details. 

\medskip
Therefore, to get symbolic models for uniformly hyperbolic symbolic systems,
it is enough to construct Markov partitions/sections.

\subsection{Markov partitions for two-dimensional hyperbolic toral automorphisms}\label{Subsec-Adler-Weiss}

This method, developed by Adler and Weiss \cite{Adler-Weiss-1967}, constructs 
finite Markov partitions for two-dimensional hyperbolic toral automorphisms. 
A particular case was constructed by Berg \cite{Berg-entropy-1967}.
Consider the cat map introduced in Example 1 of Subsection \ref{Subsec-UH-definitions},
and let $\vec{0}=(0,0)\in\mathbb T^2$. Clearly, $f(\vec{0})=\vec{0}$. Since the matrix $A$ is
hyperbolic, $\vec{0}$ has two eigendirections, call $W^s$ the contracting one and 
$W^u$ the expanding one. By linearity, $W^s$ and $W^u$ are
the (global) stable and unstable manifolds of $\vec{0}$. 

\medskip
The idea to obtain a Markov partition
is to construct a fundamental domain of $\mathbb T^2$ whose sides are pieces of $W^s$ and $W^u$,
and then subdivide this domain into finitely many rectangles satisfying the Markov property. 
In Figure \ref{figure-fundamental-domain},
we draw one possibility for the tesselation of $\R^2$ by one such fundamental domain.
\begin{figure}[hbt!]
\centering
\def\svgwidth{9cm}
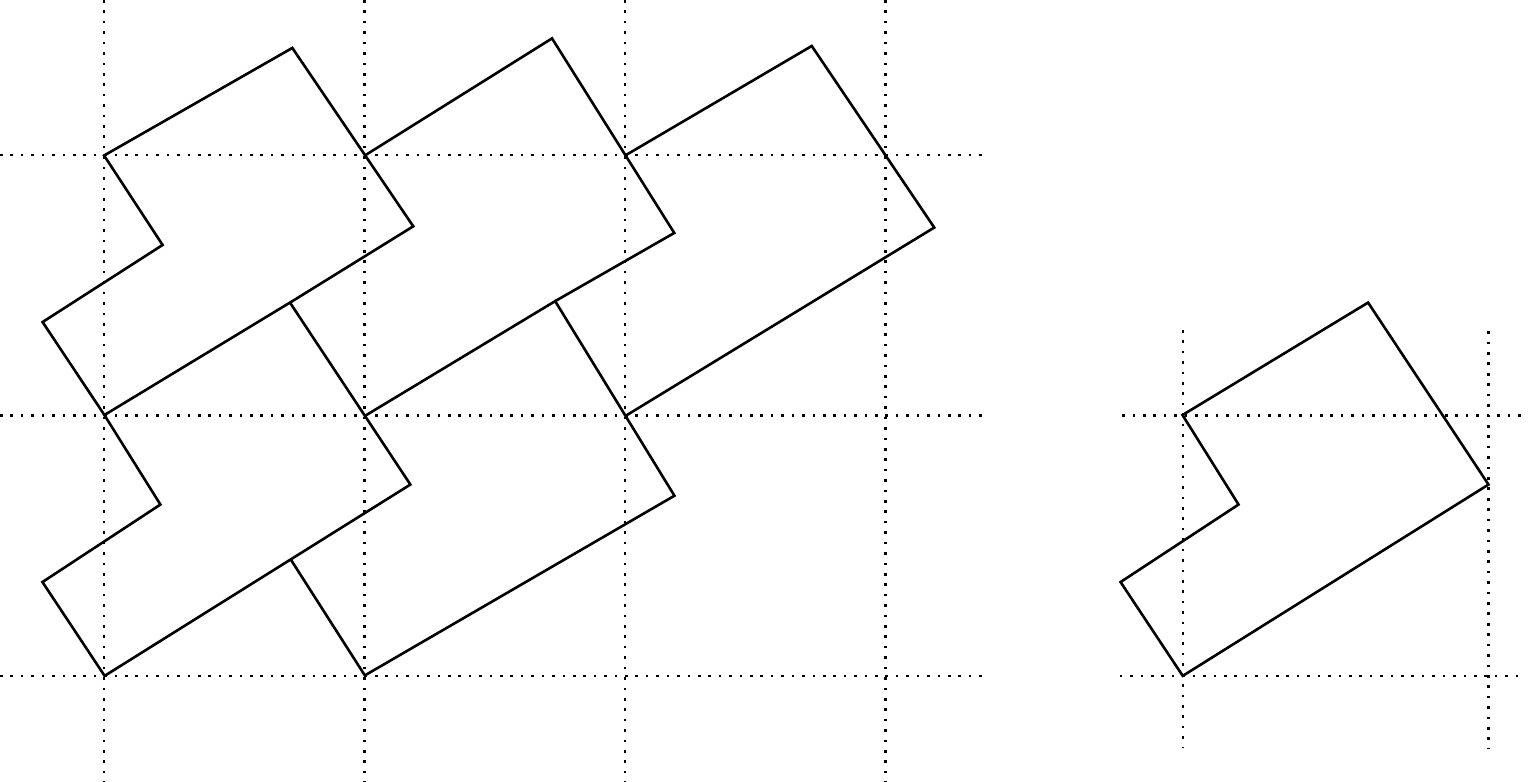
\caption{A tesselation of $\R^2$ by fundamental domains whose
sides are parallel to $W^s$ and $W^u$.}\label{figure-fundamental-domain}
\end{figure}
For a general two-dimensional hyperbolic toral automorphism, the construction
of the fundamental domain consists of two steps.

\medskip
\noindent
{\sc Step 1.} Take a cover $\mathfs R$ of $\T^2$ by finitely many rectangles whose sides
belong to $W^s$ and $W^u$ s.t. that every non-trivial intersection $f(R^*)\cap S^*$ is connected,
i.e. $f(R^*)$ does not intersect $S^*$ ``twice''.

\medskip
\noindent
{\sc Step 2.} Since $f$ contracts $W^s$, which is the stable manifold of the fixed
point $\vec{0}$, the stable boundary of $f(\mathfs R)$ is contained in $W^s$,
while its unstable boundary contains $W^u$. Partition $\mathfs R$ further by
adding the pre-image of the unstable segments of $f(\mathfs R)$.

\medskip
The final cover $\mathfs R$ is a finite Markov partition, see
\cite{Adler-Weiss-Similarity-Toral-Automorphisms} for details.
The projection map $\pi:\Sigma\to\T^2$ is a finite-to-one continuous surjection that is one-to-one
on the set $\{x\in\T^2:f^n(x)\in\bigcup_{R\in\mathfs R}R^*,\forall n\in\Z\}$.

\medskip
In our example, it is enough to divide the fundamental domain into three rectangles $R_1,R_2,R_3$
as in Figure \ref{figure-graph}. We leave as exercise to show that the images
$f(R_1),f(R_2),f(R_3)$ are as depicted in Figure \ref{figure-graph}, so that 
$\{R_1,R_2,R_3\}$ is a Markov partition.
The graph defining the TMS is also depicted in Figure \ref{figure-graph}. 
\begin{figure}[hbt!]
\centering
\def\svgwidth{12.5cm}
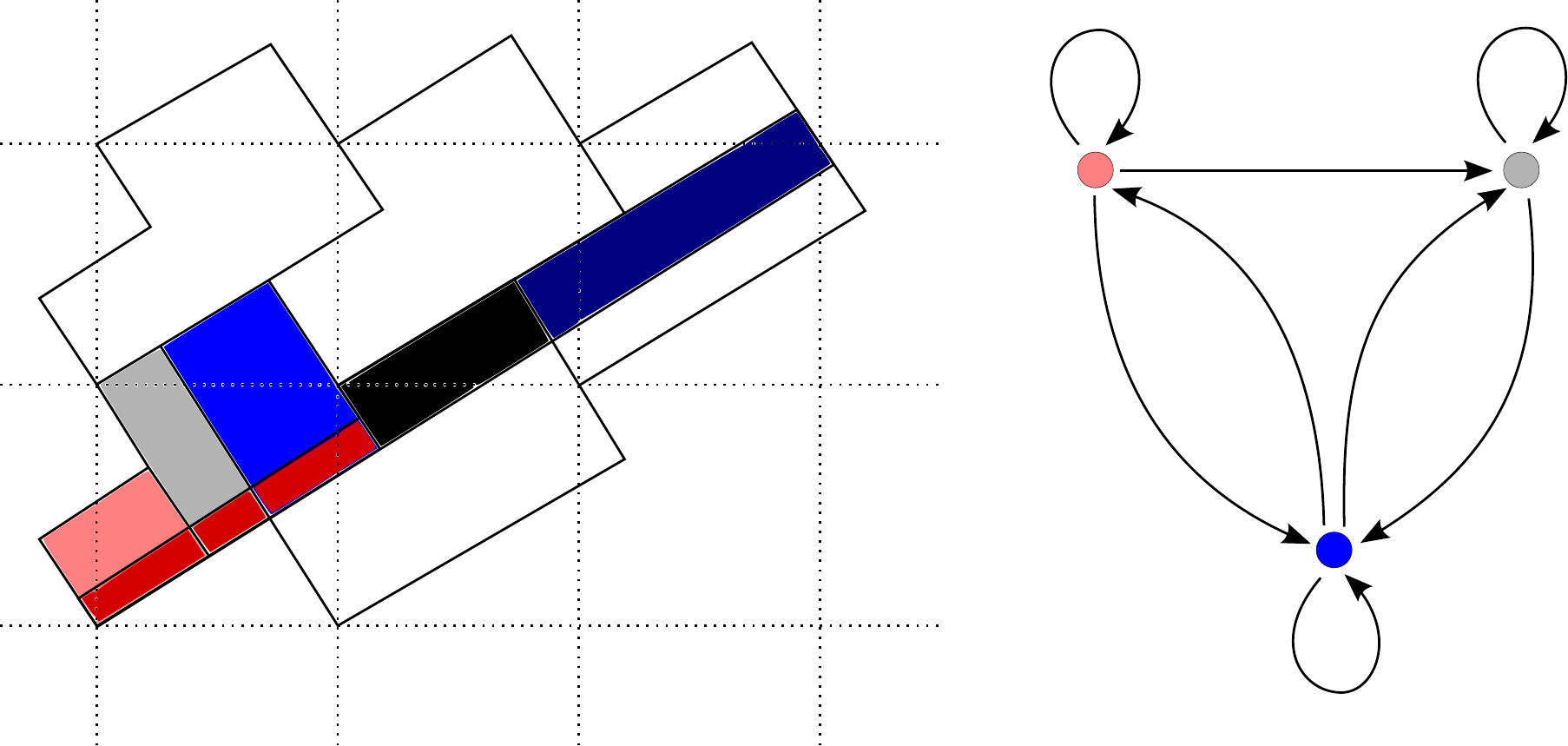
\caption{A Markov partition for the cat map by three rectangles $R_1,R_2,R_3$, and the graph
defining the respective TMS.}\label{figure-graph}
\end{figure}

\medskip
For higher dimensional hyperbolic toral automorphisms, a similar construction works, but
there is an important difference from the two-dimensional case: the boundary of a Markov
partition is not smooth \cite{Bowen-Not-smooth}.

\subsection{The method of successive approximations for diffeomorphisms}\label{Subsec-successive-diffeo}

This method, due to Sina{\u\i} \cite{Sinai-MP-U-diffeomorphisms,Sinai-Construction-of-MP},
provides Markov partitions for Anosov diffeomorphisms.
It was later modified by Bowen to also work for Axiom A diffeomorphisms \cite{Bowen-MP-Axiom-A}.
The construction consists of three main steps. Below, we explain them for Anosov diffeomorphisms.

\medskip
\noindent
{\sc Step 1 (Coarse graining).} Let $\mathfs T=\{T_i\}$ be a finite cover of $M$ by
rectangles (as we have argued in Subsection \ref{Subsec-UH-MP}, it is easy to build one such cover).

\medskip
\noindent
{\sc Step 2 (Successive Approximations).} Recursively define families $\mathfs S_k=\{S_{i,k}\}$ and
$\mathfs U_k=\{U_{i,k}\}$ of rectangles as follows:
\begin{enumerate}[$\circ$]
\item $S_{i,0}=U_{i,0}=T_i$.
\item If $\mathfs S_k,\mathfs U_k$ are defined, let
\begin{eqnarray*}
S_{i,k+1}:=&\displaystyle\bigcup_{x\in S_{i,k}}\{[y,z]:y\in S_{i,k}, z\in f(W^s(f^{-1}(x),S_{j,k}))\text{ for }f^{-1}(x)\in S_{j,k}\}\\
U_{i,k+1}:=&\displaystyle\bigcup_{x\in U_{i,k}}\{[z,y]:y\in U_{i,k}, z\in f^{-1}(W^u(f(x),U_{j,k}))\text{ for }f(x)\in U_{j,k}\}.
\end{eqnarray*}
\end{enumerate}
Let $S_i:=\bigcup_{k\geq 0}S_{i,k}$, $U_i:=\bigcup_{k\geq 0}U_{i,k}$, and $Z_i:=[\overline{U_i},\overline{S_i}]$.
Then $\mathfs Z=\{Z_i\}$ is a Markov cover.

\medskip
Let us understand the above definitions. Representing the stable direction by the horizontal
direction, we identify what are the horizontal components 
of $S_{i,k+1}\backslash S_{i,k}$, see Figure \ref{figure-successive}.
\begin{figure}[hbt!]
\centering
\def\svgwidth{9cm}
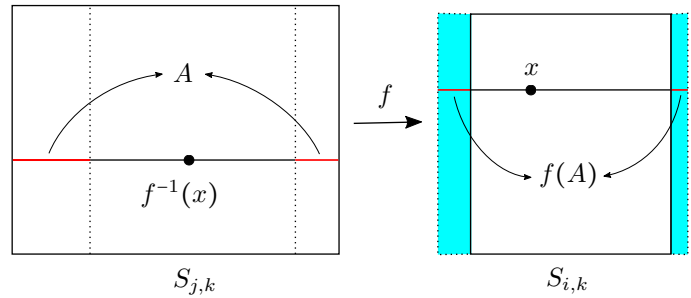
\caption{The horizontal component of $S_{i,k+1}\backslash S_{i,k}$ is the union of sets of the form $f(A)$,
where $A$ is a horizontal subset of $S_{j,k}\backslash S_{j,k-1}$.
Above, $A$ is composed of the two ticker segments on the left figure and the added set is in grey
in the right figure.}\label{figure-successive}
\end{figure}
Start observing that the horizontal component of $S_{i,1}\backslash S_{i,0}$ is the union of sets
of the form $f(A)$, where $A=W^s(f^{-1}(x),S_{j,0})\backslash f^{-1}(W^s(x,S_{i,0}))$.
Each such $A$ has diameter less than 1, hence $f(A)$ has diameter 
less than $\kappa$, thus $S_{i,1}$ equals the union of $S_{i,0}$ and a
set of horizontal diameter less than $\kappa$.  
Similarly, the horizontal component of $S_{i,2}\backslash S_{i,1}$ is the union of sets of the form
$f(A)$, where $A$ is a horizontal subset of $S_{j,1}\backslash S_{j,0}$, therefore
$S_{i,2}$ equals the union of $S_{i,1}$ and a
set of horizontal diameter less than $\kappa^2$.
By induction, $S_{i,k+1}$ equals the union of $S_{i,k}$ and a
set of horizontal diameter less than $\kappa^{k+1}$. This shows that each $S_i$ 
is well-defined, and the same occurs for each $U_i$.

\medskip
\noindent
{\sc Step 3 (Bowen-Sina{\u\i} refinement).} To destroy non-trivial intersections, refine $\mathfs Z$ as follows.
For $Z_i$, let $I_i=\{j:Z_i^*\cap Z_j^*\neq\emptyset\}$.
For $j\in I_i$, let $\mathfs E_{ij}=$ cover of $Z_i$ by rectangles (see Figure \ref{refinement}):
\begin{align*}
E_{ij}^{su}&:=\overline{\{x\in Z_i^*: W^s(x,Z_i)\cap Z_j^*\neq\emptyset,
W^u(x,Z_i)\cap Z_j^*\neq\emptyset\}}\\
E_{ij}^{s\emptyset}&:=\overline{\{x\in Z_i^*: W^s(x,Z_i)\cap Z_j^*\neq\emptyset,
W^u(x,Z_i)\cap Z_j=\emptyset\}}\\
E_{ij}^{\emptyset u}&:=\overline{\{x\in Z_i^*: W^s(x,Z_i)\cap Z_j=\emptyset,
W^u(x,Z_i)\cap Z_j^*\neq\emptyset\}}\\
E_{ij}^{\emptyset\emptyset}&:=\overline{\{x\in Z_i^*: W^s(x,Z_i)\cap Z_j=\emptyset,
W^u(x,Z_i)\cap Z_j=\emptyset\}}.
\end{align*}
\begin{figure}[hbt!]
\centering
\def\svgwidth{12.5cm}
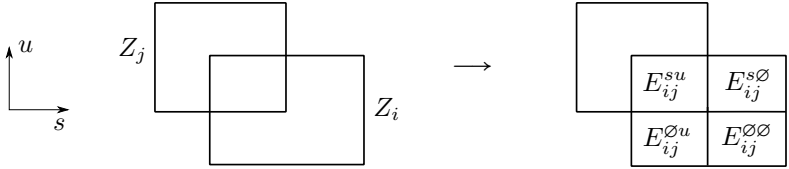\caption{$\mathfs E_{ij}=\{E_{ij}^{su},E_{ij}^{s\emptyset},E_{ij}^{\emptyset u},E_{ij}^{\emptyset\emptyset}\}$ is a cover of $Z_i$ by rectangles.}\label{refinement}
\end{figure}
Hence $\mathfs R:=$ cover defined by $\{\mathfs E_{ij}:Z_i\in \mathfs Z,j\in I_i\}$
is a Markov partition for $f$, and the induced $\pi:\Sigma\to M$ is a finite-to-one continuous surjection
that is one-to-one on $\{x\in M:f^n(x)\in \bigcup_{R\in\mathfs R}R^*,\forall n\in\Z\}$.

\subsection{The method of successive approximations for flows}\label{Subsec-successive-flow}

Ratner applied the method of successive approximations for three dimensional
Anosov flows \cite{Ratner-MP-three-dimensions}. Later she extended it for
higher dimensional Anosov flows \cite{Ratner-MP-n-dimensions},
and Bowen used it for Axiom A flows \cite{Bowen-Symbolic-Flows}. Below, we follow Bowen's
construction. As usual, the main difficulty when dealing with flows is the presence of discontinuities
for the Poincar\'e return map.

\medskip
Consider a proper section $\mathfs C$. Since the stable/unstable directions
of $\vf$ project to stable/unstable directions of the Poincar\'e map $f$,
it is easy to construct rectangles inside $\mathfs C$. Let $\mathfs R$ be a cover
of $\mathfs C\cap\Omega(\vf)$ by rectangles.
To apply successive approximations (Step 2 of the last subsection), proceed as follows:
\begin{enumerate}[$\circ$]
\item Take $L>0$ large s.t. for every $x\in R\in\mathfs R$ there are $C^+,C^-\in\mathfs C$ s.t.
$\vf^L(W^s_{\rm loc}(x))\subset \vf^{[-\alpha,\alpha]}(C^+)$ and
$\vf^{-L}(W^u_{\rm loc}(x))\subset \vf^{[-\alpha,\alpha]}(C^-)$. The existence of $L$ follows
from the uniform hyperbolicity of $\vf$.
\item For each such $x$, take a neighborhood $V\ni x$ small enough s.t. 
$\vf^L(V)\subset \vf^{[-\alpha,\alpha]}(C^+)$ and $\vf^{-L}(V)\subset \vf^{[-\alpha,\alpha]}(C^-)$,
and define $f^+_{V}:V\to C^+$ and $f^-_{V}:V\to C^-$ by:
$$
f^+_V:=(\text{projection to }C^+)\circ\vf^L\, ,\ \  f^-_V:=(\text{projection to }C^-)\circ\vf^{-L}.
$$
\item Pass to a finite collection of neighborhoods $V$ as above, and 
apply the method of successive approximations to the maps $f^+_V,f^-_V$.
The resulting cover by rectangles has a Markov property: for each $x\in\mathfs R$ there
are $k,\ell>0$ s.t. $x$ satisfies a stable Markov property at $f^k(x)$ and an unstable
Markov property at $f^{-\ell}(x)$.
\item The values of $k,\ell$ are uniformly bounded by some $N>0$.
\item To get the Markov property for $f$, apply a refinement procedure along
the iterates $-N,\ldots,N$ of $f$. The resulting partition $\mathfs M$ is a Markov section for $\vf$.
\end{enumerate}
See the details in \cite{Bowen-Symbolic-Flows}.

\subsection{The method of pseudo-orbits}\label{Subsec-pseudo-orbits}

Bowen provided an alternative method to construct Markov partitions for 
Axiom A diffeomorphisms \cite{Bowen-LNM}. His idea was to use the theory of pseudo-orbits
and shadowing, which explores the expected richness on the orbit structure
of uniformly hyperbolic systems. These notions appeared in the qualitative theory
of structural stability for uniformly hyperbolic systems. Indeed, Anosov considered
a version of pseudo-orbits for flows, which he called $\ve$--trajectories, and used them
to prove that Anosov flows are structurally stable, see  \cite[Thm. 1]{Anosov-Certain}.

\medskip
Let $f:M\to M$ be an invertible map. An orbit of $f$ is a sequence $\{x_n\}_{n\in\Z}$ s.t. $f(x_n)=x_{n+1}$
for all $n\in\Z$, while a {\em pseudo-orbit} is a sequence  $\{x_n\}_{n\in\Z}$ s.t.
$f(x_n)\approx x_{n+1}$ for all $n\in\Z$. In other words, a pseudo-orbit is an orbit up to small
errors at each iteration. This is exactly what a computer returns when we try to iterate a map:
due to roundoff errors, the sequence is not a real orbit but just a pseudo-orbit.
Since a hyperbolic matrix remains hyperbolic after a small perturbation, Theorem 
\ref{Thm-Lyapunov-chart} holds for pseudo-orbits, as we will see below: changing $f(x)$ to 
some nearby $y$, we can represent $f$ in the Lyapunov charts $\Psi_x$ and $\Psi_y$
and still obtain a small perturbation of a hyperbolic matrix. This is the main tool to introduce
the symbolic model. To maintain consistency with Part \ref{Part-1}, we continue assuming that $M$ is a surface.

\subsubsection{Pseudo-orbits}

Recall from Section \ref{Subsec-UH-charts} the definition of Lyapunov charts: for 
$\ve>0$ is small enough we let $Q=\ve^{3/\beta}$ and define, for each $x\in M$, its
Lyapunov chart $\Psi_x:[-Q,Q]^2\to M$. Recall that $\Psi_x$ is $2$--Lipschitz
and its inverse is $2\mathfs L$--Lipschitz.
The splitting $E^s\oplus E^u$ is continuous,
so there is $\delta=\delta(\ve)>0$ s.t. if $d(x,y)<\delta$ then 
$\|\Psi_ y^{-1}\circ \Psi_x-{\rm Id}\|_{1+\beta/2}<\ve^3$, where the norm is taken
in $[-Q,Q]^2$.\footnote{The composition $\Psi_y^{-1}\circ\Psi_x$ is well-defined
in $[-Q,Q]^2$. To see this, fix $\ve>0$ small enough so that each $\Psi_x$ is well-defined
in the larger domain $[-10\mathfs LQ,10\mathfs LQ]^2$.
Taking $\delta=\delta(\ve)>0$ small enough, if $d(x,y)<\delta$ then
$\Psi_x([-Q,Q]^2)\subset B(x,4Q)\subset B(y,5Q)\subset \Psi_y([-10\mathfs LQ,10\mathfs LQ]^2)$.
}

\medskip
\noindent
{\sc $\ve$--overlap:} Two Lyapunov charts $\Psi_x,\Psi_y$ are said to {\em $\ve$--overlap} if
$d(x,y)<\delta$. When this happens, we write $\Psi_x\overset{\ve}{\approx}\Psi_y$.

\medskip
Hence, if two points are close enough, the charts they define are essentially
the same.
This notation is somewhat redundant for uniformly
hyperbolic systems, but we prefer to state it as above because it helps understanding the
symbolic model and the difficulties when we consider nonuniformly 
hyperbolic systems.

\medskip
If $\Psi_{f(x)}\overset{\ve}{\approx}\Psi_y$, then we can write $f$ in the Lyapunov charts $\Psi_x$
and $\Psi_y$ as $f_{x,y}=\Psi_y^{-1}\circ f\circ \Psi_x$. Since
$f_{x,y}=\Psi_{y}^{-1}\circ\Psi_{f(x)}\circ f_x=:g\circ f_x$,
where $g:=\Psi_{y}^{-1}\circ\Psi_{f(x)}$ is a small perturbation of the identity,
the map $f_{x,y}$ is again a small perturbation of a hyperbolic matrix. The same reasoning
happens if $\Psi_{f^{-1}(y)}\overset{\ve}{\approx}\Psi_x$, in which case
$f_{x,y}^{-1}=\Psi_x^{-1}\circ f^{-1}\circ \Psi_y$, the representation of 
$f^{-1}$ in the Lyapunov charts $\Psi_x$ and $\Psi_y$, is also a small perturbation of a hyperbolic
matrix. This is summarized in the next theorem, which is the version of Theorem \ref{Thm-Lyapunov-chart}
in the present context.

\begin{theorem}\label{Thm-UH-pseudo-orbit}
The following hold for all $\ve>0$ small enough. If $\Psi_{f(x)}\overset{\ve}{\approx}\Psi_y$,
then $f_{x,y}$ is well-defined in $[-Q,Q]^2$ and can be written as
$f_{x,y}(v_1,v_2)=(Av_1+h_1(v_1,v_2),Bv_2+h_2(v_1,v_2))$ where: 
\begin{enumerate}[{\rm (1)}]
\item $|A|,|B^{-1}|<\lambda$, cf. Lemma \ref{Lemma-linear-reduction}.
\item $\|h_i\|_{1+\beta/3}<\ve$ for $i=1,2$.
\end{enumerate}
If $\Psi_{f^{-1}(y)}\overset{\ve}{\approx}\Psi_x$ then a
similar statement holds for $f_{x,y}^{-1}$.
\end{theorem}

Observe that, in contrast to Theorem \ref{Thm-Lyapunov-chart}, above we can only control 
the $C^{1+\beta/3}$ norm. This slight decrease is necessary to keep the estimates of size $\ve$.

\medskip
\noindent
{\sc Edge:} We write $\Psi_x\overset{\ve}{\to} \Psi_y$ if $\Psi_{f(x)}\overset{\ve}{\approx}\Psi_{y}$
and $\Psi_{f^{-1}(y)}\overset{\ve}{\approx}\Psi_{x}$.

\medskip
Conditions $\Psi_{f(x)}\overset{\ve}{\approx}\Psi_{y}$ and $\Psi_{f^{-1}(y)}\overset{\ve}{\approx}\Psi_{x}$
are called {\em nearest neighbor conditions}.

\medskip
\noindent
{\sc Pseudo-orbit:} A sequence of Lyapunov charts $\{\Psi_{x_n}\}_{n\in\Z}$
is called a {\em pseudo-orbit} if $\Psi_{x_n}\overset{\ve}{\to}\Psi_{x_{n+1}}$ for all $n\in\Z$.

\medskip
We point out that classically a pseudo-orbit is a sequence of points instead
of Lyapunov charts, but so far these notions coincide, since
$\Psi_{x_n}\overset{\ve}{\to}\Psi_{x_{n+1}}$ is equivalent to $d(f(x_n),x_{n+1})<\delta$ and
$d(f^{-1}(x_{n+1}),x_n)<\delta$.

\subsubsection{Graph transforms}

Assume that $\Psi_x\overset{\ve}{\to}\Psi_y$. Since $f_{xy}^{\pm 1}$ are perturbations of hyperbolic
matrices, we can proceed as in Subsection \ref{Subsect-UH-graph} and define
graph transforms between admissible manifolds at $\Psi_x$ and $\Psi_y$.
First, we need to redefine admissibility. For instance, we can no longer
require $F(0)=0$, since this property is not preserved by the maps $f_{x,y}^{\pm 1}$ (unless
$f(x)=y$). For ease of exposition, we continue not prescribing the precision in this definition.

\medskip
\noindent
{\sc Admissible manifolds:} An {\em $s$--admissible manifold} at $\Psi_x$ is a set
of the form $V^s=\Psi_x\{(t,F(t)):|t|\leq Q\}$, where $F:[-Q,Q]\to\R$ is a $C^1$ function s.t.
$F(0)\approx 0$ and $\|F'\|_{C^0}\approx 0$. Similarly, a {\em $u$--admissible manifold} at
$\Psi_x$ is a set of the form $V^u=\Psi_x\{(G(t),t):|t|\leq Q\}$, where $G:[-Q,Q]\to\R$ is a $C^1$
function s.t. $G(0)\approx 0$ and $\|G'\|_{C^0}\approx 0$.

\medskip
Let $\mathfs M^s_x,\mathfs M^u_x$ be the space of all $s,u$--admissible manifolds at $\Psi_x$
respectively, and introduce metrics on $\mathfs M^{s/u}_x$ as before. Assume that $\Psi_x\overset{\ve}{\to}\Psi_y$.

\medskip
\noindent
{\sc Graph transforms $\mathfs F^s_{x,y},\mathfs F^u_{x,y}$:} 
The {\em stable graph transform} $\mathfs F^s_{x,y}:\mathfs M^s_y\to\mathfs M^s_x$
is the map that sends $V^s\in \mathfs M^s_y$ to the unique $\mathfs F^s_{x,y}[V^s]\in \mathfs M^s_x$
with representing function $F$ s.t. $\Psi_x\{(t,F(t)):|t|\leq Q\}\subset f^{-1}(V^s)$.
Similarly, the {\em unstable graph transform}
$\mathfs F^u_{x,y}:\mathfs M^u_x\to\mathfs M^u_y$ is the map that sends
$V^u\in \mathfs M^u_x$ to the unique $\mathfs F^u_{x,y}[V^u]\in \mathfs M^u_y$
with representing function $G$ s.t. $\Psi_y\{(G(t),t):|t|\leq Q\}\subset f(V^u)$.

\medskip
Again, the hyperbolicity of $f_{x,y}^{\pm 1}$ implies the following result.

\begin{theorem}\label{Thm-UH-graph-pseudo-orbit}
If $\Psi_x\overset{\ve}{\to}\Psi_y$, then $\mathfs F^s_{x,y}$ and $\mathfs F^u_{x,y}$ are well-defined contractions.
\end{theorem}

Consequently, each pseudo-orbit has local stable and unstable manifolds.

\medskip
\noindent
{\sc Stable/unstable manifolds:} The {\em stable manifold} of $\un v=\{\Psi_{x_n}\}_{n\in\Z}$
is the unique $s$--admissible manifold
$V^s[\un{v}]\in \mathfs M^s_{x_0}$ defined by
$$
V^s[\un{v}]:=\lim_{n\to\infty}(\mathfs F^s_{x_0,x_1}\circ\cdots\circ \mathfs F^s_{x_{n-1},x_n})[V_n]
$$
for some (any) sequence $\{V_n\}_{n\geq 0}$ with $V_n\in\mathfs M^s_{x_n}$.
The {\em unstable manifold} of $\un v$ is the unique $u$--admissible manifold
$V^u[\un{v}]\in \mathfs M^u_{x_0}$ defined by
$$
V^u[\un{v}]:=\lim_{n\to-\infty}(\mathfs F^u_{x_{-1},x_0}\circ\cdots\circ \mathfs F^u_{x_n,x_{n+1}})[V_n]
$$
for some (any) sequence $\{V_n\}_{n\leq 0}$ with $V_n\in\mathfs M^u_{x_n}$.

\medskip
The observations at the end of Subsection \ref{Subsect-UH-graph} can be repeated
ipsis literis. In particular, $V^s[\un{v}]$ only depends on the future $\{\Psi_{x_n}\}_{n\geq 0}$,
while  $V^u[\un{v}]$ only depends on the past $\{\Psi_{x_n}\}_{n\leq 0}$.

\medskip
\noindent
{\sc Shadowing:} $\un v=\{\Psi_{x_n}\}_{n\in\Z}$ is said to {\em shadow} $x$
if $f^n(x)\in \Psi_{x_n}([-Q,Q]^2)$, $\forall n\in\Z$.

\begin{theorem}[Shadowing Lemma]
Every pseudo-orbit $\un v$ shadows a unique point $\{x\}=V^s[\un v]\cap V^u[\un v]$.
\end{theorem}

\medskip
This follows from the hyperbolicity of each $f_{x_n,x_{n+1}}^{\pm 1}$.

\subsubsection{Construction of a Markov partition}\label{Subsub-MP-pseudo-orbits}

Now we explain how Bowen used the above tools to construct
a Markov partition for $f$. The construction involves two codings, the first being usually
infinite-to-one and the second finite-to-one. We divide the construction into three steps.
Let $f:M\to M$ be Axiom A, and let $L>1$ be a Lipschitz constant for $f^{\pm 1}$.

\medskip
\noindent
{\sc Step 1 (Coarse graining).} Fix a finite subset of $X\subset \Omega(f)$ that is
$\tfrac{\delta}{2L}$--dense in $\Omega(f)$,
and let $\mathfs A=\{\Psi_x:x\in X\}$. Let $\mathfs G=(V,E)$ be the oriented graph with vertex
set $V=\mathfs A$ and edge set $E=\{\Psi_x\overset{\ve}{\to}\Psi_y\}$, and let $(\Sigma,\sigma)$
be the TMS defined by $\mathfs G$. Observe that an element of $\Sigma$ is a pseudo-orbit.


\medskip
\noindent
{\sc Step 2 (Infinite-to-one extension).} Using the Shadowing Lemma, define a map
$\pi:\Sigma\to \Omega(f)$ by
$$
\{\pi(\un v)\}:=V^s[\underline v]\cap V^u[\underline v].
$$
The map $\pi$ has the following properties:
\begin{enumerate}[$\circ$]
\item $\pi$ is surjective: for every $x\in \Omega(f)$, choose $\{x_n\}_{n\in\Z}\subset X$
s.t. $d(f^n(x),x_n)<\tfrac{\delta}{2L}$. Then $\un{v}=\{\Psi_{x_n}\}_{n\in\Z}$ is a pseudo-orbit, since
\begin{align*}
&\ d(f(x_n),x_{n+1})\leq d(f(x_n),f^{n+1}(x))+d(f^{n+1}(x),x_{n+1})\\
&\leq Ld(x_n,f^n(x))+d(f^{n+1}(x),x_{n+1})<\tfrac{\delta}{2}+\tfrac{\delta}{2L}<\delta,
\end{align*}
and similarly $d(f^{-1}(x_{n+1}),x_n)<\delta$. Clearly, $\pi(\un v)=x$.
\item $\pi\circ\sigma=f\circ\pi$: this follows from the Shadowing Lemma, since if $\un v$ shadows $x$
then $\sigma(\un v)$ shadows $f(x)$.
\item $\pi$ is usually infinite-to-one: imagine, for example, that for some $x\in\Omega(f)$
there are $x_n,y_n\in X$ s.t. $d(f^n(x),x_n)<\tfrac{\delta}{2L}$ and $d(f^n(x),y_n)<\tfrac{\delta}{2L}$.
Any choice of $z_n\in\{x_n,y_n\}$ defines a pseudo-orbit $\{\Psi_{z_n}\}_{n\in\Z}$ that shadows $x$,
hence $\pi^{-1}(x)$ has cardinality at least $2^\Z$, which is uncountable.
\end{enumerate}
The third property above is, in general, unavoidable. Hence, $(\Sigma,\sigma,\pi)$ is {\em not}
a symbolic model for $f$. But the Markov structure of $\Sigma$ induces, via $\pi$, a cover of
$\Omega(f)$  satisfying a (symbolic) Markov property.

\medskip
\noindent
{\sc The Markov cover $\mathfs Z$:} Let $\mathfs Z:=\{Z(v):v\in\mathfs A\}$, where
$$
Z(v):=\{\pi(\un v):\un v\in\Sigma\text{ and }v_0=v\}.
$$

\medskip
In other words, $\mathfs Z$ is the family defined by the natural partition of $\Sigma$ into
cylinder at the zeroth position. In general, each $Z(v)$ is fractal.
Admissible manifolds allow us to
define {\em invariant fibres} inside each $Z\in\mathfs Z$. Let $Z=Z(v)$.

\medskip
\noindent
{\sc $s$/$u$--fibres in $\mathfs Z$:} Given $x\in Z$, let $W^s(x,Z):=V^s[\un v]\cap Z$
be the {\em $s$--fibre} of $x$ in $Z$ for some (any) $\un v=\{v_n\}_{n\in\Z}\in\Sigma$
s.t. $\pi(\un v)=x$ and $v_0=v$. Similarly, let $W^u(x,Z):=V^u[\un v]\cap Z$ be
the {\em $u$--fibre} of $x$ in $Z$.

\medskip
Observe that, while $V^{s/u}[\un v]$ are smooth manifolds, $W^{s/u}(x,Z)$ is usually fractal.
Below we collect the main properties of $\mathfs Z$.

\begin{proposition}\label{Prop-UH-Z}
The following holds for all $\ve>0$ small enough.
\begin{enumerate}[{\rm (1)}]
\item {\sc Covering property:} $\mathfs Z$ is a cover of $\Omega(f)$.
\item {\sc Product structure:} For every $Z\in\mathfs Z$ and every $x,y\in Z$, the intersection
$W^s(x,Z)\cap W^u(y,Z)$ consists of a single point, and this point belongs to $Z$.
\item {\sc Symbolic Markov property:} If $x=\pi(\un v)$ with $\un v=\{v_n\}_{n\in\Z}\in\Sigma$, then
$$
f(W^s(x,Z(v_0)))\subset W^s(f(x),Z(v_1))\, \text{ and }\, f^{-1}(W^u(f(x),Z(v_1)))\subset W^u(x,Z(v_0)).
$$
\end{enumerate}
\end{proposition}

Part (1) follows from the surjectivity of $\pi$. To prove part (2), we define a Smale bracket
$[\cdot,\cdot]_Z$ for each $Z\in\mathfs Z$ as follows. Write $Z=Z(v)$, and let $x=\pi(\un v),y=\pi(\un w)$
where $\un v=\{v_n\}_{n\in\Z},\un w=\{w_n\}_{n\in\Z}\in\Sigma$ with $v_0=w_0=v$.
Then $W^s(x,Z)\cap W^u(y,Z)$ consists of a unique point $z=\pi(\un u)$ where $\un u=\{u_n\}_{n\in\Z}$
is defined by:
$$
u_n=\left\{
\begin{array}{rl}
v_n&\text{, if }n\geq 0\\
w_n&\text{, if }n\leq 0.\\
\end{array}
\right.
$$
The equality $z=\pi(\un u)$ follows from the Shadowing Lemma. Observe that $z\in Z$.
We write $z=:[x,y]_Z$. Finally, part (3) follows from the Markov structure of $\Sigma$.
At this point, it is also important to show that the above definitions are 
compatible among the elements of $\mathfs Z$.

\begin{lemma}\label{Lemma-UH-compatibility}
The following holds for all $\ve>0$ small enough.
\begin{enumerate}[{\rm (1)}]
\item {\sc Compatibility:} If $x,y\in Z(v_0)$ and $f(x),f(y)\in Z(v_1)$ with
$v_0\overset{\ve}{\to} v_1$ then $f([x,y]_{Z(v_0)})=[f(x),f(y)]_{Z(v_1)}$.
\item {\sc Overlapping charts properties:} If $Z=Z(\Psi_x),Z'=Z(\Psi_y)\in\mathfs Z$
with $Z\cap Z'\neq \emptyset$ then:
\begin{enumerate}[{\rm (a)}]
\item $Z\subset \Psi_y([-Q,Q]^2)$.
\item If $x\in Z\cap Z'$ then $W^{s/u}(x,Z)\subset V^{s/u}(x,Z')$. 
\item If $x\in Z,y\in Z'$ then $V^s(x,Z)$ and $V^u(y,Z')$ intersect at a unique point. 
\end{enumerate}
\end{enumerate}
\end{lemma}

Part (1) also follows from the Markov structure of $\Sigma$, while part (2) follows
from the fine control we have on the Lyapunov charts inside each rectangle $[-Q,Q]^2$.
Lemma \ref{Lemma-UH-compatibility} allows us to consider Smale brackets of different intersecting
rectangles.

\medskip
\noindent
{\sc Step 3 (Bowen-Sina{\u\i} refinement).} We repeat ipsis literis Step 3 performed in Subsection
\ref{Subsec-successive-diffeo}. The resulting partition $\mathfs R$ is a Markov partition.
By Subsection \ref{Subsec-MP-implies}, we obtain a symbolic model $(\widehat \Sigma,\widehat\sigma,\widehat\pi)$,
where $(\widehat\Sigma,\widehat\sigma)$ is the TMS defined by the graph
$\widehat{\mathfs G}=(\widehat V,\widehat E)$ with vertex set $\widehat V=\mathfs R$ and
edge set $\widehat E=\{R\to S:f(R^*)\cap S^*\neq\emptyset\}$.
Let $\widehat\pi:\Sigma\to\Omega(f)$ be defined for $\underline{R}=\{R_n\}_{n\in\Z}\in\Sigma$
by
$$
\{\widehat\pi(\underline{R})\}:=\bigcap_{n\geq 0} f^n(R_{-n})\cap\cdots\cap f^{-n}(R_n)
=\bigcap_{n\geq 0} \overline{f^n(R_{-n})\cap\cdots\cap f^{-n}(R_n)}.
$$
The $\widehat\pi$ is a finite-to-one surjection that
is one-to-one on the set $\{x\in \Omega(f):f^n(x)\in \bigcup_{R\in\mathfs R}R^*,\forall n\in\Z\}$.

\begin{remark}
The method of pseudo-orbits also works for uniformly expanding maps, in which case
the TMS is one-sided. To do this, assume for simplicity that $f:M\to M$ satisfies
$d(f(x),f(y))\geq \kappa^{-1}d(x,y)$. Define Lyapunov
charts simply by $\Psi_x:={\rm exp}_x$, then
define an edge $\Psi_x\overset{\ve}{\to}\Psi_y$ iff $d(f(x),y)\ll 1$, and prove that each
pseudo-orbit $\{\Psi_{x_n}\}_{n\geq 0}$ shadows a single point $x\in M$.
Now implement Steps 1--3.
\end{remark}

\subsubsection{Bowen relation}\label{Subsub-Bowen-relation}
We explain why $\widehat\pi$ is finite-to-one.
The proof follows Bowen \cite{Bowen-Regional-Conference}, and is
sometimes referred as the {\em diamond argument}, as 
justified by Figure \ref{figure-UH-diamond}. See also \cite[Lemma 6.7]{Adler-Survey}.
Bowen's idea was to investigate the quotient map $\pi:\Sigma\to M$.
For us, this argument will be extremely useful to prove, in Section \ref{Section-MP-NUH}, that the
coding obtained for nonuniformly hyperbolic maps is finite-to-one. For uniformly hyperbolic
maps, this argument is not essential, but we stress that the method is interesting in its own and
introduces a relation, nowadays called {\em Bowen relation}, that precisely characterizes the
loss of injectivity of $\pi$.

\medskip
Consider the triple $(\widehat\Sigma,\widehat\sigma,\widehat\pi)$ as above.
Define a relation in $\mathfs R$ by $R\sim S$ iff $R\cap S\neq\emptyset$.
Assume that $R\sim S$. If $x\in R$ and $y\in S$, let $[x,y]$ be their Smale bracket,
which is well-defined by part (2)(c) of Lemma \ref{Lemma-UH-compatibility}.
Let $\un R=\{R_n\}_{n\in\Z},\un S=\{S_n\}_{n\in\Z}\in\Sigma$.

\medskip
\noindent{\sc Bowen relation:} 
We say that $\un R\approx\un S$ if $R_n\sim S_n$ for all $n\in\Z$.

\medskip
This clearly defines an equivalence relation on $\Sigma$, with
$\widehat\pi(\un R)=\widehat\pi(\un S)$ iff $\un R\approx \un S$. 
Let $N=\#\mathfs R$.

\begin{lemma}\label{Lemma-Bowen-relation}
The following holds for all $\ve>0$ small enough.
\begin{enumerate}[{\rm (1)}]
\item If $R_n\to \cdots\to R_m$ and $S_n\to\cdots\to S_m$ are paths on $\widehat{\mathfs G}$
s.t. $R_n=S_n$, $R_m=S_m$ and $R_k\sim S_k$ for $k=n,\ldots,m$, then $R_k=S_k$ for $k=n,\ldots,m$.
\item $\widehat\pi$ is everywhere at most $N^2$-to-one, i.e. for every $x$ we
have $\#\widehat\pi^{-1}(x)\leq N^2$.
\end{enumerate}
\end{lemma}

\begin{proof}
The original reference is \cite[pp. 13--14]{Bowen-Regional-Conference}.
As mentioned by Bowen himself, part (2) was pointed out by Brian Marcus.
Write $A=R_n=S_n$ and $B=R_m=S_m$. For part (1), choose $x,y$ s.t.
$f^k(x)\in R_k^*$ and $f^k(y)\in S_k^*$ for $k=n,\ldots,m$. Define 
$z$ by the equality $f^n(z)=[f^n(x),f^n(y)]$. Since $R_k\sim S_k$, we have $f^k(z)=[f^k(x),f^k(y)]$ for $k=n,\ldots,m$.
Noting that $f^n(x),f^n(y)\in A^*$ and $f^m(x),f^m(y)\in B^*$, we have that $f^n(z)\in A^*$
and $f^m(z)\in B^*$. Now we use the Markov property:
\begin{enumerate}[$\circ$]
\item The Markov property for the stable direction at $f^n(x)\in A$ implies that $f^k(z)\in W^s(f^k(x),R_k)$
for $k=n,\ldots,m$. Indeed, 
we can prove inductively that $f^k(z)\in (W^s(f^k(x),R_k))^*$, the interior of $W^s(f^k(x),R_k)$
in the relative topology of $R_k$. In particular, $f^k(z)\in R_k^*$ for $k=n,\ldots,m$.
\item Applying the same argument for the unstable direction of $f^m(y)\in B$, we obtain that 
$f^k(z)\in S_k^*$ for $k=n,\ldots,m$.
\end{enumerate}
\begin{figure}[hbt!]
\centering
\def\svgwidth{12.5cm}
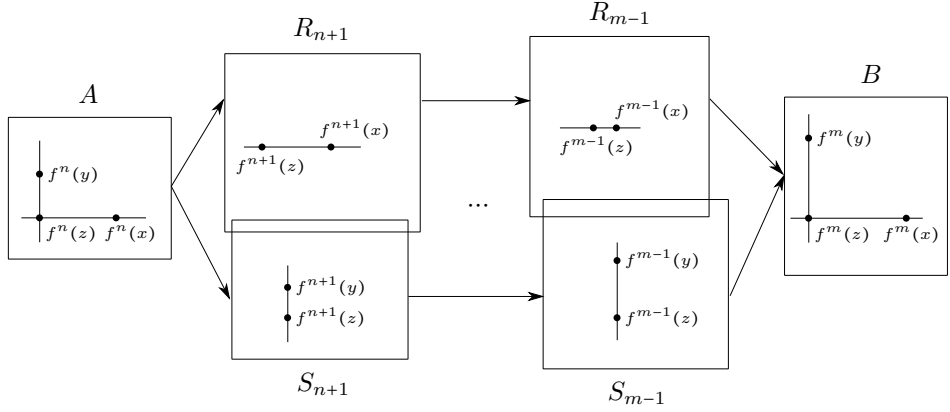
\caption{The diamond argument.}\label{figure-UH-diamond}
\end{figure}
Hence $f^k(z)\in R_k^*\cap S_k^*$ and so $R_k=S_k$, which proves (1). Now we prove (2). If some $x$ has more
than $N^2$ pre-images, then there are two of them, say $\un R$ and $\un S$, and two indices
$n<m$ s.t. $(R_n,\ldots,R_m)\neq (S_n,\ldots,S_m)$ with $R_n=S_n$ and $R_m=S_m$.
Since $f^k(x)\in R_k\cap S_k$, we have $R_k\sim S_k$ for $k=n,\ldots,m$.
This contradicts part (1).
\end{proof}

\section{Symbolic dynamics for nonuniformly hyperbolic systems}\label{Section-MP-NUH}

\medskip
We finally arrive at the core of the discussion, showing how to employ the method of pseudo-orbits
for nonuniformly hyperbolic systems. As mentioned in the introduction, Katok was the first
to use the theory of pseudo-orbits for hyperbolic measures \cite{Katok-IHES}, and constructed the nowadays
called {\em Katok horseshoes}. Restricted to $C^{1+\beta}$ surface diffeomorphisms,
his construction provides finite Markov partitions that approximate the topological entropy. 
For that, he used Pesin theory on subsets where the parameters vary continuously.
This approach is not genuinely nonuniformly hyperbolic,
because it discards regions with bad behavior of such parameters. In this section,
we explain how this difficulty was solved by Sarig \cite{Sarig-JAMS}.
The starting step is to control the hyperbolicity parameters more effectively,
as explained in Section \ref{Section-NUH-systems}. Now, we will use it to develop
a finer theory of pseudo-orbits, that in particular provides symbolic models for nonuniformly hyperbolic
systems. The idea for construction of the symbolic model is similar to Bowen's method
described in Subsection \ref{Subsec-pseudo-orbits}, but instead of Lyapunov charts
we use (double) Pesin charts as vertices of the TMS. In order to code all points with some nonuniform
hyperbolicity, invariably we will need countably many such charts. Hence, while for uniformly hyperbolic
systems the TMS has finitely many states, now it will have countably many.

\medskip
In the sequel, we will restrict ourselves to $C^{1+\beta}$ surface diffeomorphisms.
Later, we explain how to perform the construction in other settings,
which include higher dimensional diffeomorphisms, flows, and billiard maps.
We will emphasize five main ingredients in the proof:
\begin{enumerate}[$\circ$]
\item $\ve$--overlap.
\item $\ve$--double charts.
\item Coarse graining.
\item Improvement lemma.
\item Inverse theorem.
\end{enumerate}
The first two are discussed in the next subsection, and the others in the subsequent subsections.

\subsection{Preliminaries}\label{Subsec-NUH-preliminaries}

Let $M$ be a closed connected
smooth Riemannian surface, let $f:M\to M$ be a $C^{1+\beta}$ diffeomorphism, and
fix $\chi>0$. Recall from Subsection \ref{Subsec-NUH-locus} the definitions of
the nonuniformly hyperbolic locus ${\rm NUH}_\chi$ and Pesin charts $\Psi_x$.
One important part of the construction will be to restrict the domains of Pesin charts, tuning
them properly.
Since we want to end up with countably many of them, we choose their sizes from a countable set. 
Fix $\ve>0$ small enough, and let $I_\ve:=\{e^{-\frac{1}{3}\ve n}:n\geq 0\}$. We redefine $Q(x)$ as below.

\medskip
\noindent
{\sc Parameter $Q(x)$:} For each $x\in{\rm NUH}_\chi$, define $Q(x)$ to be the largest element of $I_\ve$ that
is $\leq \ve^{3/\beta}\|C(f(x))^{-1}\|^{-12/\beta}_{\rm Frob}$.

\medskip
In other words, we truncate $Q(x)$ to $I_\ve$.
Now, define the parameters $q,q^s,q^u$ and the nonuniformly hyperbolic locus ${\rm NUH}_\chi^*$
as in Subsections \ref{Subsec-q} and \ref{Subsec-q^s-q^u}. Observe that
$q,q^s,q^u\in I_\ve$.
To obtain a finite-to-one coding, we need a recurrence assumption on
the parameter $q$, so we define a subset of ${\rm NUH}_\chi^*$ as follows.

\medskip
\noindent
{\sc The nonuniformly hyperbolic locus ${\rm NUH}_\chi^\#$:}
$$
{\rm NUH}_\chi^\#=\left\{x\in {\rm NUH}_\chi^*: \limsup_{n\to+\infty}q(f^n(x))>0\text{ and }
\limsup_{n\to-\infty}q(f^n(x))>0\right\}.
$$

\medskip
It is important to notice that this recurrence assumption is harmless for measures, since an
analogue of Lemma \ref{Lemma-same-measures-1} holds: if $\mu$ is an $f$--invariant probability measure supported
on ${\rm NUH}_\chi^*$, then it is supported on ${\rm NUH}_\chi^\#$. This follows
from the Poincar\'e recurrence theorem (we leave the details to the reader). 
The main result we want to discuss is the following.

\begin{theorem}\label{Sarig}
Let $f:M\to M$ be as above. For every $\chi>0$, there exists a TMS $(\Sigma,\sigma)$
and a H\"older continuous map $\pi:\Sigma\to M$ s.t.:
\begin{enumerate}[$(1)$]
\item $\pi\circ \sigma=f\circ\pi$.
\item $\pi[\Sigma^\#]={\rm NUH}_\chi^\#$.
\item The restriction $\pi\restriction_{\Sigma^\#}:\Sigma^\#\to {\rm NUH}_\chi^\#$ is finite-to-one.
\end{enumerate}
\end{theorem}

Remember the definition of the recurrent set $\Sigma^\#$ in the beginning of Part \ref{Part-2}.
Theorem \ref{Sarig} does not rely on any measure, and instead provides a single
symbolic model that codes all $\chi$--hyperbolic measures at the same time.
Theorem \ref{Sarig} is a strengthening of Theorem 1.3 established by
Sarig in \cite{Sarig-JAMS}. The main difference between Theorem \ref{Sarig} and 
\cite[Thm. 1.3]{Sarig-JAMS} is that Sarig's construction relies on Lyapunov regularity, and as a consequence
he only obtains an inclusion of the form $\pi[\Sigma^\#]\supset {\rm NUH}_\chi^\#$.
But performing the same arguments
of his proof inside the nonuniformly hyperbolic loci we have defined here provides the above
statement. This observation grew from the ongoing work with Buzzi and Crovisier \cite{BCL},
in which we need a more intrinsic construction to make it work for three dimensional flows. 
The proof of Theorem \ref{Sarig} makes essential use of the low dimension of $M$:
since the bundles $E^s,E^u$ are one-dimensional, we are able to apply
arguments of bounded distortion. If $M$ has dimension larger than two, then
$E^s,E^u$ can both have dimension larger than one, and there is no a priori reason for them to satisfy
bounded distortion estimates. Nevertheless, building on his previous work \cite{Ben-Ovadia-2019},
Ben Ovadia was able to obtain a result similar to Theorem \ref{Sarig} that works in any
dimension \cite{Ben-Ovadia-codable}, see more in Subsection \ref{Subsec-NUH-high-dim}.

\subsubsection{$\ve$--overlap of Pesin charts}

We start establishing an analogue of Theorem \ref{Thm-UH-pseudo-orbit}.
In the uniformly hyperbolic situation, the map $x\mapsto \Psi_x$ is continuous,
hence the norm of $\Psi_y^{-1}\circ\Psi_x$ can be controlled
by $d(x,y)$. For nonuniformly hyperbolic systems, the maps $x\in{\rm NUH}_\chi\mapsto e^s_x,e^u_x$
are not necessarily continuous, so even though $d(x,y)\ll 1$ we can still have
$\Psi_{y}^{-1}\circ \Psi_{x}$ with large norm, if the behavior of $C(x)$
and $C(y)$ are very different. Therefore, we only allow overlaps when, in addition 
to taking nearby points, their matrices $C$ are close.
For the definition, we allow Pesin charts to have different domains: for each $\eta\in I_\ve$,
define $\Psi_x^\eta:=\Psi_x\restriction_{[-\eta,\eta]^2}$.

\medskip
\noindent
{\sc $\ve$--overlap:} Two Pesin charts $\Psi_{x_1}^{\eta_1},\Psi_{x_2}^{\eta_2}$ are said to
{\em $\ve$--overlap} if $\tfrac{\eta_1}{\eta_2}=e^{\pm\ve}$ and
$d(x_1,x_2)+\|\widetilde{C(x_1)}-\widetilde{C(x_2)}\|<(\eta_1\eta_2)^4$.
When this happens, we write $\Psi_{x_1}^{\eta_1}\overset{\ve}{\approx}\Psi_{x_2}^{\eta_2}$.

\medskip
This notion was introduced in \cite{Sarig-JAMS}. It constitutes the first main ingredient
in the proof of Theorem \ref{Sarig}. The definition is strong enough to guarantee that the hyperbolicity parameters
of $x_1$ and $x_2$ are almost the same, see \cite[Lemma 3.3]{Sarig-JAMS} or
\cite[Prop. 3.4]{Lima-Matheus}.
Now we are able to recover Theorem \ref{Thm-UH-pseudo-orbit}.

\begin{theorem}\label{Thm-non-linear-Pesin-2}
The following holds for all $\ve>0$ small enough. If $\Psi_{f(x)}^{\eta}\overset{\ve}{\approx}\Psi_{y}^{\eta'}$,
then $f_{x,y}$ is well-defined on $[-10Q(x),10Q(x)]^2$ and can be written as
$f_{x,y}(v_1,v_2)=(Av_1+h_1(v_1,v_2),Bv_2+h_2(v_1,v_2))$ where: 
\begin{enumerate}[{\rm (1)}]
\item $|A|,|B^{-1}|<e^{-\chi}$, cf. Theorem \ref{Thm-non-linear-Pesin}.
\item $\|h_i\|_{1+\beta/3}<\ve$ for $i=1,2$.
\end{enumerate}
If $\Psi_{f^{-1}(y)}^{\eta'}\overset{\ve}{\approx}\Psi_x^{\eta}$ then a
similar statement holds for $f_{x,y}^{-1}$.
\end{theorem}

This is \cite[Proposition 3.4]{Sarig-JAMS}.

\subsubsection{$\ve$--double charts}

Having a good notion of overlap between Pesin charts, now we want to 
define graph transforms. The approach will be similar to Subsection \ref{Subsec-NUH-graph},
when we defined stable and unstable graph transforms using different scales $q^s$ and
$q^u$. In the referred subsection, we also gave a dynamical explanation
of the recursive equations in Lemma \ref{q^s}(2) that $q^s,q^u$ satisfy.
To extend this to Pesin charts, we consider two scales for each chart, one that controls the stable direction
and another  that controls the unstable direction. Hence we {\em do not} work with Pesin charts alone,
but instead consider different objects, called {\em double charts}, and use them to define
stable and unstable graph transforms. This idea, also introduced in \cite{Sarig-JAMS}, is the 
second main ingredient in the proof of Theorem \ref{Sarig}.

\medskip
\noindent
{\sc $\ve$--double chart:} An {\em $\ve$--double chart} is a pair of Pesin charts
$\Psi_x^{p^s,p^u}=(\Psi_x^{p^s},\Psi_x^{p^u})$ where $p^s,p^u\in I_\ve$
with $0<p^s,p^u\leq Q(x)$.

\medskip
Intuitively, just like $q^{s/u}(x)$ are choices for the sizes of local stable/unstable
manifolds at $x$, the parameters $p^s/p^u$ represent candidates for the sizes of
local stable/unstable manifolds of pseudo-orbits. To make sense of this, let us first define
transitions between $\ve$--double charts.

\medskip
\noindent
{\sc Edge $v\overset{\ve}{\rightarrow}w$:} Given $\ve$--double charts $v=\Psi_x^{p^s,p^u}$
and $w=\Psi_y^{q^s,q^u}$, we draw an edge from $v$ to $w$ if the following conditions are
satisfied:
\begin{enumerate}[iii\,]
\item[(GPO1)] $\Psi_{f(x)}^{q^s\wedge q^u}\overset{\ve}{\approx}\Psi_y^{q^s\wedge q^u}$
and $\Psi_{f^{-1}(y)}^{p^s\wedge p^u}\overset{\ve}{\approx}\Psi_x^{p^s\wedge p^u}$.
\item[(GPO2)] $p^s=\min\{e^\ve q^s, Q(x)\}$ and $q^u=\min\{e^\ve p^u, Q(y)\}$.
\end{enumerate}

\medskip
Condition (GPO1) allows to represent $f$ nearby $x$ by Pesin charts
at $x$ and $y$, and similarly for $f^{-1}$.
Condition (GPO2) is a greedy algorithm that chooses the local hyperbolicity parameters
as largest as possible, and is the counterpart of Lemma \ref{q^s}(2) for pseudo-orbits.

\medskip
\noindent
{\sc $\ve$--generalized pseudo-orbit ($\ve$--gpo):} An {\em $\ve$--generalized pseudo-orbit}
is a sequence $\un{v}=\{\Psi_{x_n}^{p^s_n,p^u_n}\}_{n\in\Z}$ of $\ve$--double charts
s.t. $\Psi_{x_n}^{p^s_n,p^u_n}\overset{\ve}{\rightarrow}\Psi_{x_{n+1}}^{p^s_{n+1},p^u_{n+1}}$
for all $n\in\Z$.

\medskip
This definition is much stronger than the one given in Subsection \ref{Subsec-pseudo-orbits}.
Observe that if $\un v$ is an $\ve$--gpo then by (GPO2) we have that
$$
p^s_0=\inf\{e^{\ve n}Q(x_n):n\geq 0\} \ \text{ and }\ p^u_0=\inf\{e^{\ve n}Q(x_{-n}):n\geq 0\}.
$$
These equations are very similar to the definitions of $q^{s/u}$, see Subsection \ref{Subsec-q^s-q^u}.

\subsubsection{Graph transforms}

To finally define graph transforms, it remains to strengthen the notion of admissibility.
Let $v=\Psi_x^{p^s,p^u}$ be an $\ve$--double chart.

\medskip
\noindent
{\sc Admissible manifolds:} An {\em $s$--admissible manifold} at $v$ is a set
of the form $V^s=\Psi_x\{(t,F(t)):|t|\leq p^s\}$, where $F:[-p^s,p^s]\to\R$ is a $C^{1+\beta/3}$ function s.t.:
\begin{enumerate}
\item[(AM1)] $|F(0)|\leq 10^{-3}(p^s\wedge p^u)$.
\item[(AM2)] $|F'(0)|\leq \tfrac{1}{2}(p^s\wedge p^u)^{\beta/3}$.
\item[(AM3)] $\|F'\|_{C^0}+{\rm Hol}_{\beta/3}(F')\leq\tfrac{1}{2}$ where the norms are taken in $[-p^s,p^s]$.
\end{enumerate}
Similarly, a {\em $u$--admissible manifold at $v$} is a set
of the form $V^u=\Psi_x\{(G(t),t):|t|\leq p^u\}$ where $G:[-p^u,p^u]\to\R$ is a $C^{1+\beta/3}$ function
satisfying (AM1)--(AM3), where the norms are taken in $[-p^u,p^u]$.

\medskip
Note that $p^{s/u}$ control the domains of the representing functions, and $p^s\wedge p^u$ controls
their behaviour at 0. Let $\mathfs M^s_v,\mathfs M^u_v$ be the space of all $s,u$--admissible manifolds at
$v$, which are metric spaces with the same metrics as before. For each edge $v\overset{\ve}{\to}w$, 
define the stable graph transform $\mathfs F^s_{v,w}:\mathfs M^s_w\to\mathfs M^s_v$ and
the unstable graph transform $\mathfs F^u_{v,w}:\mathfs M^u_v\to\mathfs M^u_w$ as before.
An analogue of Theorem \ref{Thm-UH-graph-pseudo-orbit} holds, and we
can similarly define stable and unstable manifolds for every 
$\ve$--gpo $\un v$.

\medskip
\noindent
{\sc Stable/unstable manifolds:} The {\em stable manifold} of an $\ve$--gpo $\un v=\{v_n\}_{n\in\Z}$
is the unique $s$--admissible manifold
$V^s[\un v]\in \mathfs M^s_{v_0}$ defined by
$$
V^s[\un v]:=\lim_{n\to\infty}(\mathfs F^s_{v_0,v_1}\circ\cdots\circ \mathfs F^s_{v_{n-1},v_n})[V_n]
$$
for some (any) sequence $\{V_n\}_{n\geq 0}$ with $V_n\in\mathfs M^s_{v_n}$.
The {\em unstable manifold} of $\un v$ is the unique $u$--admissible manifold
$V^u[\un v]\in \mathfs M^u_{v_0}$ defined by
$$
V^u[\un v]:=\lim_{n\to-\infty}(\mathfs F^u_{v_{-1},v_0}\circ\cdots\circ \mathfs F^u_{v_n,v_{n+1}})[V_n]
$$
for some (any) sequence $\{V_n\}_{n\leq 0}$ with $V_n\in\mathfs M^u_{v_n}$.

\medskip
These manifolds are genuine Pesin invariant manifolds,
see \cite[Prop. 6.3]{Sarig-JAMS}. In particular, if $y,z\in V^s[\un v]$ then $\tfrac{s(y)}{s(z)}=e^{\pm\text{const}}$,
and similarly for $V^u[\un v]$.
We point out that if $F$ is the representing function of
$V^s[\un v]$ and $F_n$ is the representing function of 
$(\mathfs F^s_{v_0,v_1}\circ\cdots\circ \mathfs F^s_{v_{n-1},v_n})[V_n]$ for $n\geq 0$, then
$\|F-F_n\|_{C^1}\to 0$ as $n\to\infty$, and the same holds for the representing function of $V^u[\un v]$.
This follows from the Arzel\`a-Ascoli theorem, since the $C^{1+\beta/3}$ norm of
representing functions is uniformly bounded, see e.g. part (2) in the proof of \cite[Prop. 4.15]{Sarig-JAMS}.
We also define shadowing.

\medskip
\noindent
{\sc Shadowing:} We say that $\un v=\{\Psi_{x_n}^{p^s_n,p^u_n}\}_{n\in\Z}$ {\em shadows} $x$
if $f^n(x)\in \Psi_{x_n}([-p^s_n\wedge p^u_n,p^s_n\wedge p^u_n]^2)$ for all $n\in\Z$.

\medskip
The Shadowing Lemma is still valid, again with $\{x\}=V^s[\un v]\cap V^u[\un v]$,
see \cite[Thm 4.16(1)]{Sarig-JAMS}.

\subsection{Coarse graining}

The third main ingredient in the proof of Theorem \ref{Sarig} consists
on choosing countably many charts that shadow all orbits of interest.
For uniformly hyperbolic systems, a sufficiently dense set of points is enough.
For nonuniformly hyperbolic systems, the construction is more elaborate. Firstly, the definition of
$\ve$--overlap depends on $\|\widetilde{C(x)}-\widetilde{C(y)}\|$, which depends
on a comparison between $s(x),u(x),\alpha(x)$ and $s(y),u(y),\alpha(y)$.
Secondly, nearest neighbor conditions do not follow from control at $x$ and $y$.
Fortunately, all parameters involved in the construction belong to a precompact set, so
there are countable dense subsets.

\begin{theorem}\label{Thm-coarse-graining}
For all $\ve>0$ sufficiently small, there exists a countable set $\mathfs A$ of $\ve$--double charts
with the following properties:
\begin{enumerate}[{\rm (1)}]
\item {\sc Discreteness}: For all $t>0$, the set $\{\Psi_x^{p^s,p^u}\in\mathfs A:p^s,p^u>t\}$ is finite.
\item {\sc Sufficiency:} If $x\in {\rm NUH}_\chi^\#$ then there is an $\ve$--gpo $\un v\in{\mathfs A}^{\Z}$
that shadows $x$.
\item {\sc Relevance:} For all $v\in \mathfs A$ there is an $\ve$--gpo $\un{v}\in\mathfs A^\Z$
with $v_0=v$ that shadows a point in ${\rm NUH}_\chi^\#$.
\end{enumerate}
\end{theorem}

The original statement is \cite[Thm 4.16]{Sarig-JAMS}.
The discreteness property (1) says that Pesin blocks require only finitely many $\ve$--double charts, while the relevance property (3) guarantees that none of the $\ve$--double charts is redundant.

\begin{proof}[Sketch of proof.]
Define $X:=M^3\times {\rm GL}(2,\R)^3\times (0,1]$.
For each $x\in{\rm NUH}_\chi^*$, let $\Gamma(x)=(\un x,\un C,\un Q)\in X$ with
\begin{align*}
\un x=(f^{-1}(x),x,f(x)),\ \un C=(C(f^{-1}(x)),C(x),C(f(x))),\ \un Q=Q(x).
\end{align*}
Let $Y=\{\Gamma(x):x\in{\rm NUH}_\chi^*\}$. For each triple
$\un{\ell}=(\ell_{-1},\ell_0,\ell_1)\in\N_0^3$, define
$$
Y_{\un \ell}:=\left\{\Gamma(x)\in Y:e^{\ell_i}\leq\|C(f^i(x))^{-1}\|<e^{\ell_i+1},-1\leq i\leq 1
\right\}.
$$
Then $Y=\bigcup_{\un\ell\in\N_0^3}Y_{\un\ell}$, and each
$Y_{\un\ell}$ is precompact in $X$ by definition. For each $j\geq 0$, choose a finite set
$Y_{\un\ell}(j)\subset Y_{\un\ell}$ s.t. for every $\Gamma(x)\in Y_{\un\ell}$
there exists $\Gamma(y)\in Y_{\un\ell}(j)$ s.t.:
\begin{enumerate}[{\rm (a)}]
\item $ d(f^i(x),f^i(y))+\|\widetilde{C(f^i(x))}-\widetilde{C(f^i(y))}\|<e^{-8(j+2)}$ for $-1\leq i\leq 1$.
\item $\tfrac{Q(x)}{Q(y)}=e^{\pm \ve/3}$.
\end{enumerate}
We define the countable set of $\ve$--double charts as follows.

\medskip
\noindent
{\sc The alphabet $\mathfs A$:} Let $\mathfs A$ be the countable family of $\Psi_x^{p^s,p^u}$ s.t.:
\begin{enumerate}[i i)]
\item[(CG1)] $\Gamma(x)\in Y_{\un\ell}(j)$ for some
$(\un\ell,j)\in\N_0^3\times \N_0$.
\item[(CG2)] $0<p^s,p^u\leq Q(x)$ and $p^s,p^u\in I_\ve$.
\item[(CG3)] $e^{-j-2}\leq p^s\wedge p^u\leq e^{-j+2}$.
\end{enumerate}

\medskip
This alphabet satisfies (1) and (2) but not necessarily (3).
We can easily reduce it to a sub-alphabet $\mathfs A'$ satisfying (1)--(3) as follows.
Call $v\in\mathfs A$ relevant if there is $\un v\in\mathfs A^\Z$ with $v_0=v$ s.t. $\un{v}$ shadows
a point in ${\rm NUH}_\chi^\#$. Since ${\rm NUH}_\chi^\#$ is $f$--invariant, every $v_i$ is relevant.
Then $\mathfs A'=\{v\in\mathfs A:v\text{ is relevant}\}$ is discrete
because $\mathfs A'\subset\mathfs A$, it is sufficient and relevant by definition.
\end{proof}

Referring to the steps of Subsection \ref{Subsub-MP-pseudo-orbits},
we have just completed Step 1. Now let $\mathfs G=(V,E)$ be the oriented graph with vertex set
$V=\mathfs A$ and
edge set $E=\{v\overset{\ve}{\to}w\}$, and 
let $(\Sigma,\sigma)$ be the TMS generated by $\mathfs G$.
The proof of sufficiency actually gives more:
if $x\in {\rm NUH}_\chi^\#$ then there is a {\em recurrent} $\ve$--gpo $\un v\in\Sigma^\#$
that shadows $x$. By the Shadowing Lemma, $\pi:\Sigma\to M$ defined by $\{\pi(\un v)\}:=V^s[\un v]\cap V^u[\un v]$
is an infinite-to-one extension of $f$ s.t. $\pi[\Sigma^\#]\supset {\rm NUH}_\chi^\#$.

\subsection{Improvement lemma}\label{Subsec-Improvement}

The fourth main ingredient in the proof of Theorem \ref{Sarig}
is an important fact, that will imply the reverse inclusion $\pi[\Sigma^\#]\subset{\rm NUH}_\chi^\#$
as well as the Inverse Theorem in the next subsection.
We start observing that points of $\pi[\Sigma^\#]$ do have stable and unstable directions. 
If $\un v$ is an $\ve$--gpo, then for $x\in V^u[\un v]$ we can take 
$e^u_x$ to be (any of) the tangent vector to $V^u[\un v]$ at $x$.
This direction indeed contracts in the past and expands in the future,
see \cite[Prop. 6.3(2)]{Sarig-JAMS} and claim (i) in the proof of \cite[Proposition 6.5]{Sarig-JAMS}.
Similarly, for $x\in V^s[\un v]$ we can take (any of) the tangent vector to $V^s[\un v]$ at $x$
to be the stable direction $e^s_x$. 

\medskip
Fix $\un v\in\Sigma^\#$, and let $x=\pi(\un v)$.
Since $\{x\}=V^s[\un v]\cap V^u[\un v]$, $e^s_x,e^u_x$ are defined. 
If we prove that $s(x),u(x)<\infty$, then (NUH3) holds and automatically (NUH1)--(NUH2) hold as well,
implying that $x\in{\rm NUH}_\chi$. The proof that $s(x),u(x)<\infty$ is very delicate,
since a priori
there is no reason for $x$ to have a hyperbolic behaviour as good as the behaviour of the 
centers of the $\ve$--double charts of $\un v$. But the notions of $\ve$--overlap
and admissibility are so strong that indeed $s(x),u(x)<\infty$.
The proof of this fact relies on a general philosophy that $f$ improves smoothness
along the unstable direction, and $f^{-1}$ improves smoothness along the stable direction.
In terms of graph transforms, the ratios of $s,u$ parameters improve. We call this an {\em improvement lemma}.
The heuristics for such improvement can be explained as follows.
Assume that $v\overset{\ve}{\to}w$, where $v=\Psi_{x_0}^{p^s_0,p^u_0}$
and $w=\Psi_{x_1}^{p^s_1,p^u_1}$, let $V^s\in\mathfs M^s_w$ and $\widetilde V^s=\mathfs F^s_{v,w}[V^s]$,
and fix a point $y\in \widetilde V^s$. 
In particular $f(y)\in V^s$, see Figure \ref{figure-improvement}.
\begin{figure}[hbt!]
\centering
\def\svgwidth{12cm}
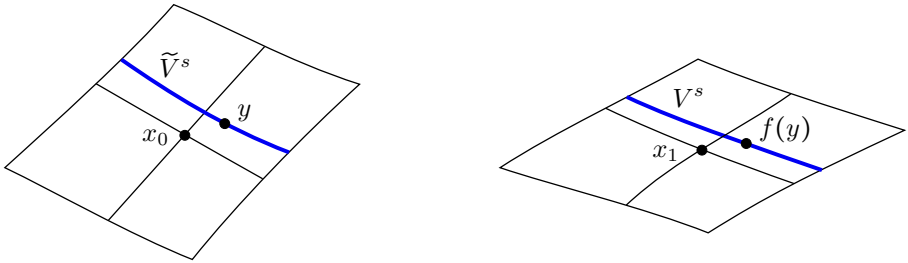
\caption{Improvement Lemma: the graph transform improves ratios.}\label{figure-improvement}
\end{figure}
Assuming that $s(y)<\infty$, we want to compare the ratios $\tfrac{s(f(y))}{s(x_1)}$ and $\tfrac{s(y)}{s(x_0)}$.
Proceeding as in the proof of Lemma \ref{Lemma-linear-reduction},
we have $s(y)^2=2+Cs(f(y))^2$, where $C=\|df e^s_y\|^2 e^{2\chi}$. By the $\ve$--overlap,
we also have $s(x_0)^2\approx 2+Cs(x_1)^2$, and so
$\tfrac{s(y)^2}{s(x_0)^2}\approx\tfrac{2+Cs(f(y))^2}{2+Cs(x_1)^2}$. If the initial ratio
is $\tfrac{s(f(y))^2}{s(x_1)^2}=K\gg 1$, then the new ratio is 
$\tfrac{s(y)^2}{s(x_0)^2}\approx\tfrac{2+KCs(x_1)^2}{2+Cs(x_1)^2}<K$.
The same occurs if $K\ll 1$, in which case the new ratio becomes $>K$.

\begin{lemma}[Improvement Lemma]
The following holds for $\ve>0$ small enough. For $\xi\geq\sqrt{\ve}$, if
$\tfrac{s(f(y))}{s(x_1)}=e^{\pm\xi}$, then $\tfrac{s(y)}{s(x_0)}=e^{\pm(\xi-Q(x_0)^{\beta/4})}$. 
\end{lemma}

Hence the ratio improves whenever it is outside $[e^{-\sqrt{\ve}},e^{\sqrt{\ve}}]$,
see \cite[Lemma 7.2]{Sarig-JAMS} for a proof. Here is the first main consequence of this important lemma.

\begin{corollary}\label{Corollary-finite-s}
$\pi[\Sigma^\#]\subset{\rm NUH}_\chi$, i.e. if $x=\pi(\un v)$ for $\un v\in\Sigma^\#$, then $s(x),u(x)<+\infty$.\end{corollary}

The proof can be found in \cite[\S 7.1]{Sarig-JAMS}. We summarize it as follows.
Let $v_{n_k}=v$ for infinitely many $n_k>0$. Since $v=\Psi_z^{p^s,p^u}$ is relevant, there is
$V^s\in \mathfs M^s_v$ s.t. $s(y)<\infty$ for every $y\in V^s$. Starting with $V_{n_k}=V^s$,
pull it back through the graph transforms to get $\widetilde V^s_k\in\mathfs M^s_{v_0}$.
If $k$ is large enough, then the original ratio $\tfrac{s(y)}{s(z)}$ passed through sufficiently many
improvements $Q(z)$ so that $\tfrac{s(w_k)}{s(x_0)}=e^{\pm \xi}$ for all
$w_k\in\widetilde V^s_k$, for some fixed $\xi\geq\sqrt{\ve}$.
Since the representing functions of $\widetilde V^s_k$ converge to the representing
function of $V^s[\un v]$ in the $C^1$ norm, we conclude that $s(x)<\infty$. Similarly, 
we obtain that $u(x)<\infty$.

\subsection{Inverse theorem}\label{Subsec-Inverse-Thm}

The fifth and final main ingredient in the proof of Theorem \ref{Sarig} is the Inverse Theorem.
To understand its importance, recall that once we have constructed 
an infinite-to-one coding $\pi$, the next step is to  
apply a Bowen-Sina{\u\i} refinement to the Markov cover $\mathfs Z$ induced by $\pi$.
Since $\Sigma$ has countably many states,
the Markov cover is countable, so that after refining the resulting partition could be
uncountable\footnote{For example, the refinement of the dyadic
intervals in $[0,1]$ is the point partition.}. One condition that guarantees an yet countable
refinement is {\em local finiteness:} $\mathfs Z$ is locally finite if
every $Z\in\mathfs Z$ only intersects finitely many others $Z'\in\mathfs Z$.
The understanding of intersections $Z\cap Z'$ comes from an {\em inverse problem}:
if $\pi(\un v)=x$, how is $\un v$ defined in terms of $x$?
The next theorem (essentially) answers this question.

\begin{theorem}[Inverse Theorem]\label{Thm-inverse}
Let $\un v=\{\Psi_{x_n}^{p^s_n,p^u_n}\}_{n\in\Z}\in\Sigma^\#$, and let
$\pi(\un v)=x$. Then the following holds for all $n\in\Z$:
\begin{enumerate}[{\rm (1)}]
\item {\sc Control of $x$:} ${\rm dist}(x_n,f^n(x))<{\rm const}$.
\item {\sc Control of $\alpha$:} $\tfrac{\sin\alpha(x_n)}{\sin\alpha(f^n(x))}=e^{\pm{\rm const}}$.
\item {\sc Control of $s,u$:} $\tfrac{s(x_n)}{s(f^n(x))}=\tfrac{u(x_n)}{u(f^n(x))}=e^{\pm{\rm const}}$.
\item {\sc Control of $p^s,p^u$:} $\tfrac{p^s_n}{q^s(f^n(x))}=\tfrac{p^u_n}{q^u(f^n(x))}=e^{\pm{\rm const}}$.
\end{enumerate}
In particular $x\in{\rm NUH}_\chi^\#$, and so $\pi[\Sigma^\#]={\rm NUH}_\chi^\#$.
\end{theorem}

The above theorem is not the original reference \cite[\S 6]{Sarig-JAMS}, since it did not
make use of $q,q^s,q^u$.
Instead, the original statement considered two $\ve$--gpo's $\un v,\un w\in\Sigma^\#$ and compared
their parameters directly. The constant appearing in the theorem
is of the order of $\sqrt[3]{\ve}$.
The assumption $\un v\in\Sigma^\#$ is essential to guarantee parts (3) and (4), since the proof
uses that the trajectories visit a Pesin block infinitely often.
Theorem \ref{Thm-inverse} states that each coordinate of $\un v$
is uniquely defined ``up to bounded error''.
Below we explain how to get the estimates for $n=0$.

\medskip
\noindent
{\sc Control of $x$ and $\alpha$.} Let $F,G$ be the representing functions of $V^s,V^u$. Since $F(0),G(0)\approx 0$, the graphs of $F,G$
intersect close to the origin $(0,0)$. Applying $\Psi_{x_0}$, we conclude that
$d(x_0,x)\ll 1$.
To control the angle, we use that $\|F'\|_{C^0},\|G'\|_{C^0}\ll  1$ and so the graphs of $F,G$ 
intersect almost perpendicularly. Applying $\Psi_{x_0}$, we get that
$\alpha(x_0)\approx \alpha(x)$, see Figure \ref{figure-IT-angle}.
\begin{figure}[hbt!]
\centering
\def\svgwidth{8cm}
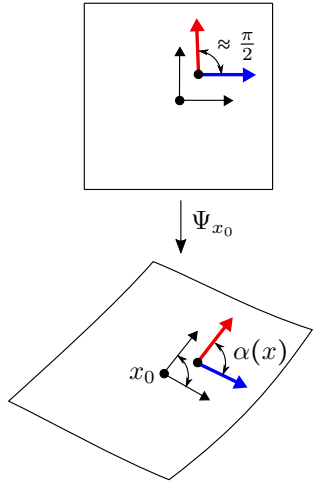
\caption{Control of $x$ and $\alpha$.}\label{figure-IT-angle}
\end{figure}

\medskip
\noindent
{\sc Control of $s,u$.} The proof is identical to the proof of Corollary \ref{Corollary-finite-s}.

%

\medskip
Observe that parts (2)--(3) imply that $\tfrac{Q(x_n)}{Q(f^n(x))}=e^{\pm{\rm const}}$ for all $n\in\Z$.

\medskip
\noindent
{\sc Control of $p^s,p^u$.} We prove the first estimate $\tfrac{p^s_0}{q^s(x)}=e^{\pm{\rm const}}$
(the second is analogous). The idea is to observe that $p^s_0$ and $q^s(x)$
are defined as infima of comparable sequences. Indeed, by definition we have
$$
q^s(x)=\inf\{e^{\ve n}Q(f^n(x)):n\geq 0\}
$$
and by (GPO2) we have
$$
p^s_0=\inf\\\{e^{\ve n}Q(x_n):n\geq 0\}.
$$
Since $\tfrac{Q(x_n)}{Q(f^n(x))}=e^{\pm{\rm const}}$ for all $n\geq 0$,
it follows that $\tfrac{p^s_0}{q^s(x)}=e^{\pm{\rm const}}$.

\medskip
The above argument differs from the original one, and does not require any maximality
assumption on $q^s$. Indeed, it comes for free from the recurrence of $\un v$.

%
%
%

\subsection{Bowen-Sina{\u\i} refinement}

Local finiteness allows us to refine the Markov cover of ${\rm NUH}_\chi^\#$
and obtain a countable Markov partition. As mentioned in Subsection \ref{Subsec-UH-MP},
the notion of Markov cover/partition for nonuniformly hyperbolic systems will be weaker than
the one for uniformly hyperbolic systems, since we are not able to control the topology of the rectangles.
Indeed, due to Theorem \ref{Thm-inverse}, the Markov cover is the projection of
the canonical partition of $\Sigma^\#$ (instead of $\Sigma$) into cylinders at the zeroth position,
hence rectangles will not have the regularity property.
This loss of topological control compels us to perform the refinement in a more abstract way,
and then check that the resulting partition still generates a finite-to-one symbolic extension.
The reader who is already comfortable with the
uniformly hyperbolic situation will not see much difficulty in the adaptations. 

\medskip
\noindent
{\sc The Markov cover $\mathfs Z$:} Let $\mathfs Z:=\{Z(v):v\in\mathfs A\}$, where
$$
Z(v):=\{\pi(\un v):\un v\in\Sigma^\#\text{ and }v_0=v\}.
$$

\medskip
Restricted to $\Sigma^\#$, cylinders are neither closed nor empty, so
the same occurs to $Z(v)$. Nevertheless,
$s$/$u$--fibres are still well-defined in $\mathfs Z$.

\begin{proposition}\label{Prop-NUH-Z}
The following are true.
\begin{enumerate}[{\rm (1)}]
\item {\sc Covering property:} $\mathfs Z$ is a cover of ${\rm NUH}_\chi^\#$.
\item {\sc Product structure:} For every $Z\in\mathfs Z$ and every $x,y\in Z$, the intersection
$W^s(x,Z)\cap W^u(y,Z)$ consists of a single point, and this point belongs to $Z$.
\item {\sc Symbolic Markov property:} If $x=\pi(\un v)$ with $\un v=\{v_n\}_{n\in\Z}\in\Sigma^\#$, then
$$
f(W^s(x,Z(v_0)))\subset W^s(f(x),Z(v_1))\, \text{ and }\, f^{-1}(W^u(f(x),Z(v_1)))\subset W^u(x,Z(v_0)).
$$
\item {\sc Local finiteness:} For every $Z\in\mathfs Z$, the set $\{Z'\in\mathfs Z:Z\cap Z'\neq\emptyset\}$
is finite.
\end{enumerate}
\end{proposition}

Parts (1)--(3) are analogues of Proposition \ref{Prop-UH-Z}. The novelty is part (4),
which follows from Theorem \ref{Thm-coarse-graining}(1) and Theorem \ref{Thm-inverse}(3):
if $v=\Psi_x^{p^s,p^u}$ and $w=\Psi_y^{q^s,q^u}$ satisfy $Z(v)\cap Z(w)\neq\emptyset$
then $\tfrac{p^s}{q^s}=\tfrac{p^u}{q^u}=e^{\pm{\rm const}}$, hence
$$
\#\{Z(w)\in\mathfs Z:Z(v)\cap Z(w)\neq\emptyset\}\leq \#\{w\in\mathfs A:q^s,q^u\geq e^{-{\rm const}}(p^s\wedge p^u)\}<\infty.
$$  
Inside each $Z\in\mathfs Z$, define the Smale bracket $[\cdot,\cdot]_Z$ as before. 
Lemma \ref{Lemma-UH-compatibility} remains valid, see \cite[Lemmas 10.7, 10.8, 10.10]{Sarig-JAMS}.
We now refine $\mathfs Z$. For $Z,Z'\in\mathfs Z$ s.t. $Z\cap Z'\neq\emptyset$, 
let $\mathfs E_{ZZ'}=$ cover of $Z$ by rectangles:
\begin{align*}
E^{su}_{Z,Z'}&=\{x\in Z: W^s(x,Z)\cap Z'\neq\emptyset,
W^u(x,Z)\cap Z'\neq\emptyset\}\\
E^{s\emptyset}_{Z,Z'}&=\{x\in Z: W^s(x,Z)\cap Z'\neq\emptyset,
W^u(x,Z)\cap Z'=\emptyset\}\\
E^{\emptyset u}_{Z,Z'}&=\{x\in Z: W^s(x,Z)\cap Z'=\emptyset,
W^u(x,Z)\cap Z'\neq\emptyset\}\\
E^{\emptyset\emptyset}_{Z,Z'}&=\{x\in Z: W^s(x,Z)\cap Z'=\emptyset,
W^u(x,Z)\cap Z'=\emptyset\}.
\end{align*}
The above definition is simpler than the one for uniformly hyperbolic systems, since
we do not take relative interiors nor closures.
Let $\mathfs R$ be the partition that refines all of $\mathfs E_{ZZ'}$.
Again due to Theorem \ref{Thm-coarse-graining}(1) and Theorem \ref{Thm-inverse}(3),
$\mathfs R$ and $\mathfs Z$ satisfy two additional local finiteness properties:
\begin{enumerate}[$\circ$]
\item For all $R\in\mathfs R$, the set $\{Z\in\mathfs Z:Z\supset R\}$ is finite.
\item For all $Z\in\mathfs Z$, the set $\{R\in\mathfs R:R\subset Z\}$ is finite.
\end{enumerate}
Inside each $R\in\mathfs R$, the Smale brackets $[\cdot,\cdot]_Z$ do not depend on $Z$, 
hence we can define $[\cdot,\cdot]$ on $R$.

\begin{lemma}\label{Lemma-MP-NUH}
$\mathfs R$ is a Markov partition:
\begin{enumerate}[{\rm (1)}]
\item {\sc Product structure:} If $x,y\in R\in\mathfs R$ then $[x,y]\in R$.
\item {\sc Markov property:} if $R,S\in\mathfs R$ and if $x\in R,f(x)\in S$ then
$$
f(W^s(x,R))\subset W^s(f(x),S)\text{ and }f^{-1}(W^u(f(x),S))\subset W^u(x,R).
$$
\end{enumerate}
\end{lemma}
 

Let $\widehat{\mathfs G}=(\widehat V,\widehat E)$ be the graph with $\widehat V=\mathfs R$ and
$\widehat E=\{R\to S:f(R)\cap S\neq\emptyset\}$ (compare this definition with the one
given in Step 3 of Subsection \ref{Subsub-MP-pseudo-orbits}).
Let $(\widehat\Sigma,\widehat\sigma)$ be the TMS defined by $\widehat{\mathfs G}$,
and define $\widehat\pi:\Sigma\to M$ by
$$
\{\widehat\pi(\underline{R})\}:=\bigcap_{n\geq 0}\overline{f^n(R_{-n})\cap\cdots\cap f^{-n}(R_n)}.
$$
In comparison to the previous constructions, we take closures because the $R_n$'s are not
necessarily closed. A priori, the image of $\widehat{\pi}$ could be much bigger than
the image of $\pi$. Fortunately, this is not the case: for each $\un R\in\widehat\Sigma^\#$,
there is an $\ve$--gpo $\un v\in\Sigma^\#$ s.t. $\widehat\pi(\un R)=\pi(\un v)$, see the proof
of \cite[Thm 12.5]{Sarig-JAMS}. Therefore $\widehat\pi[\widehat\Sigma^\#]=\pi[\Sigma^\#]={\rm NUH}_\chi^\#$.

\medskip
We point out that $\widehat\pi$
is compatible with the Smale brackets in $\Sigma$ and $\mathfs R$.
More specifically, for $\un R=\{R_n\}_{n\in\Z},\un S=\{S_n\}_{n\in\Z}\in\Sigma$ with $R_0=S_0$,
let $\un U=[\un R,\un S]$ where $\un U=\{U_n\}_{n\in\Z}$ is defined by
$$
U_n=
\left\{
\begin{array}{rl}
R_n& \text{, if }n\geq 0\\
S_n& \text{, if }n\leq 0.\\
\end{array}
\right.
$$
Then $\widehat\pi([\un R,\un S])=[\widehat\pi(\un R),\widehat\pi(\un S)]$.
This is \cite[Lemma 4.4]{BLPV}, and it is used to study the simplicity of generic
fiber-bunched cocycles over nonuniformly hyperbolic diffeomorphisms.

\subsection{Affiliation and Bowen relation}

We investigate how $\widehat\pi:\widehat\Sigma^\#\to{\rm NUH}_\chi^\#$
loses injectivity.
In the uniformly hyperbolic situation, we saw in Subsection \ref{Subsub-Bowen-relation}
that this is characterized by the Bowen relation. Sarig was able to obtain a similar
characterization \cite{Sarig-JAMS}, which was further explored
by Boyle and Buzzi \cite{Boyle-Buzzi}.

\medskip
\noindent
{\sc Affiliation:} Two rectangles $R,S\in\mathfs R$ are called {\em affiliated} if
there exist $Z,Z'\in\mathfs Z$ s.t. $Z\supset R$, $Z'\supset S$ and $Z\cap Z'\neq\emptyset$.
If  this occurs, we write $R\sim S$. See Figure \ref{figure-affiliation}.

\begin{figure}[hbt!]
\centering
\def\svgwidth{5cm}
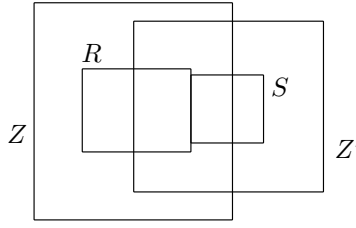
\caption{The affiliation property $R\sim S$: it might occur even when $R\cap S=\emptyset$.}\label{figure-affiliation}
\end{figure}

\medskip
Affiliation is more complicated than mere nonempty intersection, and it arises 
from the need of taking closures in the definition of $\widehat\pi$.
If $R\sim S$ as above, then Lemma \ref{Lemma-UH-compatibility} implies that
we can take Smale brackets between points of $Z$ and $Z'$.

\medskip
\noindent{\sc Bowen relation:} 
We say that $\un R\approx\un S$ iff $R_n\sim S_n$ for all $n\in\Z$.

\medskip
The following result was implicit in \cite{Sarig-JAMS}, as explained
in \cite[\S 8.3]{Boyle-Buzzi}.

\begin{lemma}
If $\un R,\un S\in\widehat\Sigma^\#$, then $\widehat\pi(\un R)=\widehat\pi(\un S)$ iff
$\un R\approx\un S$.
\end{lemma}

Now we apply the diamond argument, as in Subsection \ref{Subsub-Bowen-relation},
by choosing $n<m$ s.t. the rectangle configuration $(R,Z,S,Z')$
of Figure \ref{figure-affiliation} is that same at positions $n$ and $m$.
Introduce
$$
N(R):=\#\{(S,Z')\in\mathfs R\times\mathfs Z:R\sim S\text{ and } Z'\supset S\}.
$$
The local finiteness of $\mathfs R$ and $\mathfs Z$ imply that $N(R)<\infty$ for all $R\in\mathfs R$.
For fixed $R,S\in\mathfs R$ and $n<m$, the pigeonhole principle implies that if there are more than $N(R)N(S)$
paths $R_n\to\cdots\to R_m$ s.t. $R_n\sim R$ and $R_m\sim S$, then two of them have the same 
rectangle configuration at iterates $n$ and $m$, as expressed below.

\begin{theorem}
Let $\un R\in\widehat\Sigma^\#$, with $R_n=R$ for infinitely many $n<0$ and $R_m=S$ for infinitely
many $m>0$, and let $x=\widehat\pi(\un R)$. Then $\#[\widehat\pi^{-1}(x)\cap \widehat\Sigma^\#]\leq N(R)N(S)$.
 In particular, the restriction
$\widehat\pi\restriction_{\widehat\Sigma^\#}:\widehat\Sigma^\#\to{\rm NUH}_\chi^\#$ is finite-to-one.
\end{theorem}

The original proof \cite[Thm 12.8]{Sarig-JAMS} has a small error that was corrected
in \cite{Lima-Sarig}. This concludes the proof of Theorem \ref{Sarig}.

\subsection{Higher dimensions}\label{Subsec-NUH-high-dim}

Recently,  building on his previous work \cite{Ben-Ovadia-2019},
Ben Ovadia obtained a higher dimensional version of Theorem \ref{Sarig} \cite{Ben-Ovadia-codable}.
Regarding the five main ingredients, the first three can be adapted to higher dimensions,
with many extra technical difficulties, see \cite[Sections 2 and 3]{Ben-Ovadia-2019}.
But the improvement lemma and inverse theorem require
different approaches. Indeed, each $x=\pi(\un v)$ has many parameters
$S(x,v)$ and $U(x,w)$, and in principle it is not clear to which reference
parameters they should be compared with. Let us explain how to do the estimates
for $S(x,v)$. Let $V^s[\un v]=\Psi_{x_0}\{(t,F(t)):t\in[-p^s_0,p^s_0]^{d_s}\}$ be the stable
manifold of $\un v=\{\Psi_{x_0}^{p^s_n,p^u_n}\}_{n\in\Z}\in\Sigma^\#$.
Let $x\in V^s[\un v]$, with $x=\Psi_{x_0}(\overline{x})=\Psi_{x_0}(\overline t,F(\overline t))$.
The candidate for stable subspace at $x$ is $\widetilde E^s_x=(d\Psi_{x_0})_{\overline x}\left[H^s_{\overline x}\right]$,
where $H^s_{\overline x}$ is the subspace
tangent to $\{(t,F(t)):t\in[-p^s_0,p^s_0]^{d_s}\}$ at $\overline{x}$. Observing that the stable subspace
at $x_0$ is $(d\Psi_{x_0})_0[\R^{d_s}\times\{0\}]$, consider a linear transformation
$H^s_{\overline x}\to \R^{d_s}\times\{0\}$ given by the derivative of the projection
$(t,F(t))\mapsto (t,0)$ at $\overline{x}=(\overline t,F(\overline t))$. Applying $d\Psi_{x_0}$,
we obtain a linear transformation $\Xi_x:\widetilde E^s_x\to E^s_{x_0}$.
Since the graph of $F$ is almost horizontal, $\Xi_x$ is almost an isometry.
The improvement lemma
can be stated as follows: if $v\in \widetilde E^s_{f(y)}$ satisfies
$\tfrac{S(f(y),v)}{S(x_1,\Xi_{f(y)}v)}={\rm exp}\left[\pm\xi\right]$ for some $\xi\geq\sqrt{\ve}$,
then $\tfrac{S(y,w)}{S(x_0,\Xi_{y}w)}={\rm exp}\left[\pm(\xi-\tfrac{1}{6}Q(x_0)^{\beta/6})\right]$
for $w=df^{-1}_{f(y)}v$, see \cite[Lemma 4.6]{Ben-Ovadia-2019}. This improvement lemma
implies two properties:
\begin{enumerate}[$\circ$]
\item $\pi[\Sigma^\#]\subset {\rm NUH}_\chi$, see Claim 1 in the proof of \cite[Lemma 4.7]{Ben-Ovadia-2019}.
\item Inverse theorem, see \cite[Section 4]{Ben-Ovadia-2019}.
\end{enumerate}

\subsection{Uniform hyperbolicity versus nonuniform hyperbolicity}

Below we summarize the main differences between the constructions of symbolic models
for uniformly hyperbolic and nonuniformly hyperbolic systems.\\

{\tiny 
\begin{center}
\begin{tabular}{|p{1.8cm}|c|c|}
\hline & & \\
& {\bf UNIFORMLY HYPERBOLIC} & {\bf NON-UNIFORMLY HYPERBOLIC} \\ 
& & \\ \hline \hline
& & \\ 
Coding & All points & ${\rm NUH}_\chi^\#$ \\
& & \\ \hline
& & \\ 
Chart & Lyapunov chart: uniform size & Pesin chart: size $Q$ with $\tfrac{1}{n}\log Q(f^n(x))\to 0$ \\ 
& & \\ \hline
& & \\ 
Vertices & Finite number of Lyapunov charts & Countable number of $\ve$--double charts \\ 
& & \\ \hline
& & \\ 
Edges & $\Psi_x\to \Psi_y$: $f(x)\approx y$ and $f^{-1}(y)\approx x$&
$\Psi_x^{p^s,p^u}\overset{\ve}{\to}\Psi_y^{q^s,q^u}$: (GPO1) and (GPO2)\\
& & \\ \hline
& & \\ 
Rep. function & $F\in C^1$ s.t. $F(0)\approx 0$ and $\|F'\|_{C^0}\approx 0$& $F\in C^{1+\beta/3}$ with (AM1)--(AM3)\\
& & \\ \hline
& & \\ 
$\pi:\Sigma\to M$ & $\{\pi(\un v)\}=V^s[\un v]\cap V^u[\un v]$ & Same \\
& & \\ \hline
& & \\ 
Cover $\mathfs Z$ & $Z(v)=\{\pi(\un v):v_0=v\}$ closed sets & $Z(v)=\{\pi(\un v):\un v\in\Sigma^\#,v_0=v\}$ \\
& & \\ \hline
& & \\ 
Refinement & Relative interiors and closures & Set-theoretical refinement\\
& & \\ \hline
& & \\ 
Partition $\mathfs R$ & Markov with regular rectangles & Markov w/o control of relative interiors \\
& & \\ \hline
& & \\
Graph $(\widehat V,\widehat E)$ & $\widehat V=\mathfs R$, $\widehat E=\{R\to S$: $f(R^*)\cap S^*\neq\emptyset\}$ 
& $\widehat V=\mathfs R$, $\widehat E=\{R\to S$: $f(R)\cap S\neq\emptyset\}$ \\
& & \\ \hline
& & \\
$\widehat\pi:\widehat \Sigma\to M$ & $\{\widehat\pi(\un R)\}=\bigcap_{n\geq 0}f^n(R_{-n})\cap\cdots\cap f^{-n}(R_n)$
& $\{\widehat\pi(\un R)\}=\bigcap_{n\geq 0}\overline{f^n(R_{-n})\cap\cdots\cap f^{-n}(R_n)}$ \\
& & \\ \hline
& & \\
Finite-to-one & $\#\widehat\pi^{-1}(x)\leq (\#\mathfs R)^2$, $\forall x\in M$ & 
$\#[\widehat\pi^{-1}(x)\cap \widehat \Sigma^\#]<\infty$, $\forall x\in M$ \\ 
& & \\ \hline
\end{tabular}
\end{center}
}

\subsection{Surface maps with discontinuities}

Now we explain the constructions in the contexts of Sections \ref{Section-NUH-dis1},
\ref{Section-NUH-dis2}. Let us start with surface maps with discontinuities 
and bounded derivative. As in the previous section, take $I_\ve:=\{e^{-\frac{1}{3}\ve n}:n\geq 0\}$
and truncate $Q(x)$ to $I_\ve$, then consider $\ve$--overlap
and $\ve$--gpo as in Subection \ref{Subsec-NUH-preliminaries}. To prove
Theorem \ref{Thm-coarse-graining}, we need to control more parameters: 
for $\un{\ell}=(\ell_{-1},\ell_0,\ell_1)$ and $\un{k}=(k_{-1},k_0,k_1)$, define
$$
Y_{\un \ell,\un k}:=\left\{\Gamma(x)\in Y:
\begin{array}{cl}
e^{\ell_i}\leq\|C(f^i(x))^{-1}\|<e^{\ell_i+1},&-1\leq i\leq 1\\
e^{-k_i-1}\leq d(f^{i}(x),\mathfs D)<e^{-k_i},&-1\leq i\leq 1\\
\end{array}
\right\}.
$$
Using these precompact sets, proceed as in the proof of Theorem \ref{Thm-coarse-graining}.
Theorem \ref{Thm-inverse} works without modification, as well as Theorem \ref{Sarig}.

\medskip
As mentioned in Section \ref{Section-NUH-dis1}, the prototypical examples for surface maps with discontinuities
and bounded derivative are Poincar\'e return
maps of three dimensional flows with positive speed. Let $N$ be a three dimensional closed
Riemannian manifold, let $X$ be a $C^{1+\beta}$ vector field on $N$ s.t. $X(p)\neq 0$ for all $p\in N$,
and let $\varphi=\{\varphi^t\}_{t\in\R}$ be the flow generated by $X$. 
In Section \ref{Section-NUH-dis1}, we constructed a global Poincar\'e section $M$, equal to the
finite union of transverse discs,
s.t. the return time function $\mathfrak t:M\to (0,\infty)$ is
bounded away from zero and infinity. Then the Poincar\'e return map $f:M\to M$ has discontinuities
and bounded derivative.

\medskip
Now we relate the hyperbolic properties of $\vf$ and $f$.
Let $\widetilde\chi>0$. It is possible to define a nonuniformly hyperbolic locus ${\rm NUH}_{\widetilde\chi}(\vf)$
for $\vf$, similar to the definition in page \pageref{Def-NUH}. This is part of an ongoing
project with Buzzi and Crovisier \cite{BCL}. Let ${\rm NUH}_\chi$ be the nonuniformly hyperbolic locus of $f$.
If $x\in M\cap {\rm NUH}_{\widetilde\chi}(\vf)$, then $x$ satisfies (NUH1)--(NUH3) for  
$\chi:=\widetilde\chi \inf(\mathfrak t)$. Indeed, the flow trajectory of $x$ spends at least time 
$\inf(\mathfrak t)$ between visits to $M$, see \cite[Lemma 2.6]{Lima-Sarig}.
To prove that $x\in{\rm NUH}_\chi$, it remains to check (NUH4). 
Here we encounter a problem: the section $M$ could be chosen in
such a bad way that every trajectory of $f$ converges exponentially fast to $\mathfs D$, and so
(NUH4) never holds. To bypass this difficulty, we need to choose $M$ carefully so that most
$\vf$--trajectories define $f$--trajectories satisfying (NUH4).
Unfortunately, we do not know how to construct such section. What we know, and this is done 
in \cite{Lima-Sarig}, is to construct one section for each fixed measure (more generally,
for a countable set of fixed measures). Fortunately, this is enough for many applications. 

\medskip
Let $\mu$ be a $\widetilde\chi$--hyperbolic probability measure for $\vf$,
and let $\nu$ be its projection to $M$, which is $\chi$--hyperbolic for $f$.
The goal is to choose $M$ so that $\nu$ is $f$--adapted. 
Consider a 1--parameter family of global Poincar\'e sections $\{M_r\}$, by changing the radii 
of each disc of $M$. More specifically, let $M_r=D_r(p_1)\cup\cdots\cup D_r(p_\ell)$, $r\in [a,b]$,
s.t. each $M_r$ is a global Poincar\'e section for $\vf$ and $\mathfrak t_r:M_r\to(0,\infty)$
is uniformly bounded away from zero and infinity. Let $f_r:M_r\to M_r$ be the Poincar\'e return map,
and let $\nu_r$ be the projection of $\mu$ to $M_r$. The next result
is \cite[Thm 2.8]{Lima-Sarig}.

\begin{theorem}\label{Thm-adapted-section}
For Lebesgue-a.e. $r\in [a,b]$, the measure $\nu_r$ is $f_r$--adapted.
\end{theorem}

\begin{proof}
Let $\mathfs D_r$ be the discontinuity set of $f_r$.
It is enough to show that
$$
\nu_r\biggl\{x\in M_r: \liminf\limits_{|n|\to\infty}\frac{1}{|n|}\log d(f^n_r(x),\mathfs D_r)<0\biggr\}=0\textrm{  for  a.e. $r\in[a,b]$.}
$$
For $\alpha>0$, let
$$
A_\alpha(r):=\left\{x\in M_b:\exists \textrm{ infinitely many }n\in\Z\textrm{ s.t. } 
\tfrac{1}{|n|}\log d(f^{n}_b(x),\mathfs D_r)<-\alpha\right\}.
$$
It is enough to prove that $\nu_b[A_\alpha(r)]=0$ for  Lebesgue-a.e. $r\in[a,b]$.
Let $I_\alpha(x):=\{a\leq r\leq b:x\in A_\alpha(r)\}$.
Since $1_{A_\alpha(r)}(x)=1_{I_\alpha(x)}(r)$, by Fubini's Theorem we have
$
\int_a^b \nu_b[A_\alpha(r)]dr=\int_{M_b}\mathrm{Leb}[I_\alpha(x)]d\nu_b(x)$,
so it is enough to prove that 
$$
\mathrm{Leb}[I_\alpha(x)]=0\textrm{ for all }x\in M_b.
$$
The set $I_\alpha(x)$ is contained in the $\limsup$ of intervals $\{I_n\}_{n\in\Z}$
with $|I_n|\approx e^{-\alpha |n|}$. Since $\sum_{n\in\Z}e^{-\alpha |n|}<\infty$, by the Borel-Cantelli lemma
we get that $\mathrm{Leb}[I_\alpha(x)]=0$.
\end{proof}

Combining Theorems \ref{Sarig} and \ref{Thm-adapted-section}, we obtain the following result,
proved in \cite{Lima-Sarig} (see page \pageref{Def-TMF} for the definition of TMF).

\begin{theorem}\label{Lima-Sarig}
Let $\vf:N\to N$ be as above. For each $\chi$--hyperbolic measure $\mu$, there is a
TMF $(\Sigma_r,\sigma_r)$ and $\pi_r:\Sigma_r\to N$ H\"older continuous s.t.:
\begin{enumerate}[$(1)$]
\item $\pi_r\circ \sigma_r^t=\vf^t\circ\pi_r$ for all $t\in\R$.
\item $\pi_r[\Sigma^\#_r]$ has full $\mu$--measure.
\item $\pi_r$ is finite-to-one on $\pi_r[\Sigma_r^\#]$.
\end{enumerate}
\end{theorem}

\medskip
Now consider surface maps with discontinuities and unbounded derivative.
In some sense, the definition of $Q(x)$ given in page \pageref{Def-dis2-Q}
allows to concentrate the difficulty in this single parameter,
but the statements need to be reproved using this new definition. 
Let us see how to get Theorem \ref{Thm-coarse-graining}. Remember from Subsection \ref{Subsec-NUH-dis2-definitions}
that $f$ is well-behaved inside each ball $D_x=B(x,\mathfrak r(x))$. 
For $t>0$, let $M_t=\{x\in M: d(x,\mathfs D)\geq t\}$.
Since $M$ has finite diameter (we are even assuming it is smaller than one), each $M_t$
is precompact.
Fix a countable open cover $\mathfs P=\{D_i\}_{i\in\N_0}$ of $M\backslash\mathfs D$ s.t.:
\begin{enumerate}[$\circ$]
\item $D_i:=D_{z_i}=B(z_i,\mathfrak r(z_i))$ for some $z_i\in M$.
\item For every $t>0$, $\{D\in\mathfs P:D\cap M_t\neq\emptyset\}$ is finite.
\end{enumerate}
For $\un{\ell}=(\ell_{-1},\ell_0,\ell_1)$, $\un{k}=(k_{-1},k_0,k_1)$, 
$\un a=(a_{-1},a_0,a_1)$, define
$$
Y_{\un \ell,\un k,\un a}:=\left\{\Gamma(x)\in Y:
\begin{array}{cl}
e^{\ell_i}\leq\|C(f^i(x))^{-1}\|<e^{\ell_i+1},&-1\leq i\leq 1\\
e^{-k_i-1}\leq d(f^{i}(x),\mathfs D)<e^{-k_i},&-1\leq i\leq 1\\
f^i(x)\in D_{a_i},&-1\leq i\leq 1\\
\end{array}
\right\},
$$
then proceed as in the proof of Theorem \ref{Thm-coarse-graining}.

\medskip
Another feature that requires a better control is bounded distortion inside 
each $V^{s/u}[\un v]$. This is proved in \cite[Proposition 6.2]{Lima-Matheus}.
In summary, under finer analysis, it is possible to prove all that is needed
to obtain Theorem \ref{Sarig}. This is \cite[Theorem 1.3]{Lima-Matheus}, and establishes
problem \#17 in Bowen's notebook \cite{Bowen-notebook}.
The proof actually works under greater generality
that covers not only billiard maps but also some situations where the derivative of $M$
and the behavior of ${\rm exp}_x$ is more complicated,
for instance when $R,\nabla R,\nabla^2 R,\nabla^3R$ grow at most polynomially
fast with respect to the distance to $\mathfs D$, e.g. when $M$ is a moduli space
of curves equipped with the Weil-Petersson metric, see \cite{Burns-Masur-Wilkinson}.

\part{Applications}

There are two canonical applications of Markov partitions:
\begin{enumerate}[$\circ$]
\item Estimating the number of closed orbits.
\item Establishing ergodic properties of equilibrium measures.
\end{enumerate}
Indeed, the Markov partition generates a finite-to-one extension of ${\rm NUH}_\chi^\#$, and
it is possible to lift measures without increasing their entropy.

\medskip
Let $X$ be a set, $\mathfs A$ a sigma-algebra, and $T:X\to X$ a measurable $\Z$ or $\R$--action.
Given a $T$--invariant probability measure $\mu$, let $h_\mu(T)$ denote its
{\em Kolmogorov-Sina{\u\i} entropy}. Given two such systems $(X,\mathfs A,T)$ and $(Y,\mathfs B,S)$,
let $\pi:(X,\mathfs A)\to(Y,\mathfs B)$ be a {\em surjective} measurable extension map,
i.e. $\pi\circ T=S\circ \pi$. Assume that $\pi^{-1}(y)$ is finite for {\em all} $y\in Y$.

\medskip
\noindent
{\sc Projection of measure:} If $\widehat\mu$ is a $T$--invariant probability measure on $(X,\mathfs A)$,
then its projection $\mu=\widehat\mu\circ\pi^{-1}$ is an $S$--invariant probability measure on $(Y,\mathfs B)$
s.t. $h_\mu(S)=h_{\widehat\mu}(T)$.

\medskip
Indeed, by the Abramov-Rokhlin formula, $h_{\widehat\mu}(T)-h_\mu(S)$ is equal to the average entropy on the
fibers. Since each of them is finite, they do no carry entropy. See \cite{einsiedler2011entropy}
for a proof of the Abramov-Rokhlin formula. In general, it is much harder to lift measures 
without increasing the entropy. In our setting, we can do this.

\medskip
\noindent
{\sc Lift of measure:} If $\mu$ is an $S$--invariant probability measure on $(Y,\mathfs B)$,
then
$$
\widehat\mu=\int_Y \tfrac{1}{|\pi^{-1}(y)|}\left(\sum_{x\in\pi^{-1}(y)}\delta_x\right) d\mu(y)
$$
is a $T$--invariant probability measure on $(X,\mathfs A)$ s.t. $h_\mu(S)=h_{\widehat\mu}(T)$.

\medskip
Firstly, one has to check that $\widehat\mu$ is well-defined, see e.g.
\cite[Prop. 13.2]{Sarig-JAMS}. The preservation of entropy follows again from the Abramov-Rokhlin formula.
In particular, if the variational principle holds in $(X,\mathfs A,T)$ and $(Y,\mathfs B,S)$, then
their topological entropies coincide. 

\medskip
The above conclusions also hold if there are sets
$X^\#\subset X$ and $Y^\#\subset Y$ s.t.:
\begin{enumerate}[$\circ$]
\item The restriction $\pi\restriction_{X^\#}:X^\#\to Y^\#$ is finite-to-one.
\item Every $T$--invariant probability measure is supported
on $X^\#$, and every $S$--invariant probability measure is supported on $Y^\#$.
\end{enumerate}
It is in this way that invariant measures for nouniformly hyperbolic systems relate with invariant measures
for symbolic spaces. In the next two section we explain some of the applications.

\section{Estimates on the number of closed orbits}\label{Section-periodic}

Given two sequences $\{a_n\}_{n\geq 1},\{b_n\}_{n\geq 1}$,
let us write $a_n\asymp b_n$ if there are constants $C,n_0>1$
s.t. $C^{-1}\leq \tfrac{a_n}{b_n}\leq C$ for all $n\geq n_0$, and 
$a_n\sim b_n$ if $\lim_{n\to\infty}\tfrac{a_n}{b_n}=1$.
Assume that $f:M\to M$ is a transitive uniformly hyperbolic diffeomorphism,
and let $(\Sigma,\sigma,\pi)$ be a symbolic model.
Given $n\geq 1$, let ${\rm Per}_n(f),{\rm Per}_n(\sigma)$ denote the number of
periodic orbits of period $n$ for $f,\sigma$ respectively,
and write $h=h_{\rm top}(f)=h_{\rm top}(\sigma)$.
Recall that $\Sigma$ has finitely many states.
Also $(\Sigma,\sigma)$ is transitive, and if $f$ is topologically mixing then
$(\Sigma,\sigma)$ is as well \cite[Proposition 30]{Bowen-MP-Axiom-A}. 
Since $\pi$ is finite-to-one (indeed, bounded-to-one), ${\rm Per}_n(f)\asymp{\rm Per}_n(\sigma)$.
If $p\geq 1$ is the period of the TMS, then
${\rm Per}_{pn}(\sigma)\sim p e^{pnh}$, see e.g. \cite[Observation 1.4.3]{Kitchens-Book}.
Hence ${\rm Per}_{pn}(f)\asymp e^{pnh}$. If $f$ is topologically mixing,
then  ${\rm Per}_{n}(f)\asymp e^{nh}$.

\medskip
Now let $f:M\to M$ be nonuniformly hyperbolic, where $M$ can 
have any dimension. Since $\pi$ is finite-to-one, there is a constant $C>0$ s.t.
${\rm Per}_{n}(f)\geq C\times {\rm Per}_{n}(\sigma)$ for all $n\geq 1$. 
The reverse inequality might not hold, because $\pi$ only codes
trajectories inside ${\rm NUH}_\chi$,
and $f$ can have many more (even uncountably many) periodic orbits outside of ${\rm NUH}_\chi$. 
Even if we only count isolated periodic orbits, the growth rate of ${\rm Per}_n(f)$
can be superexponential \cite{Kaloshin-Super-Exponential}.

\medskip
For a TMS with countably many states, good estimates
on ${\rm Per}_{n}(\sigma)$ are related to the existence of measures of maximal entropy, as observed by
Gurevi{\v{c}} \cite{Gurevich-Topological-Entropy,Gurevich-Measures-Of-Maximal-Entropy}.
He showed that every transitive TMS admits at most one measure of maximal entropy, and such measure
exists iff there is $p\geq 1$ s.t. for every vertex $v$ it holds
$$
\#\{\un v\in\Sigma: \sigma^{pn}(\un v)=\un v,v_0=v\}\asymp e^{pnh_{\rm max}},
$$
where $h_{\rm max}=h_{\rm max}(\sigma)=\sup\{h_\nu(\sigma):\nu\text{ is $\sigma$--invariant probability measure on }\Sigma\}$.
We add the assumption that $f$ has a $\chi$--hyperbolic measure of maximal entropy, in which case
$h=h_{\rm top}(f)=h_{\rm max}(\sigma)$. Measures of maximal entropy exist for $C^\infty$ 
diffeomorphisms \cite{Newhouse-Entropy}, although they might not be hyperbolic (e.g.
the product of an Anosov diffeomorphism and an irrational rotation).
For surface diffeomorphism with positive topological entropy, ergodic measures of maximal entropy are
$\chi$--hyperbolic for every $\chi<h_{\rm top}(f)$, by the Ruelle inequality.

\medskip
Assuming that $f$ possesses a $\chi$--hyperbolic measure of maximal entropy,
repeat the arguments used for uniformly hyperbolic systems inside a transitive component
of $(\Sigma,\sigma)$. If $p\geq 1$ is the period of this component, then
${\rm Per}_{pn}(\sigma)\asymp e^{pnh}$, and so
${\rm Per}_{pn}(f)\geq C\times e^{pnh}$ for all $n\geq n_0$. This was proved
by Sarig in dimension two \cite[Theorem 1.1]{Sarig-JAMS} and by Ben Ovadia
in higher dimension \cite[Theorem 1.4]{Ben-Ovadia-2019}.

\medskip
Let us make some comments on the period $p$. Assume that $f$ is topologically mixing.
Recently, Buzzi was able to use the finite-to-one coding to define a one-to-one coding. Together
with a precise counting on the TMS, he concluded that
${\rm Per}_{n}(f)\geq C\times e^{nh}$ for all $n\geq n_0$ \cite{Buzzi-2019}.
His proof works whenever the smooth system has a finite-to-one coding with some
extra properties, which are true for diffeomorphisms \cite{Sarig-JAMS,Ben-Ovadia-2019},
and billiard maps \cite{Lima-Matheus}.

\medskip
Now let $\vf:M\to M$ be a flow and $(\Sigma_r,\sigma_r)$ be a TMF. Given $T>0$,
let ${\rm Per}_T(\vf)$ and ${\rm Per}_T(\sigma_r)$ denote
the number of closed trajectories of $\vf$ and $\sigma_r$, respectively, with minimal period $\leq T$.
Note that $(\un v,t)\in{\rm Per}_T(\sigma_r)$ iff there is a minimal period $n\geq 1$ s.t.
$\un v\in{\rm Per}_n(\sigma)$ and $r_n(\un v)\leq T$, hence estimating ${\rm Per}_T(\sigma_r)$ is more
complicated. For uniformly hyperbolic flows, there are precise estimates:
\begin{enumerate}[$\circ$]
\item Geodesic flows on closed hyperbolic surfaces: 
in constant curvature, Huber proved that ${\rm Per}_T(\vf)\sim e^{T}/T$ \cite{Huber-Closed-Geodesics}.
In variable curvature, Sina{\u \i} gave the first estimates \cite{Sinai-Closed-Geodesics}, which were
later significantly sharpened by Margulis \cite{Margulis-Closed-Orbits}, who proved that
${\rm Per}_T(\vf)\sim Ce^{Th}/T$ where $C=1/h$ (C. Toll, unpublished).
\item Axiom A flows: Bowen proved ${\rm Per}_T(\vf)\asymp e^{Th}/T$ \cite{Bowen-Closed-Geodesics}.
If the flow is topologically weak mixing, Parry and Pollicott proved that
${\rm Per}_T(\vf)\sim e^{Th}/Th$ \cite{Parry-Pollicott-PNT}, and Pollicott and Sharp found 
an estimate for the error term \cite{Pollicott-Sharp-Error-Term}.
\end{enumerate}
We also mention a result for manifolds with Gromov hyperbolic fundamental group (e.g.
manifolds that admit a metric with Anosov geodesic flow).
Knieper and Coornaert counted free homotopy classes of closed geodesics estimating
the growth rate of conjugacy classes in the fundamental group \cite{Knieper-Archiv,Knieper-Coornaert}.

\medskip
Now consider nonuniformly hyperbolic flows. For geodesic flows in nonpositively curved rank one manifolds,
the following are known:
\begin{enumerate}[$\circ$]
\item Knieper showed that $\pi_0(T)\asymp e^{Th}/T$,
where $\pi_0(T)$ counts the homotopy classes of simple closed geodesics with
length less than $T$ \cite{Knieper-Closed-Orbits,Knieper-Handbook-Chapter}.
\item For certain metrics constructed by Donnay \cite{Donnay} and Burns and Gerber \cite{Burns-Gerber},
${\rm Per}_T(\vf)\sim e^{Th}/Th$ \cite{Weaver}.
\end{enumerate}
For the flows in Theorem \ref{Lima-Sarig},
if there exists a measure of maximal entropy then there is $T_0>0$ s.t.
${\rm Per}_T(\vf)\geq C\times e^{Th}/T$ for all $T\geq T_0$ \cite[Thm 8.1]{Lima-Sarig}.
This estimate strengthens Katok's bound $\liminf_{T\to\infty}\frac{1}{T}\log{\rm Per}_T(\vf)\geq h$,
see \cite{Katok-IHES,Katok-Closed-Geodesics}. 
The proof in \cite{Lima-Sarig} uses a dichotomy for TMF, see \cite[Theorem 4.6]{Ledrappier-Lima-Sarig}.

\medskip
We end this section mentioning some results for two-dimensional billiard maps. As seen in
Subsection \ref{Subsec-NUH-dis2-definitions}, every billiard map preserves
an invariant Liouville measure $\mu_{\rm SRB}$.
Using the countable Markov partition constructed in \cite{Bunimovich-Chernov-Sinai},
Chernov proved that $\liminf \tfrac{1}{n}\log{\rm Per}_n(f)\geq h_{\mu_{\rm SRB}}(f)$ \cite[Corollary 2.4]{Chernov-91}. 
Better estimates can be obtained using measures of maximal entropy.
Recently, Baladi and Demers gave sufficient conditions for periodic Lorentz gases (Sina{\u\i} billiards
with non-intersecting scatterers) to have measures of maximal entropy \cite{Baladi-Demers}.
This occurs when the billiard map satisfies two properties:
\begin{enumerate}[(1)]
\item Finite horizon: there is no trajectory that makes only tangential collisions.
\item $h_*>s_0\log 2$. 
\end{enumerate}
The second assumption requires some explanation.
Part of their work consists on defining a topological entropy $h_*$ for finite horizon Lorentz gases, which
is an upper bound for all metric entropies \cite[Thm 2.3(4)]{Baladi-Demers}.
Fixing an angle $\theta_0\approx \tfrac{\pi}{2}$ and $n_0>0$, let $s_0\in (0,1)$ be the smallest
number s.t. any orbit of length $n_0$ has at most $s_0n_0$ collisions with $|\theta|>\theta_0$. 
Under conditions (1)--(2), there is an $f$--adapted
measure $\mu_*$ s.t. $h_{\mu_ *}(f)=h_ *$ \cite[Thm 2.4]{Baladi-Demers}.
Using \cite{Lima-Matheus} and \cite{Buzzi-2019}, it follows that ${\rm Per}_n(f)\geq C\times e^{nh_*}$
for $n$ sufficiently large.

\medskip
Here are some examples of billiards satisfying the conditions (1)--(2) of
\cite{Baladi-Demers}.
If $\mathcal K_{\min}$ be the minimum curvature of the scatterer boundaries and $\tau_{\min}$ be
the minimum free flight time, then $h_*>\log(1+2\mathcal K_{\min}\tau_{\min})$.
Consider the two-parameter family $(r,R)$ of Lorentz gases in $\T^2$ with two discs as scatterers,
one centered at the origin $(0,0)$ with radius $R$ and the other at $(\tfrac{1}{2},\tfrac{1}{2})$
with radius $r$, see figure \ref{figure-lorentz}.
\begin{figure}[hbt!]
\centering
\def\svgwidth{7cm}
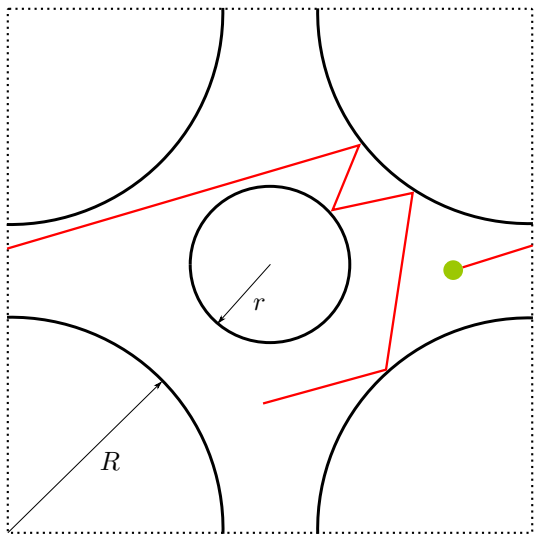
\caption{If $r,R$ are chosen inside a specific polygon in the parameter space then 
there is an $f$--adapted measure of maximal entropy, and
${\rm Per}_n(f)\geq C\times e^{nh_*}$ for all large $n$.}\label{figure-lorentz}
\end{figure}
Baladi and Demers found a domain in the parameter space for which
$\log(1+2\mathcal K_{\min}\tau_{\min})\geq \tfrac{1}{2}\log 2\geq s_0\log 2$,
hence \cite[Thm 2.4]{Baladi-Demers} applies.
There are also numerical experiments dealing with scatterers located in a triangular lattice
indicating that $h_*>s_0\log 2$ whenever the scatterers do not intersect and the billiard
has finite horizon \cite{Gaspard-Baras}.

\section{Equilibrium measures}

Let $(Y,S)$, where $Y$ is a complete metric separable space and $S:Y\to Y$ is
continuous, and let $\psi:Y\to\R$ be a continuous potential.
The following definitions are standard.

\medskip
\noindent
{\sc Topological pressure:} The {\em topological pressure} of $\psi$ is
$P_{\rm top}(\psi):=\sup\{h_\mu(S)+\int\psi d\mu\}$,
where the supremum ranges over all $S$--invariant probability measures
for which $\int\psi d\mu$ makes sense and $h_\mu(S)+\int\psi d\mu\neq \infty-\infty$.

\medskip
\noindent
{\sc Equilibrium measure:} An {\em equilibrium measure} for $\psi$ is an $S$--invariant probability 
measure $\mu$ s.t. $P_{\rm top}(\psi)=h_\mu(S)+\int\psi d\mu$.


\medskip
A special case occurs when $\psi\equiv 0$: equilibrium measures are measures of maximal entropy.
If $\pi:(X,T)\to (Y,S)$ is finite-to-one, then equilibrium measures for
$\psi$ lift to equilibrium measures for $\widehat\psi=\psi\circ\pi$.
If $\pi$ is H\"older continuous, then $\widehat\psi$ is H\"older
whenever $\psi$ is. In our context, we can apply the thermodynamical formalism for H\"older
continuous potentials in TMS to obtain ergodic properties of equilibrium measures of
H\"older continuous potential in uniformly and nonuniformly hyperbolic systems.

\medskip
Since a transitive TMS with finitely many states has a unique measure of maximal entropy
\cite{Parry-Intrinsic}, then every uniformly hyperbolic
transitive diffeomorphism has a unique measure of maximal entropy \cite{Bowen-MP-Axiom-A},
equal to the projection of the measure of maximal entropy in $(\Sigma,\sigma)$.
Prior to this, Gurevi{\u{c}} obtained some partial results, using the work
of Sina{\u\i}, and of Berg\footnote{Berg proved that for hyperbolic toral automorphisms the Haar
measure is the only measure of maximal entropy \cite{Berg-convolutions}.}.
Bowen also showed that every H\"older continuous potential
has a unique equilibrium measure \cite{Bowen-unique-equilibrium}, and it is
either Bernoulli or Bernoulli times a period \cite{Bowen-Bernoulli}.

\medskip
Using the same analogy, Bowen and Ruelle proved that H\"older continuous potentials on uniformly
hyperbolic flows have unique equilibrium measures \cite{Bowen-Ruelle-SRB}. 
In this case, equilibrium measures of $(\Sigma_r,\sigma_r)$ are related to equilibrium measures of
$(\Sigma,\sigma)$, see \cite[Prop. 3.1]{Bowen-Ruelle-SRB}.

\medskip
For nonuniformly hyperbolic $C^{1+\beta}$ surface diffeomorphisms,
Sarig proved that each H\"older continuous potential has at most countably many ergodic
hyperbolic equilibrium measures \cite[Thm 1.2]{Sarig-JAMS}, and each of them is either Bernoulli
or Bernoulli times a period \cite{Sarig-Bernoulli-JMD}. The proof uses that for topologically
transitive TMS each H\"older continuous potential has at most one equilibrium measure \cite{Buzzi-Sarig}, and 
different topologically transitive subgraphs of a TMS have disjoint vertex sets.
The same holds for higher dimensional diffeomorphisms \cite{Ben-Ovadia-2019},
and for three dimensional flows \cite{Lima-Sarig}.  In the flow case, each such equilibrium measure
is either Bernoulli or Bernoulli times a rotation \cite{Ledrappier-Lima-Sarig}.
Since geodesic flows cannot have rotational components (they are a particular case of Reeb flows),
the following corollary holds:
if $S$ is a closed smooth orientable Riemannian surface
with nonpositive and non-identically zero curvature, then the geodesic flow
of $S$ is Bernoulli with respect to its (unique) measure of maximal entropy, see
\cite[Corollary 1.3]{Ledrappier-Lima-Sarig}.

\medskip
Let us mention some results on the uniqueness of measures of maximal entropy for
nonuniformly hyperbolic geodesic flows. The uniqueness referred in the previous paragraph
follows from the work of Knieper, who proved it for geodesic flows on
closed rank one manifolds \cite{Knieper-Rank-One-Entropy}, and also
for geodesic flows on symmetric spaces of higher rank \cite{Knieper-Higher-rank}.
Gelfert and Ruggiero proved the uniqueness for geodesic flows on surfaces without focal
points and genus greater than one \cite{Gelfert-Ruggiero}.
Burns, Climenhaga, Fisher, and Thompson proved the uniqueness of many equilibrium states
(including some multiples of the geometric potential and the zero potential) of geodesic flows on rank one manifolds
\cite{Burns-et-al}, and there is a recent preprint that obtains similar results for geodesic flows on
surfaces without focal points \cite{Chen-Kao-Park}. There is also a recent preprint
that proves the uniqueness of the measure of maximal entropy for geodesic flows on surfaces without conjugate
points \cite{Climenhaga-Knieper-War}.

\medskip
Uniqueness of measures of maximal entropy for $C^\infty$ transitive
surface diffeomorphisms with positive topological entropy has recently been obtained. Essentially, the results of \cite{Sarig-JAMS}
were not able to give uniqueness because it was not clear how to overrule the possibility that two measures
in the surface lift to two different transitive components of the TMS.
This difficulty was solved by Buzzi, Crovisier, and Sarig \cite{Buzzi-Crovisier-Sarig}, who showed
that if two measures have positive entropy and are homoclinically related (a notion introduced in \cite{RH-RH-T-U})
then they can be lifted to a same transitive component of the TMS. This can be regarded as a
version of \cite[Proposition 30]{Bowen-MP-Axiom-A}. They also prove that if the diffeomorphism is $C^\infty$,
then all measures of maximal entropy are homoclinically related (this uses Yomdin's theory), 
hence there is a unique measure of maximal entropy. It would be interesting to obtain similar
results for three dimensional flows and billiard maps.

\bibliographystyle{alpha}
\bibliography{bibliography}{}

\end{document}